\pdfoutput=1
\documentclass[a4paper,12pt,reqno,oneside]{amsart}

\usepackage{geometry}
\usepackage{lmodern}
\usepackage[stretch=10]{microtype}
\usepackage{graphicx}
\usepackage{tikz-cd}

\usepackage{amsmath}
\usepackage{amssymb}
\usepackage{amsthm}
\usepackage[matrix,arrow,cmtip]{xy}
\usepackage{bbm}
\usepackage[shortlabels]{enumitem}
\PassOptionsToPackage{hyphens}{url}
\usepackage[bookmarksnumbered,bookmarksdepth=2,bookmarksopen,colorlinks,linkcolor=black,citecolor=black,urlcolor=black,linktoc=all]{hyperref}
\usepackage[capitalise]{cleveref}

\newcommand{\Qbar}{\overline \bQ}

\newcommand{\bC}{\mathbb{C}}

\newcommand{\bN}{\mathbb{N}}

\newcommand{\bQ}{\mathbb{Q}}

\newcommand{\bZ}{\mathbb{Z}}
\newcommand{\ZZ}{\mathbb{Z}}
\newcommand{\QQ}{\mathbb{Q}}
\newcommand{\RR}{\mathbb{R}}
\newcommand{\CC}{\mathbb{C}}

\newcommand{\GG}{\mathbb{G}}
\newcommand{\PP}{\mathbb{P}}
\newcommand{\AAA}{\mathbb{A}}

\newcommand{\cA}{\mathcal{A}}
\newcommand{\cB}{\mathcal{B}}
\newcommand{\cC}{\mathcal{C}}
\newcommand{\cD}{\mathcal{D}}

\newcommand{\cG}{\mathcal{G}}
\newcommand{\cH}{\mathcal{H}}
\newcommand{\cI}{\mathcal{I}}

\newcommand{\cM}{\mathcal{M}}

\newcommand{\cO}{\mathcal{O}}
\newcommand{\cP}{\mathcal{P}}

\newcommand{\fp}{\mathfrak{p}}

\newcommand{\fC}{\mathfrak{C}}
\newcommand{\fG}{\mathfrak{G}}

\newcommand{\fI}{\mathfrak{I}}
\newcommand{\fR}{\mathfrak{R}}

\newcommand{\fX}{\mathfrak{X}}

\newcommand{\Ag}{\mathcal{A}_g}

\newcommand{\gGSp}{\mathbf{GSp}}
\newcommand{\gPGSp}{\mathbf{PGSp}}
\newcommand{\gPSp}{\mathbf{PSp}}

\newcommand{\gK}{\mathbf{K}}

\newcommand{\gM}{\mathbf{M}}

\newcommand{\gSp}{\mathbf{Sp}}

\newcommand{\rM}{\mathrm{M}}

\DeclareMathOperator{\Aut}{Aut}

\DeclareMathOperator{\codim}{codim}

\DeclareMathOperator{\disc}{disc}
\DeclareMathOperator{\End}{End}
\DeclareMathOperator{\Frac}{Frac}
\DeclareMathOperator{\Gal}{Gal}
\DeclareMathOperator{\GL}{GL}

\DeclareMathOperator{\Lie}{Lie}
\DeclareMathOperator{\Mor}{Mor}

\DeclareMathOperator{\rk}{rk}

\DeclareMathOperator{\Sp}{Sp}
\DeclareMathOperator{\Spf}{Spf}
\DeclareMathOperator{\Spec}{Spec}

\newcommand{\an}{\mathrm{an}}

\newcommand{\adj}{\mathrm{adj}}

\newcommand{\can}{\mathrm{can}}

\newcommand{\der}{\mathrm{der}}
\newcommand{\formal}{\mathrm{for}}
\newcommand{\id}{\mathrm{id}}
\newcommand{\inv}{\mathrm{inv}}

\newcommand{\red}{\mathrm{red}}
\newcommand{\rig}{\mathrm{rig}}

\newcommand{\van}{{v\textrm{-}\mathrm{an}}}
\newcommand{\vhatan}{{\hat v\textrm{-}\mathrm{an}}}
\newcommand{\vrig}{{v\textrm{-}\mathrm{rig}}}

\newcommand{\CCan}{{\CC\textrm{-}\mathrm{an}}}

\newcommand{\abs}[1]{\lvert #1 \rvert}

\newcommand{\powerseries}[2]{#1 [\![ #2 ]\!]}
\newcommand{\tatealgebra}[2]{#1 \langle #2 \rangle}

\newcommand{\ov}{\overline}

\newcommand{\fullmatrix}[4]{\left( \begin{matrix} #1 & #2 \\ #3 & #4 \end{matrix} \right)}
\newcommand{\fullsmallmatrix}[4]{\bigl( \begin{smallmatrix} #1 & #2 \\ #3 & #4 \end{smallmatrix} \bigr)}

\newcommand{\defterm}[1]{\textbf{#1}}

\newtheorem{lemma}{Lemma}[section]
\newtheorem{proposition}[lemma]{Proposition}
\newtheorem{theorem}[lemma]{Theorem}
\newtheorem{corollary}[lemma]{Corollary}
\newtheorem{conjecture}[lemma]{Conjecture}

\Crefname{conjecture}{Conjecture}{Conjectures} 

\Crefname{claim}{Claim}{Claims}

\newtheorem*{lemma*}{Lemma}
\newtheorem*{proposition*}{Proposition}
\newtheorem*{theorem*}{Theorem}
\newtheorem*{corollary*}{Corollary}
\newtheorem*{claim*}{Claim}

\theoremstyle{definition}
\newtheorem*{definition}{Definition}
\newtheorem{remark}[lemma]{Remark}

\usepackage{ifthen}
\newcounter{constant}

\newcommand{\newC}[1]{%
  \ifthenelse{\equal{#1}{*}} {%
      \stepcounter{constant} c_{\theconstant}%
  } {%
      \refstepcounter{constant} c_{\theconstant} \label{C:#1}%
  }%
}
\newcommand{\refC}[1]{c_{\ref*{C:#1}}}

\usepackage{color}

\title[The large Galois orbits conjecture]{The large Galois orbits conjecture under multiplicative degeneration}

\author{Christopher Daw}
\author{Martin Orr}

\begin{document}

\begin{abstract}
    We establish the PEL type large Galois orbits conjecture for Hodge generic curves in $\cA_g$ possessing multiplicative degeneration. Combined with our earlier works, this concludes the proof of the Zilber--Pink conjecture in $\cA_2$ for such curves. We also deduce several new cases of Zilber--Pink in $\cA_g$ for $g\geq 3$. Our proof uses Andr\'e's G-functions method, using formal and rigid uniformisation of semiabelian schemes to interpret the $p$-adic evaluations of the period G-functions.
\end{abstract}

\maketitle

\setcounter{tocdepth}{1}
\tableofcontents

\section{Introduction}\label{intro}

The large Galois orbits conjecture (LGO) is commonly recognised as the hardest ingredient in the Pila--Zannier strategy for solving questions of unlikely intersections, such as the Zilber--Pink conjecture.
It predicts that points of unlikely intersection between special subvarieties and an arbitrary subvariety of a Shimura variety should have Galois orbits whose size is bounded below by a polynomial in the ``complexity'' of a special subvariety containing the point.  For a more detailed discussion of LGO, see \cref{sec:LGO}.

In this article, we will only consider the Shimura variety~$\Ag$, the moduli space of principally polarised abelian varieties, and its special subvarieties of PEL type, that is, those which parameterise abelian varieties with specified endomorphism rings.
The complexity of a PEL type special subvariety is the absolute value of the discriminant of its generic endomorphism ring.  We refer to LGO in this setting as \textbf{PEL type LGO} (see \cref{lgo-pel}).

Our main result is PEL type LGO for Hodge generic curves in $\cA_g$ satisfying a multiplicative degeneracy condition.
In order to state it precisely, we need to introduce some terminology.
We say that a geometrically irreducible algebraic curve $C \subset \Ag$ is \defterm{Hodge generic} if it is not contained in any special subvariety of~$\Ag$ of positive codimension.

For a point $s\in\cA_g(\ov\QQ)$, we denote by $A_s$ the corresponding abelian variety (defined up to isomorphism over~$\ov\QQ$) and by $\End(A_s) \otimes \QQ$ the endomorphism algebra $\End(A_s)\otimes_\ZZ \QQ$ of $A_s$. For a point $s \in C$, we say that the abelian variety $A_s$ has \defterm{unlikely endomorphisms} if, either
\begin{enumerate}
\item $g \geq 3$ and $\End(A_s) \otimes \QQ \not\cong \QQ$, or
\item $g=2$ and $\End(A_s) \otimes \QQ \not\cong \QQ, \QQ \times \QQ$, or a real quadratic field.
\end{enumerate}
This terminology is derived from the fact that these conditions are equivalent to the condition that $s$ lies on a PEL type special subvariety of $\Ag$ of codimension at least~$2$ (see \cite[Table~1]{ExCM} for $g=2$ and \cite[Prop.~1.7]{PELtype} for $g \geq 3$). In other words, the condition characterises exactly the unlikely intersections between $C$ and the PEL type special subvarieties.

Our main theorem is as follows.

\begin{theorem} \label{main-bound}
Let $g \geq 2$.
Let $C$ be a geometrically irreducible Hodge generic algebraic curve in $\Ag$ defined over $\Qbar$ such that the Zariski closure of~$C$ in the Baily--Borel compactification of $\Ag$ intersects the zero-dimensional stratum of the boundary.
Then there exist constants $\newC{main-bound-mult}, \newC{main-bound-exp} > 0$ such that, for every $s \in C(\Qbar)$ for which $A_s$ has unlikely endomorphisms, we have
\[ [\QQ(s):\QQ] \geq \refC{main-bound-mult} \abs{\disc(\End(A_s))}^{\refC{main-bound-exp}}. \]
\end{theorem}

We recall that the Baily--Borel compactification of $\cA_g$ is naturally stratified as a disjoint union
\[\cA_g\sqcup\cA_{g-1}\sqcup\cdots\sqcup\cA_1\sqcup\cA_0\]
of locally closed subvarieties. The zero-dimensional stratum is the point $\cA_0$, and the condition that $C$ intersects it is equivalent to saying
that the associated family of principally polarised abelian varieties degenerates to
a torus (see \cite[Prop.~9.4]{ExCM} for a formal statement.) This is the condition to which we refer by the term \defterm{multiplicative degeneration}.

\subsection{Implications for Zilber--Pink}

Our interest in LGO is motivated by the Zilber--Pink conjecture.
Let $S$ denote a Shimura variety and let $\Sigma(d)$ denote the set of its special subvarieties of dimension at most $d$. The Zilber--Pink conjecture states that the intersection of $\Sigma(d)$ with an irreducible algebraic subvariety $V \subset S$ of dimension less than $\dim(S)-d$, not contained in a proper special subvariety, is not Zariski dense in~$V$. At the time of writing, it is a major unsolved problem in Diophantine geometry. 

Pila and Zannier devised a strategy to solve questions of unlikely intersections using the theory of o-minimal structures from model theory, as well as Diophantine bounds. This strategy first appeared in a new proof \cite{PZ08} of the Manin--Mumford conjecture for abelian varieties. For the Zilber--Pink conjecture for Shimura varieties, the strategy is described by Ren and the first-named author in \cite{DR18}, which builds on the work \cite{HP16} of Habegger and Pila on products of modular curves.

An immediate consequence of our main theorem is that, using results obtained in our earlier papers \cite{ExCM} and \cite{QRTUI}, we can complete the proof of Zilber-Pink in $\cA_2$ for curves with multiplicative degeneration. The new part of this result is that the intersection of $C$ with the union of the so-called $E^2$ curves is finite.  

\begin{corollary}\label{cor:A2}
Let $C$ be a geometrically irreducible Hodge generic algebraic curve in $\cA_2$ defined over $\Qbar$ such that the Zariski closure of~$C$ in the Baily--Borel compactification of $\cA_2$ intersects the zero-dimensional stratum of the boundary.
Then the intersection of $C$ with the union of all special curves of $\cA_2$ is finite.
\end{corollary}

\begin{proof}
This follows from \cite[Theorem 1.1]{ExCM} and \cite[Theorem 1.3]{QRTUI}, together with \cref{main-bound}.
\end{proof}

Similarly, building on our work \cite{PELtype}, we obtain the following cases of Zilber--Pink in $\Ag$ for $g \geq 3$, again, all under multiplicative degeneration.

\begin{corollary}\label{cor-ZP>2}
Let $g \geq 3$.
Let $C$ be a geometrically irreducible Hodge generic algebraic curve in $\Ag$, defined over $\Qbar$, such that the Zariski closure of~$C$ in the Baily--Borel compactification of $\Ag$ intersects the zero-dimensional stratum of the boundary.
Then there are only finitely many points in $C(\Qbar)$ of the following types:
\begin{enumerate}[(i)]
\item $\End(A_s) \otimes \QQ$ is a totally real field other than $\QQ$ (that is, a division algebra of type~I in the Albert classification);
\item $\End(A_s) \otimes \QQ$ is a non-split totally indefinite quaternion algebra over a totally real field (that is, a division algebra of type~II in the Albert classification).
\end{enumerate}
\end{corollary}

\begin{proof}
This follows from \cite[Theorem 1.3]{PELtype}, together with \cref{main-bound}.
\end{proof}

The reason that \cref{cor-ZP>2} does not cover all points with unlikely endomorphisms is that the Pila--Zannier strategy requires two arithmetic inputs: LGO, which is the problem addressed in this article, and so-called {\it parameter height bounds}, which were the main focus of \cite{QRTUI} and \cite{PELtype}. In the setting of \cref{cor-ZP>2}, one can derive finiteness whenever such parameter height bounds are available for the relevant special subvarieties of PEL type.

At the time of writing, parameter height bounds have been proved when $\End(A_s)$ is a division algebra of type I or~II in the Albert classification \cite{PELtype}, leading to \cref{cor-ZP>2}.  Parameter height bounds when $\End(A_s)$ is a division algebra of type III or~IV are work in progress by Bhatta.  That would complete the generalisation of \cref{cor-ZP>2} to all points where $A_s$ is simple and has unlikely endomorphisms.  It is likely that parameter height bounds for non-simple $A_s$ could be proved using the existing techniques, but there may be some technical complications.

\subsection{Large Galois orbits conjecture}\label{sec:LGO}

We formulate the aforementioned PEL type LGO conjecture as follows.

\begin{conjecture}\label{lgo-pel}
Let $d \in \ZZ_{\geq 0}$.
Let $V$ be a Hodge generic irreducible subvariety of $\Ag$, of dimension less than $\dim(\Ag)-d$, defined over~$\ov\QQ$.
Then there exist constants $\newC{lgo-pel-multiplier}, \newC{lgo-pel-exponent} > 0$, depending only on $V$, such that, for every point $s \in V$, if $\{s\}$ is a component of the intersection of~$V$ with a special subvariety of PEL type $Z \subset \Ag$ with $\dim(Z) \leq d$, we have
\[ \# \Gal(\ov\QQ/\QQ) \cdot s \geq \refC{lgo-pel-multiplier} \abs{\disc(\End(A_s))}^{\refC{lgo-pel-exponent}}. \]
\end{conjecture}

In order to generalise this conjecture to an arbitrary Shimura variety~$S$, it is necessary to define a suitable notion of complexity $\Delta(Z)$ for special subvarieties $Z \subset S$.
In the context of \cref{lgo-pel}, we take $\Delta(Z)$ to be the discriminant of the endomorphism ring of an abelian variety parameterised by a very general point of~$Z$.
Given a suitable notion of complexity, one could state a version of \cref{lgo-pel} with $\Ag$ replaced by the Shimura variety~$S$ and with $\abs{\disc(\End(A_s))}$ replaced by the complexity of the smallest special subvariety containing~$s$.
See \cite[Conj.~11.1]{DR18} for a precise formulation of such a conjecture for general Shimura varieties.
A version of this conjecture for $Y(1)^n$ was first formulated at \cite[Conj.~8.2]{HP16}.

\begin{remark} \label{unlikely-condition-superfluous}
Surprisingly, \cref{lgo-pel} for all $g$ implies \cref{lgo-pel} with the condition $\dim(Z) \leq d$ removed (that is, the condition that the intersections $V \cap Z$ are unlikely).
This may be proved by embedding $\Ag$ as a special subvariety of a sufficiently large $\cA_{g'}$, via a closed immersion $\iota \colon \Ag \to \cA_{g'}$, and choosing a Hodge generic subvariety $W \subset \cA_{g'}$ which contains $\iota(V)$ and satisfies $\dim(W)=\dim(V)+1$.
Then, away from an exceptional subvariety of~$V$, which may be dealt with by induction on~$\dim(V)$, every component of an intersection $V \cap Z$ (where $Z$ is a special subvariety of~$\Ag$ of any dimension) becomes a component of the intersection $W \cap \iota(Z)$, and the latter is unlikely in~$\cA_{g'}$ when $g'$ is large enough.

This argument does not permit us to relax the ``unlikely'' condition in \cref{main-bound}, because applying the argument to a curve~$V \subset \Ag$ would require LGO for a surface $W \subset \cA_{g'}$.
Indeed, we have been unable to prove the analogue of \cref{main-bound} for ``just likely intersections,'' namely, those points $s \in \cA_2$ with $\End(A_s) \otimes \QQ \cong \QQ \times \QQ$ or a real quadratic field (see \cref{just-likely-relations}). 
\end{remark}

The first large Galois orbits conjecture, for CM points in $\Ag$, was stated by Edixhoven \cite[Problem~14]{EM:problems}.
This was proved by Tsimerman \cite[Thm.~5.2]{Tsi18}, using the averaged Colmez conjecture due, independently, to Andreatta, Goren, Howard and Madapusi Pera \cite{AGHMP17} and Yuan and Zhang \cite{YZ18}.
More recently, the analogous conjecture for special points in all Shimura varieties was proved by Pila, Shankar, Tsimerman, Esnault and Groechenig \cite{PST+}, building on work of Binyamini, Schmidt, and Yafaev \cite{BSY}.
This finished the proof of the André--Oort conjecture for all Shimura varieties.

At first sight, Edixhoven's conjecture for CM points is stronger than the $d=0$ case of \cref{lgo-pel}, because it proposes a uniform bound for all CM points in~$\Ag$, while the bound in \cref{lgo-pel} for $d=0$ is only uniform across CM points contained in a codimension-$1$ subvariety $V \subset \Ag$.
However, \cref{unlikely-condition-superfluous} shows that Edixhoven's conjecture is implied by \cref{lgo-pel} (for $g$ replaced by some larger ~$g'$).

Going much further, one could ask whether a bound as in \cref{lgo-pel} might hold with constants depending only on~$g$, uniformly for all points of~$\Ag$ contained in proper special subvarieties of PEL type (without fixing a subvariety $V \subset \Ag$ and restricting to zero-dimensional components of intersections between $V$ and special subvarieties).
This would imply Coleman's conjecture on endomorphisms of abelian varieties of number fields \cite[Conj.~$\mathrm{C}(e,g)$]{BFGR06}, but we have very limited evidence for such a bold conjecture beyond the CM case. We do, however, note that a recent preprint of Fité and Goodman establishes Coleman's conjecture for abelian threefolds whose endomorphism ring is the ring of integers of an imaginary quadratic field \cite{FG}.

Related results have also been proved for commutative algebraic groups, in the form of height bounds for unlikely or ``just likely'' intersections with algebraic subgroups (note that all the known proofs of cases of LGO, including this paper, deduce the Galois bound from a height bound).
These results include \cite[Thm.~1]{BMZ99} and \cite[Lemme~3.3]{Rem05}, where $V$ is a curve, and \cite[Thm.]{Hab09:AVs}, \cite[Cor.~1.3]{Hab09:tori} and \cite[Thm.~2]{Kuh20}, where $V$ is of arbitrary dimension.
All of these results are restricted to intersections which are not in the so-called ``anomalous locus.''
This restriction prevents using the argument of \cref{unlikely-condition-superfluous} to remove the ``unlikely or just likely'' condition from the theorems on commutative group height bounds, because, in the construction of \cref{unlikely-condition-superfluous}, $\iota(V)$ is contained in the anomalous locus of~$W$.

\subsection{The strategy}

Our proof of \cref{main-bound} uses André's strategy from \cite[Ch.~X]{And89}: (i) near a point of multiplicative degeneration, the Taylor series of locally invariant periods of an abelian scheme (over~$\CC$) are G-functions; (ii) endomorphisms of fibres of the abelian scheme lead to algebraic relations between evaluations of the G-functions; (iii) by the ``Hasse principle for G-functions'', non-trivial global relations between evaluations of G-functions imply a height bound.
We then use (iv) Masser--Wüstholz isogeny estimates to convert the height bound into a Galois bound, similarly to previous work on LGO in moduli spaces of abelian varieties.

Since we deal with all endomorphism rings larger than~$\ZZ$, in particular including rings which inject into $\rM_g(\ZZ)$, points with these endomorphism rings may occur arbitrarily $p$-adically close to the point of multiplicative degeneration.  Therefore, we have to construct relations between evaluations of G-functions at all relevant places, both archimedean and non-archimedean.  This paper's new contributions in order to construct these relations are as follows.

\subsubsection{Non-archimedean interpretation of G-functions} \label{sssec:intro-non-arch-interp}

In order to interpret non-archimedean evaluations of the period G-functions, we use rigid analytic uniformisation of abelian schemes with multiplicative reduction.  This is a direct generalisation of the approach in our previous paper on $Y(1)^n$ \cite{Y(1)}, where we used the Tate uniformisation of elliptic curves.
Working with abelian varieties of dimension greater than~$1$ brings substantial new technical challenges.  We can no longer describe the uniformisation via explicit $q$-expansions, instead using rigid and formal geometry.
In order to ``match up'' the periods at different places of a number field, we use formal uniformisation of an abelian scheme, due to Raynaud \cite{Ray71} and Faltings and Chai \cite{FC90}, and verify compatibility between this formal uniformisation and the rigid and complex analytic uniformisations at all places.

In \cite{Y(1)}, we only obtained a non-archimedean interpretation of G-functions corresponding to periods of invariant differential forms.  In this paper, we go slightly further, considering also de Rham classes in the span of Gauss--Manin derivatives of invariant differential forms.  (In the setting of \cref{main-bound}, namely Hodge generic subvarieties of~$\Ag$, this includes all of $H^1_{DR}$, thanks to \cref{gauss-manin-basis}.)  This forces us to work with rigid uniformisations, rather than remaining purely in the world of formal uniformisations.  Being able to interpret the periods of these de Rham classes is essential to our construction of non-archimedean period relations.

In this paper, we only discuss multiplicative uniformisation of abelian schemes with completely multiplicative reduction.  It is likely that our methods would generalise to ``partial multiplicative uniformisation'' of abelian schemes with partially multiplicative reduction.  This would allow non-archimedean interpretation of those period G-functions which degenerate to periods of a torus.
These would be the same G-functions as considered in \cite[Thm.~4.3]{Urbanik} and \cite[Thm.~2.3]{Papas:heights} (restricted to abelian schemes), and would give an alternative non-archimedean interpretation of such G-functions (see Section \ref{sec:lit-review} for a comparison with Urbanik's approach).

\subsubsection{Non-archimedean relations}
For relations between non-archimedean periods, we use a new calculation via the rigid uniformisations established in~\ref{sssec:intro-non-arch-interp}.  As input, this calculation simply requires a single endomorphism of the abelian variety which is not in~$\ZZ$.
Urbanik and Papas have previously constructed non-archimedean period relations for abelian varieties with extra endormorphisms, under additional conditions \cite[Prop.~1.16]{Urbanik}, \cite[Lemma~9.1]{Papas:heights}.

\subsubsection{Archimedean relations}
In order to construct relations between archimedean periods, in most cases, we follow \cite[X, Construction~2.4.1]{And89}.
In fact, the construction from \cite{And89}, along with additional information from \cite[sec.~8]{ExCM}, suffice to deal with all cases except when $\End(A) \otimes \QQ$ is $\QQ \times \QQ$ or a real quadratic field, although not all cases were previously written down.  The reason we have fewer exceptions to this construction than \cite{ExCM} is that we only consider generic endomorphism ring~$\ZZ$, while \cite[X]{And89} and \cite[sec.~8]{ExCM} allowed the generic endomorphism algebra to be any totally real field of odd degree.
In the remaining cases, namely when $\End(A) \otimes \QQ$ is a real quadratic field or $\QQ \times \QQ$, we use new calculations to construct archimedean period relations.

\subsection{Literature review}\label{sec:lit-review}

In \cite{HP12}, Habegger and Pila proved the Zilber--Pink conjecture for so-called {\it asymmetric} curves in $Y(1)^n$. This was one of very few results beyond the Andr\'e--Oort conjecture at that time. They established LGO in this setting via a clever application of a theorem of N\'eron, in combination with Masser--W\"ustholz isogeny estimates. Unfortunately, however, their technique appears to depend essentially on the product decomposition of $Y(1)^n$, as well as the asymmetry condition in their theorem. 

In \cite{ExCM}, \cite{QRTUI}, and \cite{PELtype}, we observed that a technique of Andr\'e \cite[Ch.~X]{And89} could be applied to yield LGO, and hence Zilber--Pink, for certain curves in $\cA_g$ and a variety of endomorphism types, assuming multiplicative degeneration. At the same time, Binyamini and Masser used Andr\'e's technique to give effective versions of the Andr\'e--Oort conjecture for Hilbert modular varieties \cite{BM:AO}. 

Papas has since applied Andr\'e's method in a considerably more general situation, obtaining a height bound for certain exceptional points on arbitrary families of smooth projective varieties, under a ``normal crossings'' compactification condition \cite[arXiv v1]{Papas:heights}.
In a subsequent paper \cite{Papas:ZP}, he extended Andr\'e's method to prove LGO for ``strongly exceptional'' endomorphism rings in abelian schemes with only partially multiplicative degeneration (which is to say, toric rank at least $2$).

In the interim, the authors of this paper extended Andr\'e's technique to $Y(1)^n$ \cite{Y(1)}. The important novelty was that our approach could handle period relations at unboundedly many finite places, bypassing the ``exceptional endomorphisms'' constraint of previous applications of the technique. Andr\'e had previously constructed relations at a fixed finite set of finite places (see \cite{And95}, for example), but, to our knowledge, this was the first work to handle all of them simultaneously.

Recently, in \cite{Urbanik}, Urbanik has established an alternative approach to non-archimedean interpretations of geometric G-functions using modern developments in $p$-adic Hodge theory due to Scholze.
As with our approach, this facilitates the use of period relations in Andr\'e's technique at all finite places.
With his approach, Urbanik obtains height bounds for points possessing a $\QQ$-Hodge summand with complex multiplication. This enables him to prove several striking results on Zilber--Pink \cite[Theorems 1.1 and~1.5]{Urbanik}.

We note that Urbanik's approach is more general than ours in that (i) it deals with variations of Hodge structures coming from arbitrary families of algebraic varieties, (ii) it has access to periods with respect to the full de Rham cohomology (not just derivatives of invariant differential forms), and (iii) it allows degenerations defined by a normal crossings compactification condition, including (in the abelian variety context) some degenerations which are only partially multiplicative, as in \cite{Papas:ZP}.
The restriction to families of abelian varieties seems essential to our approach.  A new idea would also be required to remove the restriction to de Rham classes which are derivatives of invariant differential forms. Our method could, however, be extended to partially multiplicative degenerations, using rigid uniformisations of abelian varieties via Raynaud extensions.

Most recently, Papas has released a new version of \cite{Papas:heights}, incorporating the results and ideas of Urbanik.
In the abelian setting, assuming multiplicative degeneration, he establishes LGO for families of Jacobians of curves of odd genus, as well as many cases for even genus \cite[Theorem~1.4]{Papas:heights}.
We believe his methods are likely capable of proving a result for Jacobians with a notion of unlikely endomorphisms only slightly more restrictive than in \cref{main-bound}.

\subsection{Structure of the paper}

In \cref{sec:prelims}, we introduce some preliminary material. We discuss rigid analytification functors and their relevant properties. We also state and prove some necessary results on the composition of power series. We include a few remarks on invariant differential forms and their derivatives. In \cref{sec:unif}, we discuss various (compatible) uniformisations of semiabelian schemes with multiplicative degeneration.

In \cref{sec:formal-periods}, we prove \cref{exist-formal-periods}. This is the result that produces G-functions whose $v$-adic evaluations, at any place $v$, are $v$-adic periods. In \cref{sec:relations}, we construct global relations between our G-functions at points with unlikely endormorphisms.

In \cref{sec:main-proofs}, we prove our main results. We begin with some auxiliary results on fibered surfaces. We make some reduction steps, using Gabber's Lemma and \cite[Lemma 5.1]{Y(1)}. After the setup is complete, we invoke \cref{exist-formal-periods} to obtain our period G-functions and \cref{exist-relations} to obtain global relations between them. Finally, we establish that these relations are non-trivial.

\subsection*{Acknowledgements}

The authors wish to thank Ariyan Javanpeykar for several helpful suggestions and, in particular, for pointing them to Gabber's Lemma. They thank George Papas and David Urbanik, also for their comments, and for sharing early versions of their work.
They thank Werner L\"utkebohmert for valuable discussions about uniformisation of abelian schemes in rigid geometry. They are grateful to the anonymous referees for many very helpful comments.

\section{Preliminaries}\label{sec:prelims}

\subsection{Notation}

Often we shall work over a base scheme~$C$, with a chosen closed subscheme $C_0$ of~$C$.
We shall write $C_\formal$ for the formal completion of $C$ along $C_0$.
If $X$ is a scheme over~$C$, then $X_0$ will denote the base change $X \times_C C_0$ and $X_\formal$ will denote the formal completion of $X$ along $X_0$, which is canonically isomorphic to the base change $X \times_C C_\formal$ in the category of formal schemes \cite[Prop.~10.9.7]{EGAI}.
This notation may also be applied to morphisms of schemes over~$C$.

If $K$ is a discretely valued non-archimedean field, then a \defterm{$K$-analytic space} shall mean a rigid analytic space over~$K$.

If $K$ is either a discretely valued non-archimedean field or~$\CC$, we write $D(r, K)$ for the open disc of radius~$r$ and centre~$0$ over~$K$, considered as a $K$-analytic space.

\subsection{Places of a number field}

Let $K$ be a number field.
Recall that a \defterm{place} of $K$ is an equivalence class of absolute values on~$K$.
The non-archimedean places of $K$ are naturally in bijection with the maximal ideals of the ring of integers~$\cO_K$, while the archimedean places of $K$ are naturally in bijection with the embeddings $K \to \CC$, up to complex conjugation.

For a place $v$ of $K$, let $K_v$ denote the $v$-adic completion of $K$ if $v$ is non-archimedean, and write $K_v=\CC$ if $v$ is archimedean.
Note that this notation is not standard when $v$ is archimedean: even if the $v$-adic completion of~$K$ is isomorphic to~$\RR$, we still write $K_v=\CC$.
This is because we wish to do complex analytic, rather than real analytic, geometry, because of its close links to algebraic geometry.

Let $R$ be a Dedekind domain whose fraction field is a number field~$K$ (in other words, $R$ is a localisation of the ring of integers~$\cO_K$).
We say that a place $v$ of $K$ is \defterm{visible to~$R$} if either $v$ is archimedean or $v$ is non-archimedean and $R$ is contained in the $v$-adic completion of $\cO_K$ (equivalently, no element of the prime ideal of $\cO_K$ corresponding to~$v$ is invertible in~$R$).

\subsection{Inverse function theorem for globally bounded power series}

Let $K$ be a number field.
Let $f(X) = \sum_{n=0}^\infty a_nX^n \in \powerseries{K}{X}$ be a power series.
For each place $v$ of $K$, write $R_v(f) \in \RR_{\geq 0} \cup \{\infty\}$ for the $v$-adic radius of convergence of $f$.
Write $f^\van$ for the $K_v$-analytic function $D(R_v(f), K_v) \to K_v$ defined by $v$-adic evaluation of~$f$.

The power series $f$ is said to be \defterm{globally bounded} if:
\begin{enumerate}[(a)]
\item for every place $v$ of~$K$, $R_v(f) > 0$;
\item for almost all places $v$ of~$K$, for all~$n \geq 0$, $\abs{a_n}_v \leq 1$.
\end{enumerate}

\begin{lemma} \label{power-series-locally-invertible}
Let $f \in \powerseries{K}{X}$ be a globally bounded power series satisfying $f(0)=0$ and $f'(0) \neq 0$.
Then, for each place $v$ of~$K$, there exists a $K_v$-analytic open neighbourhood $Y_v$ of~$0$ in~$K_v$ and a real number $r^*_v(f)$ such that:
\begin{enumerate}[(a)]
\item $r^*_v(f) > 0$ for all $v$ and $r^*_v(f) = 1$ for almost all~$v$;
\item $Y_v$ is contained in the open disc $D(R_v(f), K_v)$ for all~$v$;
\item for every place $v$, $f^\van$ restricts to an isomorphism (of $K_v$-analytic spaces) from~$Y_v$ to $D(r^*_v(f),K_v)$;
\item for almost all non-archimedean places~$v$, $Y_v = D(1,K_v)$.
\end{enumerate}
\end{lemma}

\begin{proof}
Since $f$ is globally bounded, for almost all non-archimedean places $v$ of~$K$, $f \in \powerseries{\cO_{K,v}}{X}$.
For almost all places, we have furthermore $f'(0) \in \cO_{K,v}^\times$.
For a place with these properties, $f$ has a compositional inverse $g \in \powerseries{\cO_{K,v}}{X}$.
Then $f^\van$ and $g^\van$ both define $K_v$-analytic maps $D(1,K_v) \to D(1,K_v)$, which are inverse to each other, so $f^\van$ restricts to an isomorphism $D(1,K_v) \to D(1,K_v)$.
Therefore, such places satisfy (b) and~(c) with $r_v^*(f)=1$ and $Y_v=D(1,K_v)$.

For the remaining non-archimedean places, we can choose $a_v, b_v \in K_v$ such that $h_v(X) = a_vf(b_vX)$ satisfies $h_v \in \powerseries{\cO_{K,v}}{X}$ and $h_v'(0) \in \cO_{K,v}^\times$.
By the argument above, $h_v^\van$ restricts to an isomorphism $D(1,K_v) \to D(1,K_v)$.
Hence $f^\van$ restricts to an isomorphism $D(\abs{b_v}^{-1}, K_v) \to D(\abs{a_v}, K_v)$.

For archimedean places~$v$, (b) and~(c) with $r_v^*(f) > 0$ follow from the inverse function theorem applied to the holomorphic function $f^\van$.
\end{proof}

\subsection{Analytification functors} \label{subsec:rigid-analytification-functors}

In this section, we discuss two analytification functors and their relation to one another.

\subsubsection{Analytification of schemes}

Let $K$ be either $\CC$ or a discretely valued non-archimedean field.
We shall write $X \mapsto X^\an$ for the classical analytification functor
\[ \an \colon \{ \text{schemes locally of finite type over } K \} \to \{ \text{analytic spaces over } K \}. \]
The underlying point set of $X^\an$ is the set of closed points of a $K$-scheme~$X$.
For example, if $X$ is the affine line (as a $K$-scheme), then $X^{\an}$ is the affine line (as a $K$-analytic space).

If $K$ is a number field, $v$ is a place of~$K$, and $X$ is a $K$-scheme locally of finite type, then we write $X^\van$ for the $v$-adic analytification $(X \times_{\Spec(K)} \Spec(K_v))^\an$.

\subsubsection{Rigid generic fibre for formal schemes}

Let $R$ be a complete discrete valuation ring with maximal ideal $\fp$ and let $K = \Frac(R)$.
We shall write $\Spf(R)$ to mean $\Spf(R, \fp)$.
Write $\tatealgebra{R}{T}$ for the $\fp$-adic Tate algebra over~$R$ in one variable~$T$.

We write $\fX \mapsto \fX^\rig$ for Berthelot's rigid generic fibre functor
\begin{multline*}
\rig \colon \left\{
\begin{aligned}
  & \text{locally noetherian formal schemes } \fX \text{ over } \Spf(R)
\\& \text{whose reduction $\fX_\red$ is locally of finite type over $\Spec(R)$}
\end{aligned} \right\}
\\ \to \{ \text{rigid spaces over } K \}.
\end{multline*}
This functor is defined at \cite[0.2.6]{Berthelot}.
It generalises Raynaud's generic fibre functor which is defined only for formal schemes which are themselves locally of finite type over~$\Spf(R)$ (in particular, formal schemes in the domain of Raynaud's functor must be adic over~$\Spf(R)$).

For example, if $\fX_1 \cong \Spf(\tatealgebra{R}{T}, (\fp))$ is the formal completion of $\AAA^1_R$ along the closed fibre of $\AAA^1_R \to \Spec(R)$, then $\fX_1^{\rig}$ is the closed unit disc (as a $K$-analytic space).
If $\fX_2 \cong \Spf(\powerseries{R}{T}, (\fp, T))$ is the formal completion of $\AAA^1_R$ along the origin in the closed fibre of $\AAA^1_R \to \Spec(R)$, then $\fX_2^\rig$ is the open unit disc.
Note that $\fX_1$ lies within the domain of Raynaud's generic fibre functor, but $\fX_2$ does not.


\subsubsection{Relationship between the functors \texorpdfstring{$\an$}{an} and \texorpdfstring{$\rig$}{rig}}

Again let $R$ be a complete discrete valuation ring with maximal ideal~$\fp$.  Let $K = \Frac(R)$ and $k = R/\fp$.

The following classical theorem describes the relationship between Raynaud's generic fibre functor for formal schemes and the $K$-analytification functor for $K$-schemes.
\begin{proposition} \label{rig-an-compatibility-raynaud} \cite[Prop.~0.3.5]{Berthelot}
Let $\fX$ be a scheme locally of finite type over~$R$.
Let $X = \fX \otimes_R K$ and let $\fX_{\fp,\formal}$ denote the formal completion of $\fX$ along its fibre over $\fp \in \Spec(R)$.
Then there is a canonical morphism of $K$-rigid spaces $\alpha \colon \fX_{\fp,\formal}^{\rig} \to X^{\an}$.
If $\fX$ is separated and of finite type over~$R$, then $\alpha$ is an open immersion.
\end{proposition}

\Cref{rig-an-compatibility-raynaud} only applies to the completion of $\fX$ along its special fibre; in other words, the formal scheme $\fX_{\fp,\formal}$ is always in the domain of Raynaud's generic fibre functor.
We shall require a generalisation of \cref{rig-an-compatibility-raynaud} to Berthelot's generic fibre functor, allowing completions along arbitrary closed subsets of the special fibre, as follows.

\begin{proposition} \label{rig-an-compatibility}
Let $\fX$ be a separated $R$-scheme of finite type and let $X = \fX \times_{\Spec(R)} \Spec(K)$.
Let $\fX_\fp$ be the closed fibre of $\fX \to \Spec(R)$ and let $\fX_{0,\fp}$ be a closed subscheme of $\fX_\fp$.
Let $\fX_{\formal\dag}$ denote the formal completion of $\fX$ along $\fX_{0,\fp}$.
Then there is a canonical open immersion $\fX_{\formal\dag}^\rig \to X^\an$.
\end{proposition}

\begin{proof}
Let $\fX_{\fp,\formal}$ denote the formal completion of $\fX$ along $\fX_\fp$.
Applying Berthelot's functor to $\fX_{\fp,\formal}$ and $\fX_{\formal\dag}$, and using \cite[Prop.~0.2.7]{Berthelot} with $X = \fX_\formal$, $Z = \fX_{0,\fp}$, $\hat X = \fX_{\formal\dag}$, we obtain a canonical open immersion of rigid spaces
\[ \fX_{\formal\dag}^\rig \to \fX_{\fp,\formal}^\rig. \]
By \cref{rig-an-compatibility}, there is a canonical open immersion of rigid spaces
\[ \fX_{\fp,\formal}^\rig \to X^\an. \]
Composing these immersions yields the required open immersion $\fX_{\formal\dag}^\rig \to X^\an$.
\end{proof}

As an example of \cref{rig-an-compatibility}, consider $\fX = \AAA^1_R = \Spec(R[T])$ and $\fX_{0,\fp} = \{0\} \subset \AAA^1_k = \fX_\fp$ (that is, $\fX_{0,\fp}$ is the closed subscheme of $\fX_\fp$ corresponding to the ideal $(T) \subset k[T]$).
As we noted previously, in this example,
\[ \fX_{\fp,\formal} \cong \Spf(\tatealgebra{R}{T}, (\fp)), \quad \fX_{\formal\dag} \cong \Spf(\powerseries{R}{T}, (\fp, T)), \]
$X^\an$ is the rigid affine line $(\AAA^1_K)^\an$, $\fX_{\fp,\formal}^\rig$ is the closed unit disc, and $\fX_{\formal\dag}^\rig$ is the open unit disc.
The open immersions $\fX_{\formal\dag}^\rig \to \fX_{\fp,\formal}^\rig \to X^\an$ appearing in \cref{rig-an-compatibility} and its proof are the natural open immersions of $K$-rigid spaces
\[ D(1,K) \to\text{closed unit disc} \to  (\AAA^1_K)^\an. \]

\subsection{Invariant forms on abelian schemes and their derivatives} \label{subsec:derivatives-forms}

If $C$ is a scheme and $\pi \colon G \to C$ is a group scheme, we write $\Omega^\inv_{G/C}$ for the $\cO_C$-module of $G$-translation-invariant forms in $\pi_*\Omega^1_{G/C}$.
By \cite[Proposition~3.15]{EvdGM}, $\Omega^{\inv}_{G/C}$ is canonically isomorphic to $e^*\Omega^1_{G/C}$, where $e \colon C \to G$ is the zero section.
If $G/C$ is smooth, then (by definition) $\Omega^1_{G/C}$ is locally free, so $\Omega^{\inv}_{G/C}$ is locally free.

Suppose that $C$ is a smooth irreducible $K$-variety, where $K$ is a field of characteristic zero, and that $G \to C$ is an abelian scheme.
Then we get an injection of $\cO_C$-modules $\Omega^{\inv}_{G/C} \to H^1_{DR}(G/C)$ \cite[(1.2)]{And17}.

Let $\nabla \colon H^1_{DR}(G/C) \to H^1_{DR}(G/C) \otimes_{\cO_C} \Omega^1_C$ denote the Gauss--Manin connection attached to the morphism $G \to C$.
This connection gives rise to a $\cD_C$-module structure on $H^1_{DR}(G/C)$, where $\cD_C$ is the ring of differential operators on~$C$.
Let $\cM_{G/C}$ denote the $\cD_C$-submodule of $H^1_{DR}(G/C)$ generated by $\Omega^{\inv}_{G/C}$.
In other words, letting $T_C$ denote the tangent bundle of~$C$ (which we can identify with the differential operators on~$C$ of order~$1$), we have
\[ \cM_{G/C} = \sum_{n=0}^\infty \sum_{\partial_1, \dotsc, \partial_n \in T_C} \cO_C \nabla_{\partial_1} \dotsm \nabla_{\partial_n} \Omega^{\inv}_{G/C}. \]

In the situations for which this paper proves the large Galois orbits conjecture, $\cM_{G/C} = H^1_{DR}(G/C)$, as shown by the following lemma due to André \cite{And17}.
As noted in \cite[4.1.2]{And17}, \cref{gauss-manin-basis} is not always true when the geometric generic endomorphism algebra has Albert type~IV.

\begin{lemma} \label{gauss-manin-basis}
Let $K$ be a subfield of~$\bC$, let $C$ be a smooth irreducible $K$-variety, and let $G \to C$ be an abelian scheme.
Suppose that the abelian variety $G_{\ov{K(C)}}$ is simple, $\End(G_{\ov{K(C)}}) \otimes_\ZZ \QQ$ has Albert type I, II or III, and the monodromy group of $G_\CC/C_\CC$ is as large as possible given its endomorphisms and polarisation (for a single choice of embedding $K \to \CC$).
Then $\cM_{G/C} = H^1_{DR}(G/C)$.
\end{lemma}

\begin{proof}
The rank $\rk(\cM_{G/C}/\Omega^{\inv}_{G/C})$ is called $r$ in \cite{And17}.
By the hypotheses of the lemma, $G/C$ is an abelian scheme of ``restricted PEM type'' in the sense of \cite{And17}.
Hence, by \cite[Thm.~4.2.2]{And17}, we have $r = g$.
Hence,
\[ \rk(\cM_{G/C}) = \rk(\cM_{G/C}/\Omega^{\inv}_{G/C}) + \rk(\Omega^{\inv}_{G/C}) = g + g = \rk(H^1_{DR}(G/C)).
\qedhere \]
\end{proof}

\section{Uniformisation of semiabelian schemes with multiplicative degeneration}\label{sec:unif}

The goal of this section is to describe notions of ``multiplicative uniformisation'' for abelian schemes in formal geometry, rigid geometry and complex analytic geometry.
This generalises the Tate uniformisation for elliptic curves, which was used in the authors' previous work on the Zilber--Pink conjecture for~$Y(1)^n$ \cite{Y(1)}.

Later in this paper, we shall be concerned with abelian schemes over a number field.
In this case, we may have an analytic uniformisation (either rigid or complex analytic) for each place of the number field.
It is important for our application that the uniformisations at different places should be ``compatible'' with each other.
In the case of elliptic curves, this compatibility can be described in a down-to-earth way: the Tate uniformisation is described by explicit power series ($q$-expansions), independent of the place.
When dealing with higher-dimensional abelian varieties, it is difficult to be so explicit.
This is where formal uniformisation comes in: we can formally uniformise our abelian scheme over the ring of integers of the number field, then require each analytic uniformisation to be compatible with the formal uniformisation after base change to the appropriate completion of the number field.

Formal uniformisation appeared in \cite{Ray71}, and is briefly sketched in greater generality in \cite[II.1]{FC90}.
Rigid analytic uniformisation for an abelian variety over a point is due to Raynaud \cite{Ray71}, and has been studied in detail by Bosch and Lütkebohmert \cite{BL84,BL91}.
The existence of rigid uniformisation in the relative setting seems to be new, and is proved in a companion article by the second-named author \cite{Orr:rigid-homs}.
Complex analytic ``multiplicative'' uniformisation of an abelian scheme, in the form in which we use it, is described in \cite[IX.4]{And89}.

Note that, throughout this paper, when we talk about ``uniformisation'' of an abelian scheme, we mean uniformisation by a split torus $\GG_m^g$, not the classical complex analytic uniformisation by a vector space.
Therefore, it applies only to abelian schemes with multiplicative degeneration in a suitable sense.
We expect that one could extend the definitions and results in this section to uniformisation by semiabelian schemes at each place (Raynaud extensions), but this would introduce additional technical complications, and it is not clear what statements one could make about compatibility between places.

On the other hand, as mentioned in Section \ref{sssec:intro-non-arch-interp}, for abelian schemes with partially multiplicative degeneration, it should be relatively straightforward to generalise the results of this section to a ``partial multiplicative uniformisation.''
By this, we mean uniformisation of the relative subgroup which degenerates to a torus (over a formal or small analytic base).
This would suffice to establish the analogue of \cref{exist-formal-periods} for those periods which degenerate to periods of the torus.

Although the machinery behind rigid and complex analytic uniformisations is different, we will often be able to treat them uniformly in the applications.  We have chosen notations to facilitate this uniformity as much as possible.

\subsection{Formal uniformisation}

Let $\fC$ be a noetherian integral scheme and let $\fC_0 \subset \fC$ be a non-empty closed subscheme.
Write $\fC_\formal$ for the formal completion of $\fC$ along~$\fC_0$.
Let $\fG \to \fC$ be a semiabelian scheme.
We write $\fG_0 := \fG\times_{\fC}\fC_0$ and $\fG_\formal := \fG \times_{\fC} \fC_{\formal}$ (in other words, $\fG_\formal$ is the formal completion of $\fG$ along $\fG_0$).
Suppose that $\fG_0 \cong \GG_m^g \times \fC_0$.

\begin{definition}
A \defterm{formal uniformisation of $\fG$ over $\fC_\formal$} is an isomorphism of $\fC_\formal$-formal group schemes $\phi_\formal \colon \GG_m^g \times \fC_\formal \to \fG_\formal$.
\end{definition}

\begin{proposition} \label{exists-formal-uniformisation}
If $\fG \to \fC$ is a semiabelian scheme and $\phi_0 \colon \GG_m^g \times \fC_0 \to \fG_0$ is an isomorphism of $\fC_0$-group schemes, then there exists a unique formal uniformisation of~$\fG$ over~$\fC_\formal$ whose base change to $\fC_0$ is equal to $\phi_0$.
\end{proposition}

\begin{proof}
The uniqueness in the proposition ensures that we can glue together formal uniformisations over an open cover of~$\fC$.
Thus it suffices to prove the proposition under the assumption that $\fC$ is affine.

Write $\fC = \Spec(\fR)$ and let $\fI \subset \fR$ be the ideal which defines $\fC_0$.
Let $\fC_i = \Spec(\fR/\fI^{i+1})$ for each $i \in \ZZ_{\geq0}$.

For each~$i$, $\fG_i := \fG \times_\fC \fC_i$ is a group of multiplicative type by \cite[X, Cor.~2.3]{SGA3mult}.
Hence, by \cite[IX, Thm.~3.6]{SGA3mult}, for each $i \in \ZZ_{>0}$, there exists a homomorphism of $\fC_i$-group schemes $\phi_i \colon \GG_m^g \times \fC_i \to \fG_i$ which lifts $\phi_0$, as well as a homomorphism $\phi_i' \colon \fG_i \to \GG_m^g \times \fC_i$ lifting $\phi_0^{-1}$.
Since $\fG$ and $\GG_m^g$ are commutative, by \cite[IX, Cor.~3.4]{SGA3mult}, $\phi_i$ and $\phi_i'$ are inverses, so $\phi_i$ is an isomorphism $\GG_m^g \times \fC_i \to \fG_i$.
Again using the commutativity of~$\fG$, each~$\phi_i$ is unique, so the $\phi_i$ are compatible with each other.
Hence we can take the limit to obtain an isomorphism $\phi_\formal \colon \GG_m^g \times \fC_\formal \to \fG_\formal$.
\end{proof}

\subsection{Rigid analytic uniformisation} \label{subsec:rigid-uniformisation}

Let $R$ be a complete discrete valuation ring with maximal ideal $\fp$ and let $K = \Frac(R)$.

Let $\GG_m(1)$ denote the open subgroup of the rigid $K$-group $\GG_m^\an$ defined by
\[ \GG_m(1) = \{ x \in \GG_m^\an : \abs{x} = 1 \}. \]
Note that $\GG_m(1) = \GG_{m,\formal}^\rig$, where $\GG_{m,\formal}$ is the formal completion of $\GG_{m,R}$ along the closed fibre of $\GG_{m,R} \to \Spec(R)$.

Let $\cC$ be a rigid space over~$K$.
Let $\cG$ be a smooth commutative rigid $\cC$-group.

\begin{definition}
A \defterm{formal rigid uniformisation} of $\cG$ over~$\cC$ is an isomorphism of rigid $\cC$-groups from $\GG_m(1)^g \times \cC$ to an open $\cC$-subgroup $\cH \subset \cG$.
\end{definition}

\begin{definition}
A \defterm{rigid uniformisation} of $\cG$ over~$\cC$ is a homomorphism of rigid $\cC$-groups $\phi \colon (\GG_m^\an)^g \times \cC \to \cG$ whose restriction to $\GG_m(1)^g \times \cC$ is a formal rigid uniformisation of~$\cG$.
\end{definition}

\begin{proposition} \label{exists-rigid-uniformisation}
Suppose that $\cC$ is geometrically reduced, quasi-compact and quasi-separated.
Let $\cG \to \cC$ be a smooth separated commutative rigid $\cC$-group with connected fibres, such that the morphism $\cG \to \cC$ is without boundary, in the sense of \cite[Def.~1.5]{Lam99}.
If $\cG$ possesses a formal rigid uniformisation $\bar\phi \colon \GG_m(1)^g \times \cC \to \cG$, then there is a unique rigid uniformisation $\phi \colon (\GG_m^\an)^g \times \cC \to \cG$ which restricts to $\bar\phi$.
\end{proposition}

\begin{proof}
This is an immediate consequence of \cite[Cor.~1.3]{Orr:rigid-homs}.
The only condition of \cite[Cor.~1.3]{Orr:rigid-homs} which requires checking is condition~(i).
By the definition of formal rigid uniformisation, the image of $\bar\phi$ is an open $\cC$-subgroup $\cH \subset \cG$ which is isomorphic to $\GG_m(1)^g \times \cC$, hence condition~(i) of \cite[Cor.~1.3]{Orr:rigid-homs} is satisfied.
\end{proof}

\subsubsection{Compatibility with formal uniformisation} \label{sssec:formal-compatibility}

Let $\fC$ be a normal integral $R$-scheme of finite type, let $\fC_\fp$ denote the closed fibre of $\fC \to \Spec(R)$, and let $\fC_{0,\fp}$ be a closed subscheme of $\fC_\fp$.
Let $C$ denote the generic fibre of $\fC \to \Spec(R)$.
Let $\fC_{\formal\dag}$ denote the formal completion of $\fC$ along~$\fC_{0,\fp}$.
Applying \cref{rig-an-compatibility} to $\fC$, we obtain a canonical open immersion
\[ \fC_{\formal\dag}^\rig \to C^\an. \]

Let $\fG \to \fC$ be a semiabelian scheme such that $\fG_{0\dag} := \fG \times_\fC \fC_{0,\fp}$ is isomorphic to $\GG_m^g \times \fC_{0,\fp}$.
Letting $\fG_{\formal\dag} = \fG \times_\fC \fC_{\formal\dag}$ and applying \cref{rig-an-compatibility}, we obtain a canonical open immersion of rigid spaces
\[ \fG_{\formal\dag}^\rig \to G^\an \]
compatible with the open immersion $\fC_{\formal\dag}^\rig \to C^\an$ and with the rigid $C^\an$-group structures on $\fG_{\formal\dag}^\rig$ and $G^\an$.

Let $\cC = \fC_{\formal\dag}^\rig$, which we identify with its image in $C^\an$.
Write $\cG = G^\an|_\cC$, so that $\fG_{\formal\dag}^\rig$ is an open $\cC$-subgroup of $\cG$.

Let $\phi_{\formal\dag} \colon \GG_m^g \times \fC_{\formal\dag} \to \fG_{\formal\dag}$ be a formal uniformisation over $\fC_{\formal\dag}$.
(Such a $\phi_{\formal\dag}$ exists by \cref{exists-formal-uniformisation}.)
Applying Berthelot's functor to $\phi_{\formal\dag}$, we obtain an isomorphism of rigid $\cC$-groups
\[ \phi_{\formal\dag}^\rig \colon (\GG_m^g \times \fC_{\formal\dag})^\rig \to \fG_{\formal\dag}^\rig. \]
Note that $(\GG_m^g \times \fC_{\formal\dag})^\rig$ is canonically isomorphic to $\GG_m(1)^g \times \cC$, so $\phi_{\formal\dag}^\rig$ is a formal rigid uniformisation of $\cG$ over~$\cC$.

\begin{definition}
We say that a rigid uniformisation $\phi$ of $\cG$ over~$\cC$ is \defterm{compatible} with $\phi_{\formal\dag}$ if the restriction of $\phi$ to $\GG_m(1)^g \times \cC$ is equal to~$\phi_{\formal\dag}^\rig$.
\end{definition}

\begin{lemma} \label{exists-rigid-uniformisation-compatible}
Let $\phi_{\formal\dag}$ be a formal uniformisation of $\fG_{\formal\dag}$ over $\fC_{\formal\dag}$.
There exists a unique rigid uniformisation of $\cG$ over~$\cC$ compatible with $\phi_{\formal\dag}$.
\end{lemma}

\begin{proof}
By \cite[Cor.~5.11]{Lut90}, $G^\an \to \Sp(K)$ is without boundary.
Since $C \to \Spec(K)$ is separated, so is $C^\an \to \Sp(K)$.
Applying \cite[Prop.~1.13]{Lam99} to $\varphi \colon G^\an \to S^\an$ and $\psi \colon C^\an \to \Sp(K)$, we deduce that $G^\an \to C^\an$ is without boundary.
By \cite[Prop.~1.14]{Lam99}, the property of being without boundary is preserved by base change, so it follows that $\cG \to \cC$ is without boundary.
Therefore, we can apply \cref{exists-rigid-uniformisation} to $\bar\phi = \phi_{\formal\dag}^\rig \colon \GG_m(1)^g \times \cC \to \cG$, proving the lemma.
\end{proof}

\subsection{Complex uniformisation}

Let $\cC$ be a simply connected complex manifold of dimension~$1$, let $s_0 \in \cC$ and let $\cC^* = \cC \setminus \{s_0\}$.
Let $\cG \to \cC$ be a smooth commutative analytic $\cC$-group such that $\cG_0:=\cG_{s_0} \cong (\GG_{m,\CC}^\an)^g$ and $\cG|_{\cC^*}$ is proper.

We define a uniformisation of $\cG$ by a multiplicative rather than additive group because we are interested in $\cC$-groups with multiplicative degeneration and want to match our definitions of formal and rigid uniformisations.

\begin{definition}
A \defterm{complex analytic uniformisation} of $\cG$ over~$\cC$ is a homomorphism of analytic $\cC$-groups $\phi \colon (\GG_{m,\CC}^\an)^g \times \cC \to \cG$ which restricts to an isomorphism $(\GG_{m,\CC}^\an)^g \to \cG_0$.
\end{definition}

\begin{lemma} \label{exists-complex-uniformisation}
Let $\cG \to \cC$ be a smooth commutative complex analytic $\cC$-group such that $\cG_0 \cong (\GG_{m,\CC}^\an)^g$ and $\cG|_{\cC^*}$ is proper.
Then there exists a complex analytic uniformisation $(\GG_{m,\CC}^\an)^g \times \cC \to \cG$.
\end{lemma}

\begin{proof}
Consider the exponential map $\exp \colon \Lie(\cG/\cC) \to \cG$.
Since $\cC$ is simply connected, $\Lie(\cG/\cC) \cong \CC^g \times \cC$.
The kernel of $\exp$ is a $\cC$-subgroup of $\Lie(\cG/\cC)$ with discrete fibres.
Above $\cC^*$, $\ker(\exp)$ is locally constant of rank~$2g$.

Let $\Lambda$ denote the group of global sections $\cC \to \ker(\exp)$.  Let $\Lambda_\cC$ be the union of the images of these global sections, which is a $\cC$-subgroup of $\Lie(\cG/\cC)$. According to \cite[IX, 4.3]{And89}, the natural map $\Lambda \to (\Lambda_\cC)_0$ is bijective, so $\Lambda \cong \ZZ^g$.
Furthermore, $\Lambda \to (\Lambda_\cC)_s$ is injective for all $s \in \cC$.
Hence $\Lambda_\cC$ is a constant $\cC$-group, isomorphic to $\ZZ^g \times \cC$.

We deduce that there is an isomorphism of complex analytic $\cC$-groups
\[ \Lie(\cG/\cC) / \Lambda_\cC \cong (\CC^g \times \cC) / (\ZZ^g \times \cC) \cong (\GG_{m,\CC}^\an)^g \times \cC. \]
Now, $\exp \colon \Lie(\cG/\cC) \to \cG$ factors through the quotient $\Lie(\cG/\cC) / \Lambda_\cC$, so this isomorphism induces a homomorphism of analytic $\cC$-groups $\phi \colon (\GG_{m,\CC}^\an)^g \times \cC \to \cG$.
Since $\Lambda \to (\Lambda_\cC)_0$ is bijective, the restriction $\phi_0 \colon (\GG_{m,\CC}^\an)^g \to \cG_0$ is an isomorphism.
Thus $\phi$ is a complex analytic uniformisation, as required.
\end{proof}

\subsubsection{Compatibility with formal uniformisation}

Let $C$ be a smooth algebraic curve over~$\CC$, let $C^*$ be a Zariski open subset of~$C$, and let $s_0 \in C(\CC) \setminus C^*(\CC)$.
Let $G \to C$ be a smooth commutative $S$-group scheme such that $G_0 \cong \GG_{m,\CC}^g$ and $G|_{C^*}$ is an abelian scheme.

We shall make use of the notion of formal completion of analytic spaces over~$\CC$ (see \cite{Har75}).
Recall that, for any $\CC$-scheme~$X$ of finite type and any closed subscheme $X_0 \subset X$, if $X^\an_\formal$ denotes the formal completion of $X^\an$ along $X_0^\an$, then there is a commutative diagram of ringed spaces \cite[p.~20]{Har75}
\begin{equation} \label{eqn:an-for-diagram}
\begin{tikzcd}
    X^\an_\formal  \ar[r] \ar[d]
  & X^\an          \ar[d]
\\  X_\formal      \ar[r]
  & X.
\end{tikzcd}
\end{equation}

Let $\cC$ be a simply connected open neighbourhood of $s_0^\an$ in $C^\an$, such that $\cC \setminus \{ s_0^\an \} \subset (C^*)^\an$.
Then $\cC_\formal$ (the formal completion of $\cC$ along $\{s_0\}$) is canonically isomorphic to $C^\an_\formal$.
Let $\cG = G^\an|_{\cC}$.

\begin{definition}
Let $\phi_\formal \colon \GG_m^g \times C_\formal \to G_\formal$ be a formal uniformisation of $G$ over~$C_\formal$.
Let $\phi_\an \colon (\GG_{m,\CC}^\an)^g \times \cC \to \cG$ be a complex analytic uniformisation of $\cG$ over~$\cC$.
We say that $\phi_\formal$ and $\phi_\an$ are \defterm{compatible} if they have the same base change to $\cC_\formal \cong C^\an_\formal$.
\end{definition}

In order to establish the existence of a complex analytic uniformisation compatible with a given formal uniformisation, we shall require the following lemma, which is a complex analytic version of \cite[IX, Cor.~3.4]{SGA3mult}. Its proof is based on Jason Starr's comments on a MathOverflow question of the second author \cite{MO:S-groups}.

\begin{lemma} \label{formal-analytic-torus-end}
Let $\cC$ be a complex manifold (not necessarily of dimension~$1$) and let $s_0 \in \cC$.
Let $\cC_{\formal}$ denote the formal completion of~$\cC$ along~$\{s_0\}$.
Then the map $\End_{\cC_\formal\mathrm{\mathbf{-Grp}}}((\GG_{m,\CC}^\an)^g \times \cC_\formal) \to \End_{\CC\mathrm{\mathbf{-Grp}}}((\GG_{m,\CC}^\an)^g)$ given by base change along $\{ s_0 \} \to \cC_\formal$ is injective.
\end{lemma}

\begin{proof}
Since $\End_{\textbf{Grp}}(\GG_m^g) \cong \rM_{g \times g}(\End_{\textbf{Grp}}(\GG_m))$ (over any base space), it suffices to prove the lemma for $g=1$.

Let $\cI \subset \cO_\cC$ be the ideal sheaf of the closed subset $\{s_0\}$.
For $n \in \ZZ_{\geq 0}$, let $\cC_n$ denote the analytic subvariety of~$\cC$ defined by the ideal~$\cI^{n+1}$ (a fat point).
Let $\cG_n = \GG_{m,\CC}^\an \times \cC_n$.
Since $\cC_\formal = \varinjlim_n \cC_n$, it suffices to prove that
\[ \End_{\cC_n\textbf{-Grp}}(\cG_n) \to \End_{\cC_{n-1}\textbf{-Grp}}(\cG_{n-1}) \]
is injective for all $n \geq 1$.
Since this is a group homomorphism, it suffices to show that its kernel is trivial.

We first consider morphisms of $\cC_n$-analytic spaces instead of morphisms of $\cC_n$-analytic groups.
By the universal property of the direct product,
we have
\[ \End_{\cC_n\textbf{-AnSp}}(\cG_n) \cong \Mor_{\CC\textbf{-AnSp}}(\cG_n, \GG_{m,\CC}^\an)
\]
Since $\cG_n$ and $\cG_{n-1}$ have the same reduced space, the condition $\phi|_{\cG_{n-1}}=1$ on a holomorphic function $\phi \in \cO(\cG_n)$ forces $\phi$ to map $\cG_n$ into $\CC \setminus \{0\} = \GG_{m,\CC}^\an$.
Hence, we obtain
\[ \{ \phi \in \End_{\cC_n\textbf{-AnSp}}(\cG_n) : \phi|_{\cG_{n-1}} = 1 \} \cong \{ \phi \in \cO(\cG_n) : \phi|_{\cG_{n-1}} = 1 \}. \]

Since $\cC$ is a complex manifold, letting $r = \dim(\cC)$, we have
\[ \cO(\cG_n) \cong \cO(\GG_{m,\CC}^\an)[X_1,\dotsc,X_r]/(X_1,\dotsc,X_r)^{n+1}. \]
Consequently we may identify $\{ \phi \in \End_{\cC_n\textbf{-AnSp}}(\cG_n) : \phi|_{\cG_{n-1}} = 1 \}$ with
\[ \{ 1 + f \colon f \in \cO(\GG_{m,\CC}^\an)[X_1,\dotsc,X_r] \text{ homogeneous of degree } n  \text{ in } X_1,\dotsc,X_r \}. \]

Returning to morphisms of $\cC_n$-analytic groups, if $1+f$ corresponds to a homomorphism of $\cC_n$-groups, then
\[ 1+f(YZ;X_1,\dotsc,X_r) = (1+f(Y; X_1,\dotsc,X_r))(1+f(Z; X_1,\dotsc,X_r)) \]
in $\cO(\cG_n \times_{\cC_n} \cG_n) = \cO(\GG_{m,\CC}^\an \times \GG_{m,\CC}^\an \times \cC_n)$, where $Y$, $Z$ denote coordinates on the two copies of $\GG_m$.
Since $f$ is homogeneous of degree~$n$ in $X_1,\dotsc,X_r$, the product $f(Y;X_1,\dotsc,X_r)f(Z;X_1,\dotsc,X_r)$ is zero modulo the ideal $(X_1,\dotsc,X_r)^{n+1}$.
Hence we obtain
\[ f(YZ;X_1,\dotsc,X_r) = f(Y; X_1,\dotsc,X_r) + f(Z; X_1,\dotsc,X_r). \]
In other words, viewing $f$ as a polynomial in $X_1,\dotsc,X_r$ with coefficients in $\cO(\GG_{m,\CC}^\an)$, each coefficient is a holomorphic group homomorphism $\GG_{m,\CC}^\an \to \GG_{a,\CC}^\an$.
The only such homomorphism is the trivial one, so we deduce that $f=0$.
Hence the kernel of $\End_{\cC_n\textbf{-Grp}}(\cG_n) \to \End_{\cC_{n-1}\textbf{-Grp}}(\cG_{n-1})$ is $\{1\}$, as required.
\end{proof}

\begin{lemma} \label{exists-complex-uniformisation-compatible}
Let $\phi_\formal$ be a formal uniformisation of $G$ over~$C_\formal$.
Then there exists a unique complex analytic uniformisation of $\cG$ over~$\cC$ compatible with $\phi_\formal$.
\end{lemma}

\begin{proof}
By \cref{exists-complex-uniformisation}, there exists a complex analytic uniformisation $\phi_\an \colon (\GG_{m,\CC}^\an)^g \times \cC \to \cG$.
By composing $\phi_\an$ with a suitable element of
\[ \Aut_{\cC\textbf{-Grp}}((\GG_{m,\CC}^\an)^g \times \cC) \cong \GL_g(\ZZ) \cong \Aut_{\CC\textbf{-Grp}}((\GG_{m,\CC}^\an)^g), \]
we may assume that $(\phi_\an)_0$ is equal to the analytification of $(\phi_\formal)_0 \colon \GG_{m,\CC}^g \to G_0$ (composed with the natural identification $G^\an_0 \to \cG_0$).

Let $(\phi_\an)_\formal \colon (\GG_{m,\CC}^\an)_\formal^g \times \cC_\formal \to \cG_\formal$ denote the base change of $\phi_\an$ to the formal analytic completion $\cC_\formal$.
Let $(\phi_\formal)^\an \colon (\GG_{m,\CC}^\an)_\formal^g \times \cC_\formal \to \cG_\formal$ denote the analytification of the formal uniformisation $\phi_\formal$ (using the canonical isomorphism $\cC_\formal \cong C^\an_\formal$).
By the definition of formal uniformisation, $\phi_\formal$ is invertible, so we may compose $(\phi_\an)_\formal$ with $(\phi_\formal^{-1})^\an$ to obtain
\[ (\phi_\formal^{-1})^\an \circ (\phi_\an)_\formal : (\GG_{m,\CC}^\an)_\formal^g \times \cC_\formal \to (\GG_{m,\CC}^\an)_\formal^g \times \cC_\formal. \]
By our choice of $\phi_\an$,
the reduction $(\phi_\formal^{-1})_0 \circ (\phi_\an)_0$ is the identity endomorphism of $(\GG_{m,\CC}^\an)^g$.
By \cref{formal-analytic-torus-end}, it follows that $(\phi_\formal^{-1})^\an \circ (\phi_\an)_\formal$ is the identity endomorphism of $(\GG_{m,\CC}^\an)_\formal^g \times \cC_\formal$.
Hence, we obtain $(\phi_\an)_\formal = (\phi_\formal)^\an$, as required.

For the uniqueness of~$\phi_\an$: by our hypotheses, $C$ is a smooth algebraic curve, so $(\GG_{m,\CC}^\an)^g \times \cC$ is a complex manifold.
Furthermore, $\cC$, and hence $(\GG_{m,\CC}^\an)^g \times \cC$, are connected.
It follows, by the identity principle, that holomorphic functions on~$(\GG_{m,\CC}^\an)^g \times \cC$ are uniquely determined by their Taylor series at any single point.
Applying this at a point in~$(\GG_{m,\CC}^\an)^g \times \{s_0\}$, we conclude that the equation $(\phi_\an)_\formal = (\phi_\formal)^\an$ uniquely determines~$\phi_\an$. 
\end{proof}

\subsection{Summary: uniformisations at all places over the ring of integers of a number field} \label{subsec:uniformisations-places}

Let $R$ be a Dedekind domain whose field of fractions is a number field~$K$ (in other words, $R$ is a localisation of the ring of integers~$\cO_K$).
Let $\fC$ be an integral regular scheme of dimension~$2$ equipped with a morphism $\fC \to \Spec(R)$ which is flat and of finite type.
Let $\fC_0$ be the image of a section $\Spec(R) \to \fC$.
Note that $\fC_0$ is a closed subscheme of $\fC$ (see, for example, \cite[Exercise II.4.8]{Har75}).

Write $C = \fC \times_{\Spec(R)} \Spec(K)$, $\{s_0\} = \fC_0 \times_{\Spec(R)} \Spec(K)$, and $\fC_\formal = $ formal completion of $\fC$ along $\fC_0$.
Let $C^*$ be a Zariski open subset of $C \setminus \{s_0\}$.

Let $\fG \to \fC$ be a semiabelian scheme such that $\fG_0 \cong \GG_m^g \times \fC_0$ and $G^* \to C^*$ is an abelian scheme.

For each place $v$ of~$K$ visible to~$R$, we define a formal scheme $\fC_{v,\formal\dag}$ (together with a morphism of formal schemes $\fC_{v,\formal\dag} \to \fC_\formal$) and an open subset $\cC_v \subset C^\van$ as follows:
\begin{enumerate}
\item If $v$ is a non-archimedean place of~$K$ visible to~$R$, then let $\fp_v$ denote the ideal of $R$ associated with~$v$ and let $R_v$ denote the $v$-adic completion of~$R$.
Let $\fC_{v,\formal\dag}$ denote the fibre product of formal schemes $\fC_\formal \times_{\Spec(R)} \Spf(R_v, (\fp_v))$.
(By \cite[Prop.~10.9.7]{EGAI} with $X = \fC$, $X' = \fC_0$, $Y = \Spec(R_v)$, $Y' = \Spec(R_v/(\fp_v))$, $S=S'=\Spec(R)$, this could equivalently be described as the formal completion of $\fC_v := \fC \times_{\Spec(R)} \Spec(R_v)$ along $\fC_{v,0} := $ the fibre of $\fC_0 \to \Spec(R)$ above~$\fp_v$.) Let $\cC_v = \fC_{v,\formal\dag}^\rig$.

\item If $v$ is an archimedean place of~$K$, then let $\fC_{v,\formal\dag}$ denote the fibre product of formal schemes $\fC_\formal \times_{\Spec(R)} \Spec(\CC)$, via the morphism $\Spec(R) \to \Spec(\CC)$ associated with the place~$v$. Choose a simply connected neighbourhood $\cC_v$ of $s_0^\van$ in the complex manifold $C^\van$, such that $\cC_v \setminus \{s_0^\van\} \subset (C^*)^\van$.
\end{enumerate}
Let $G_{v,\formal\dag} = \fG_\formal \times_{\fC_\formal} \fC_{v,\formal\dag}$ and let $\cG_v = G^\van|_{\cC_v}$.

The following proposition follows from \cref{exists-rigid-uniformisation-compatible,exists-complex-uniformisation-compatible}.

\begin{proposition} \label{exist-uniformisations-places}
Let $\phi_\formal$ be a formal uniformisation of $\fG_\formal$ over~$\fC_\formal$.
For each place $v$ of~$K$ visible to~$R$, let $\phi_{v,\formal\dag} \colon \GG_m^g \times \fC_{v,\formal\dag} \to \fG_{v,\formal\dag}$ denote the base change of $\phi_\formal$ along $\fC_{v,\formal\dag} \to \fC_\formal$.
There exists a unique $v$-adic analytic uniformisation $\phi_v \colon (\GG_m^\van)^g \times \cC_v \to \cG_v$ which is compatible with $\phi_{v,\formal\dag}$.
\end{proposition}

\section{Construction of period G-functions} \label{sec:formal-periods}

Let $C$ be a smooth curve over a number field $K$, and let $G \to C$ be a semiabelian scheme such that $G|_{C^*}$ is an abelian scheme for some non-empty open subset $C^* \subset C$, and $G_{s_0} \cong \GG_{m,K}^g$ for some $s_0 \in C(K) \setminus C^*(K)$.
Andr\'e observed that, for the complex analytification of such a semiabelian scheme, the ``locally invariant'' periods near~$s_0$ are given by evaluations of G-functions (which are power series with coefficients in~$K$).
Furthermore, the periods with respect to different archimedean places of~$K$ are given by evaluating the same power series at the respective places of~$K$.
This observation was key to Andr\'e's strategy for proving height bounds for fibres of such a semiabelian scheme with large endomorphism rings.

In this section, we interpret geometrically the non-archimedean as well as the archimedean evaluations of these period G-functions, using formal and rigid uniformisations.
By using compatibility between formal uniformisation over the ring of integers of~$K$ and analytic uniformisations at all places, we ensure that our ``periods'' at all places come from evaluations of the same power series in $\powerseries{K}{X}$.

In order to discuss formal uniformisation, we require an integral model of the semiabelian scheme $G \to C$.
It is not clear, however, what integrality condition we could impose on de Rham classes directly, because de Rham cohomology does not behave well on formal integral models.
Instead, we choose invariant differential forms $\omega_1, \dotsc, \omega_m$ extending over the integral model, and then we consider additional de Rham classes $\eta_1, \dotsc, \eta_n$ belonging to the span of Gauss--Manin derivatives of $\omega_1, \dotsc, \omega_m$.
This allows us to repeatedly use the following strategy:
\begin{enumerate}
\item Use the integral model $\fG \to \fC$ to construct some object or prove some property associated with invariant differential forms $\omega_1, \dotsc, \omega_m$.
\item Differentiate using the Gauss--Manin connection to construct a corresponding object or prove a corresponding property associated with de Rham classes $\eta_1, \dotsc, \eta_n$.
\end{enumerate}
Thanks to \cref{gauss-manin-basis}, in the situation in which we apply \cref{exist-formal-periods} in this article, it is possible to choose $\eta_1, \dotsc, \eta_n$ so that $\omega_1, \dotsc, \omega_m, \eta_1, \dotsc, \eta_n$ span $H^1_{DR}(G|_{C^*}/C^*)$.

\subsection{Statement of theorem on period G-functions} \label{subsec:formal-periods-statement}

Let $R$ be a Dedekind domain whose field of fractions is a number field~$K$.
Let $\fC$ be an integral regular scheme of dimension~$2$ equipped with a morphism $\fC \to \Spec(R)$ which is flat and of finite type.

Let $\fC_0$ denote the image of a section $\Spec(R) \to \fC$, which is a closed subscheme of $\fC$.
Let $\fC_\formal$ denote the formal completion of $\fC$ along $\fC_0$.
Write $C = \fC \times_{\Spec(R)} \Spec(K)$ and $\{s_0\} = \fC_0 \cap C$.
Let $C^*$ be a Zariski open subset of $C \setminus \{s_0\}$.
Let $C' = C^* \cup \{s_0\}$, which is also a Zariski open subset of~$C$.

Let $\pi \colon \fG \to \fC$ be a semiabelian scheme such that $\fG_0 \cong \GG_{m,\fC_0}^g$.
Write $G := \fG \times_\fC C$ and $G^* := G \times_C C^*$ and suppose that $G^* \to C^*$ is an abelian scheme.
Let $\tilde\fG = \GG_m^g \times \fC$ and define $\tilde G$ and $\tilde G^*$ analogously.

Let $\nabla$ denote the Gauss--Manin connection on the relative de Rham cohomology $H^1_{DR}(G^*/C^*)$.
We write $H^1_{DR}(G^*/C^*)^\can$ for the canonical extension of $(H^1_{DR}(G^*/C^*), \nabla)$ over $C' = C^* \cup \{s_0\}$.
As in Section \ref{subsec:derivatives-forms}, we define $\Omega^\inv_{G/C}$ to be the $\cO_C$-module of $G$-invariant forms in $\pi_*\Omega^1_{G/C}$.
Since the canonical extension can be realised as the hypercohomology of the log de Rham complex, there is a morphism of $\cO_{C'}$-modules 
\[ \iota \colon \Omega_{G/C}^\inv|_{C'} \to H^1_{DR}(G^*/C^*)^\can. \]
Since $G^* \to C^*$ is an abelian scheme, $\iota|_{C^*}$ is injective.  Since $\Omega_{G/C}^\inv$ is locally free, it follows that $\iota$ is injective.
We henceforth identify $\Omega_{G/C}^\inv|_{C'}$ with its image under~$\iota$.

\begin{theorem} \label{exist-formal-periods}
Let $x \in \cO(\fC)$ be a function such that $x|_C$ is a local parameter around~$s_0$.

Let $\omega_1, \dotsc, \omega_m \in \Omega^{\inv}_{G/C}(C')$.
Let $\eta_1, \dotsc, \eta_n$ be elements of the $\cO(C')$-submodule of $H^1_{DR}(G^*/C^*)^\can(C')$ generated by $\{ (\nabla_{d/dx})^k \omega_\ell : k \in \ZZ_{\geq 0}, 1 \leq \ell \leq m \}$.

Let $\tilde\omega_1, \dotsc, \tilde\omega_g$ denote the standard $\cO_\fC$-basis for $\Omega^{\inv}_{\GG_m^g \times \fC/\fC}$ (that is, $\tilde\omega_i = du_i/u_i$, where $u_i$ is the coordinate on the $i$-th copy of $\GG_m$).

Let $\phi_\formal \colon \tilde\fG_\formal \to \fG_\formal$ be a formal uniformisation of $\fG_\formal$ over~$\fC_\formal$.
For each place $v$ of~$K$ visible to~$R$, let $\cC_v \subset C^\van$ denote the open set called $\cC_v$ in~\cref{subsec:uniformisations-places}.

Let $\cG_v = G^\van|_{\cC_v}$ and $\tilde\cG_v = (\GG_m^\van)^g \times \cC_v$.
Let $\phi_v \colon \tilde\cG_v \to \cG_v$ denote the $v$-adic uniformisation of $\cG_v$ over $\cC_v$ compatible with~$\phi_\formal$ given by \cref{exist-uniformisations-places}.

Then there exist G-functions $F_{ij} \in \powerseries{K}{X}$ (for $1 \leq i \leq m, 1 \leq j \leq g$) and $G_{ij} \in \powerseries{K}{X}$ (for $1 \leq i \leq n, 1 \leq j \leq g$),
open sets $U_v \subset C'^\van$ and real numbers $r_v$ (for each place~$v$ of $K$ visible to~$R$) with the following properties:
\begin{enumerate}
\item $r_v>0$ for all $v$, and $r_v=1$ for almost all~$v$;
\item $s_0^\van \in U_v \subset \cC_v$ for all~$v$;
\item $x^{\van}$ maps $U_v$ isomorphically (in the category of $K_v$-analytic spaces) onto the open disc $D(r_v,K_v)$;
\item the $v$-adic radii of convergence of $F_{ij}$ and $G_{ij}$ are at least $r_v$ for all $v$, $i$ and~$j$;
\item for each place $v$ of~$K$ visible to~$R$, for every finite extension of complete fields $L/K_v$ and for all $s \in U_v^*(L)$, where $U^*_v:=U_v\setminus\{s^{\van}_0\}$, we have
\begin{align*}
   \phi_{v,s}^*(\omega_{i,s}^\van) & = \sum_{j=1}^g F_{ij}^{\van}(x^\van(s)) \tilde\omega_{j,s}^\van \text{ for } 1 \leq i \leq m,
\\ \phi_{v,s}^*(\eta_{i,s}^\van)   & = \sum_{j=1}^g G_{ij}^{\van}(x^\van(s)) \tilde\omega_{j,s}^\van \text{ for } 1 \leq i \leq n,
\end{align*}
in the analytic de Rham cohomology $H^1_{DR}((\GG_m^\van)^g / L)$.
\end{enumerate}
\end{theorem}

For the definition of G-functions used in this article, see \cite[p. 1]{And89}.

\subsection{Simplifying the structure of the formal completion} \label{subsec:Cfor-reduction}

We begin by showing that, in order to prove \cref{exist-formal-periods}, it suffices to prove it subject to additional conditions on the structure of $\fC_\formal$.
The purpose of these conditions is so that the restriction to $\fC_\formal$ of a function on~$\fC$ is precisely the Taylor series of that function.
We impose these additional conditions in two stages.

\subsubsection{Reduction to abstract power series ring}

We first reduce to the case where $\fC_\formal$ is $\Spf$ of a power series ring.

\begin{lemma} \label{exist-formal-periods-reduction-1}
In order to prove \cref{exist-formal-periods}, it suffices to prove it under the following additional conditions:
\begin{enumerate}[label=(\roman*)]
\item $\fC$ is affine.
\item $\fC_\formal$ is isomorphic (as a formal scheme over $\Spec(R)$) to $\Spf(\powerseries{R}{T}, (T))$, via an isomorphism under which the indeterminate $T$ corresponds to the formal completion of a regular function on $\fC$.
\end{enumerate}
\end{lemma}

In order to prove \cref{exist-formal-periods-reduction-1}, we will use the following lemma to replace $\fC$ with suitable open subsets. Its proof is inspired by nfdc23's MathOverflow comment~\cite{MO:formal-comp}.

\begin{lemma}\label{lem:opens}
Let $K$ be a number field and let $R$ be a Dedekind domain with $\Frac(R)=K$. Let $f:\fC\to\Spec(R)$ be a flat separated morphism locally of finite type with $\fC$ integral of dimension $2$, and let $j:\Spec (R)\to\fC$ denote a section of $f$ such that, for every closed point $\mathfrak{p}\in\Spec(R)$, the fibre $\fC_{\mathfrak{p}}$ is smooth at $j(\mathfrak{p})$.
Let $\fC_0 \subset \fC$ denote the image of~$j$.
Then there exist finite collections of affine open subschemes $\fC_\tau$ of $\fC$ and of elements $\alpha_\tau \in R$ such that:
\begin{enumerate}[(a)]
\item the open subschemes $\Spec(R[1/\alpha_\tau])$ cover $\Spec(R)$;
\item for each~$\tau$, $f(\fC_\tau) \subset \Spec(R_\tau)$;
\item for each~$\tau$, condition~(ii) of \cref{exist-formal-periods-reduction-1} holds when $\fC$ is replaced by $\fC_\tau$ and $R$ is replaced by $R_\tau = R[1/\alpha_\tau]$ for some $\alpha_\tau \in R$. 
\end{enumerate}
\end{lemma}

\begin{proof}
It suffices to show that, for each closed point $\mathfrak{p}\in\Spec(R)$, there exists an $\alpha \in R$ such that $\fp \in \Spec(R[1/\alpha])$ and an affine open subscheme $U \subset \fC$ such that $(U, \alpha)$ have the properties required of $(\fC_\tau, \alpha_\tau)$.

By \cite[Lemma 29.34.14]{stacks}, $f$ is smooth at $j(\mathfrak{p})$. Hence, by \cite[Cor.~17.12.3]{EGAIV}, $j$ is a quasi-regular immersion in an open neighbourhood $V' \subset \Spec(R)$ of~$\mathfrak{p}$.

Let $U'$ denote an affine open neighbourhood of $j(\fp)$ in~$\fC$.
Let $V = V' \cap j^{-1}(U)$, which is an open neighbourhood of $\fp$ in $\Spec(R)$.
Shrinking $V$, we may assume that $V = \Spec(R')$ where $R' = R[1/\alpha]$ for some $\alpha \in R$, still with $\fp \in V$.

Let $U = U' \cap f^{-1}(V) = U' \times_{\Spec(R)} V$.
Since $U'$, $V$ and $\Spec(R)$ are all affine schemes, so is $U$.
By construction, $f(U) \subset V$ and $j(V) = \fC_0 \cap U$.

By \cite[Cor.~16.9.9]{EGAIV}, since $j|_V \colon V \to U$ is a quasi-regular immersion, the formal completion $U_{\formal}$ of $U$ along $\fC_0\cap U$ is isomorphic to $\Spf(\powerseries{R'}{T},(T))$. (See \cite[Remark 16.1.11]{EGAIV} for the connection with the formal completion.) 

Let $I$ denote the ideal of $\cO:=\cO_{\fC}(U)$ corresponding to the reduced closed subscheme $\fC_0\cap U \subset U$. Then $U_{\formal}=\Spf(\hat{\cO},I\hat{\cO})$, where $\hat{\cO}$ denotes the $I$-adic completion of $\cO$. The isomorphism $\Spf(\powerseries{R'}{T},(T)) \to \fC_0\cap U$ obtained above induces an isomorphism $\beta:\hat{\cO}\to \powerseries{R'}{T}$, which is compatible with the homomorphisms $\hat{\cO}\to R'$ (induced by $j$) and $\powerseries{R'}{T}\to R'$ (induced by $T\mapsto 0$). Thus, we obtain an isomorphism
\(\cO/I^2\cong\hat{\cO}/(I\hat{\cO})^2\to \powerseries{R'}{T}/(T^2) \)
and so the composition
\[\cO\to\hat{\cO}\to \powerseries{R'}{T}\to \powerseries{R'}{T}/(T^2)\]
is surjective. In particular, there exists $x\in\cO$ such that $\beta(x)\equiv T\text{ mod }(T^2)$. As a result, there exists a $T$-adically continuous automorphism $\sigma$ of $\powerseries{R'}{T}$ sending $\beta(x)$ to $T$. Therefore, $\sigma\circ\beta$ induces the desired isomorphism of formal schemes. 
\end{proof}

Now we prove \cref{exist-formal-periods-reduction-1}.

Indeed, suppose that we are in the situation of \cref{exist-formal-periods}.
Applying \cref{lem:opens} to $\fC$ and $\fC_0$, we obtain rings $R_\tau$ and open subsets $\fC_\tau \subset \fC$ which satisfy conditions (i) and~(ii) from \cref{exist-formal-periods-reduction-1}.

Assume that \cref{exist-formal-periods} holds under the additional conditions (i) and~(ii) of \cref{exist-formal-periods-reduction-1}.
Thus, we are assuming that for each $\tau$, \cref{exist-formal-periods} holds with $R$ replaced by~$R_\tau$, $\fC$ by $\fC_\tau$, and all the other objects in the input of \cref{exist-formal-periods}, namely $C^*$, $\fG$, $x$, $\omega_1, \dotsc, \omega_m, \eta_1, \dotsc, \eta_n$, and $\phi_\formal$, base-changed from $\fC$ to $\fC_\tau$.
This gives us, for each~$\tau$, power series $F_{\tau ij}, G_{\tau ij} \in \powerseries{K}{X}$, open sets $U_{\tau,v} \subset C_\tau'^\van \subset C'^\van$ and real numbers $r_{\tau,v} > 0$ for each place of $K$ visible to~$R_\tau$, satisfying the conclusions of \cref{exist-formal-periods} at places visible to~$R_\tau$.

The key claim that we need to deduce \cref{exist-formal-periods} for $\fC$ itself is the following:

\begin{lemma}
For each $i, j$, the power series $F_{\tau ij}, G_{\tau ij} \in \powerseries{K}{X}$ are independent of~$\tau$.
\end{lemma}

\begin{proof}
We shall show that $F_{1ij} = F_{2ij}$ and $G_{1ij} = G_{2ij}$.  The same argument applies to all other $\tau = 1, \dotsc, m$.

Choose a place $v$ of~$K$ which is visible to both $R_1$ and~$R_2$ (for example, any archimedean place).
Assume without loss of generality that $r_{1,v} \leq r_{2,v}$ or, equivalently, $U_{1,v} \subset U_{2,v}$.

Then, for all $s \in U_{1,v}^*(L)$, for any finite extension $L/K_v$, we can apply \cref{exist-formal-periods}(5) for both~$\fC_1$ and $\fC_2$, obtaining
\begin{align*}
   \phi_{v,s}^*(\omega_{i,s}^\van) & = \sum_{j=1}^g F_{1ij}^{\van}(x^\van(s)) \tilde\omega_{j,s}^\van
\\ & = \sum_{j=1}^g F_{2ij}^{\van}(x^\van(s)) \tilde\omega_{j,s}^\van
\end{align*}
in $H^1_{DR}((\GG_m^\van)^g/L)$, for $1 \leq i \leq g$.

Since $\tilde\omega_{1,s}^\van, \dotsc, \tilde\omega_{g,s}^\van$ form an $L$-basis of $H^1_{DR}((\GG_m^\van)^g/L)$, it follows that
\[ F_{1ij}^{\van}(x^\van(s)) =  F_{2ij}^{\van}(x^\van(s)) \]
for all $s \in U_{1,v}^*(\ov K_v)$, or in other words,
\[ F_{1ij}^\van(\xi) = F_{2ij}^\van(\xi) \]
for all $\xi \in D^*(0, r_{1,v}, K_v)$.

Consequently $F_{1ij} = F_{2ij}$ as power series in $\powerseries{K}{X}$.

A similar argument, using the second equation from \cref{exist-formal-periods}(5), establishes that $G_{1ij} = G_{2ij}$.
\end{proof}

Now we can simply choose $F_{ij} = F_{\tau ij}$, $G_{ij} = G_{\tau ij}$ for any~$\tau$.

Thanks to conclusion~(a) of \cref{lem:opens}, every place of $K$ visible to~$R$ is visible to at least one~$R_\tau$.
So, if $v$ is a place of~$K$ visible to~$R$, we can pick some~$\tau$ such that $v$ is visible to~$R_\tau$ and define $U_v = U_{\tau,v}$ and $r_v = r_{\tau,v}$.
Then properties (2)--(5) in \cref{exist-formal-periods} for $\fC$ for the place~$v$ follow from the analogous statements for~$\fC_\tau$ for the place~$v$.

It only remains to check the claim in \cref{exist-formal-periods}(1) that $r_v = 1$ for almost all~$v$.
For each $\tau$, there are only finitely many places visible to~$R_\tau$ for which $r_{\tau,v} \neq 1$.
Since there are only finitely many indices~$\tau$, it follows that the set
\[ \{ v \text{ visible to } R : r_{\tau,v} \neq 1 \text{ for some } \tau \text{ such that } v \text{ is visible to } R_\tau \} \]
is finite.
If $v$ is not in this set, then $r_v=1$, as required.

\subsubsection{Reduction to power series ring generated by \texorpdfstring{$x$}{x}}

We strengthen \cref{exist-formal-periods-reduction-1} by showing that we may additionally assume that the indeterminate of the power series ring corresponds to (the formal completion of) our given local parameter~$x$ on~$\fC$.
Note that this additional assumption in \cref{exist-formal-periods-reduction-2}(ii') is the only difference between \cref{exist-formal-periods-reduction-2} and \cref{exist-formal-periods-reduction-1} (we have relabelled the indeterminate of the power series ring from $T$ to~$X$ because it is this indeterminate that gives rise to the ``$X$'' appearing \cref{exist-formal-periods}).

\begin{lemma} \label{exist-formal-periods-reduction-2}
In order to prove \cref{exist-formal-periods}, it suffices to prove it under the following additional conditions:
\begin{enumerate}[label=(\roman*)]
\item $\fC$ is affine.
\item[(ii')] $\fC_\formal$ is isomorphic (as a formal scheme over $\Spec(R)$) to $\Spf(\powerseries{R}{X}, (X))$, via an isomorphism under which the indeterminate $X$ corresponds to the formal completion of the local parameter $x \in \cO(\fC)$ given in \cref{exist-formal-periods}.
\end{enumerate}
\end{lemma}

\begin{proof}
Suppose that we are in the setting of \cref{exist-formal-periods}, under the additional conditions (i) and~(ii) of \cref{exist-formal-periods-reduction-1}.

Let $t$ be the regular function on~$\fC$ which corresponds to the indeterminate~$T$ via the isomorphism $\cO(\fC_\formal) \to \powerseries{R}{T}$ given by condition~(ii) of \cref{exist-formal-periods-reduction-1}.
Note that $t|_C$ is a local parameter around~$s_0$, because $t$ induces an isomorphism $C_\formal \to \Spf(\powerseries{K}{T}, (T))$.

We assume that \cref{exist-formal-periods} holds under the additional conditions (i) and~(ii') of \cref{exist-formal-periods-reduction-2}.
Hence we can apply \cref{exist-formal-periods}, to the same objects $\fC$, $C^*$, $\fG$, $\omega_1, \dotsc, \omega_m$, $\eta_1, \dotsc, \eta_n$, $\phi_\formal$ that we already have, but with local parameter~$t$ instead of~$x$ (since $t$, $T$ are related as in condition~(ii') by construction).
Thus we obtain G-functions $F_{ij}^\square, G_{ij}^\square \in \powerseries{K}{T}$, open sets $U_v^\square \subset C^\van$ and real numbers~$r_v^\square$ satisfying \cref{exist-formal-periods}(1)--(5) with $x$ replaced by~$t$.

Let $\lambda \in \powerseries{R}{T}$ denote the image of $x \in \cO(\fC) \subset \cO(\fC_\formal)$ under the isomorphism $\cO(\fC_\formal) \to \powerseries{R}{T}$ (in other words, $\lambda$ is the Taylor series of $x|_C$ around $s_0$ in terms of~$t|_C$).
Since $x$ and~$t$ are both local parameters around~$s_0$, we have $\lambda(0)=0$ and $\lambda'(0) \neq 0$.
Hence we can apply \cref{power-series-locally-invertible} to~$\lambda$, obtaining real numbers $r_v^*(\lambda) > 0$ and $K_v$-analytic open sets $Y_v \subset K_v$, for each place $v$ of~$K$ visible to~$R$.

For each place visible to~$R$, $Y_v \cap D(r_v^\square, K_v)$ is an open subset of $Y_v$, so is mapped by $\lambda^\van$ isomorphically onto an open neighbourhood of~$0$ in $D(r_v^*(\lambda), K_v)$, thanks to \cref{power-series-locally-invertible}(c).
We can therefore choose an open neighbourbood $Z_v$ of~$0$ in $Y_v \cap D(r_v^\square, K_v)$ such that $\lambda^\van$ maps $Z_v$ isomorphically onto an open disc $D(r_v, K_v)$, for some $r_v$ with $0 < r_v \leq r_v^*(\lambda)$.

Furthermore, by \cref{power-series-locally-invertible}(d) and \cref{exist-formal-periods}(1) for $r_v^\square$, for almost all places~$v$, we have $Y_v = D(1, K_v) = D(r_v^\square, K_v)$.
For such places, we can choose $Z_v = Y_v$ and $r_v = r_v^*(\lambda)$.
Hence by \cref{power-series-locally-invertible}(a), we obtain $r_v \geq 1$ for almost all places~$v$.
Thus these values $r_v$ satisfy conclusion~(1) of \cref{exist-formal-periods}.

Let $U_v = U_v^\square \cap (t^\van)^{-1}(Z_v)$.

We can represent this in the following diagram, where all horizontal inclusions are inclusions of open sets and (thanks to the above discussion and to \cref{exist-formal-periods}(3) for~$U_v^\square$) all vertical arrows are isomorphisms of $K_v$-analytic spaces:
\[ \begin{tikzcd}
  & U_v                  \ar[d, "t^\van"]       \ar[r, phantom, "\subset"]
  & U_v^\square          \ar[d, "t^\van"]
\\  Y_v                  \ar[d, "\lambda^\van"] \ar[r, phantom, "\supset"]
  & Z_v                  \ar[d, "\lambda^\van"] \ar[r, phantom, "\subset"]
  & D(r_v^\square, K_v)
\\  D(r_v^*(\lambda), K_v)                    \ar[r, phantom, "\supset"]
  & D(r_v, K_v)
  &
\end{tikzcd} \]
Now $x^\van|_{U_v}$ and $\lambda^\van \circ t^\van|_{U_v}$ are $K_v$-analytic functions on an irreducible $K_v$-analytic space with the same Taylor series at~$s_0^\van$, so they must be equal.
Thus $x^\van$ restricts to an isomorphism $U_v \to D(r_v, K_v)$, giving \cref{exist-formal-periods}(3) for~$U_v$.

Since $0 \in Z_v$ and $U_v \subset U_v^\square$, conclusion~(2) of \cref{exist-formal-periods} for $U_v$ follows from the analogous conclusion for $U_v^\square$.

Since $\lambda(0) = 0$ and $\lambda'(0) \neq 0$, the power series $\lambda$ possesses a compositional inverse $\mu \in \powerseries{K}{X}$.
Since $\mu(0)=0$, we may algebraically compose power series with $\mu$ on the right, allowing us to define
\[ F_{ij} = F_{ij}^\square \circ \mu,
  \quad  G_{ij} = G_{ij}^\square \circ \mu \in \powerseries{K}{X}. \]

For each place $v$ visible to~$R$, $(\lambda^\van)^{-1} \colon D(r_v, K_v) \to Z_v$ is a $K_v$-analytic function, whose Taylor series is $\mu$.
Since $Z_v \subset D(r_v^\square, K_v)$, and $r_v^\square \leq R_v(F_{ij}^\square)$ by \cref{exist-formal-periods}(4) for $F_{ij}^\square$, the $K_v$-analytic function $F_{ij}^{\square\van} \circ (\lambda^\van)^{-1}$ is defined on the disc $D(r_v, K_v)$.
Since the Taylor series of $F_{ij}^{\square\van} \circ (\lambda^\van)^{-1}$ is $F_{ij}$, we conclude that $R_v(F_{ij}) \geq r_v$, establishing conclusion~(4) of \cref{exist-formal-periods} for $F_{ij}$.
The same argument establishes \cref{exist-formal-periods}(4) for~$G_{ij}$.

Finally, conclusion~(5) of \cref{exist-formal-periods} (for $x$, $U_v$, $F_{ij}$) follows from conclusion~(5) of \cref{exist-formal-periods} (for $t$, $U_v^\square$, $F_{ij}^\square$) by the following calculation, valid for places $v$ visible to~$R$, all finite extensions $L/K_v$ and all $s \in U_v^*(L)$ and $1 \leq i \leq m$:
\begin{align*}
    \phi_{v,s}^* (\omega_{i,s}^\van)
    = \sum_{j=1}^g F_{ij}^{\square\van}(t^\van(s)) \tilde\omega_{j,s}^\van
  & = \sum_{j=1}^g F_{ij}^{\square\van}(\mu^\van(x^\van(s))) \tilde\omega_{j,s}^\van
\\& = \sum_{j=1}^g F_{ij}^{\van}(x^\van(s)) \tilde\omega_{j,s}^\van
\end{align*}
A similar calculation works for~$G_{ij}$.
\end{proof}

\subsection{Construction of the power series}\label{sec:construction-of-power-series}

We show that it suffices to prove \cref{exist-formal-periods} under further additional assumptions, which will be used in the non-archimedean interpretation of the period G-functions.

\begin{lemma} \label{exist-formal-periods-reduction-3}
In order to prove \cref{exist-formal-periods}, it suffices to prove it under the additional assumptions (i) and~(ii') of \cref{exist-formal-periods-reduction-2}, as well as the following further assumptions:
\begin{enumerate}
\item[(iii)] The differential forms $\omega_1, \dotsc, \omega_m \in \Omega^\inv_{G/C}(C')$ are restrictions of global sections (over~$\fC$) of $\Omega^\inv_{\fG/\fC}$.
\item[(iv)] $C^*$ is equal to the Artin set $A_1(\pi|_{G^*})$, as defined in \cite[Def.~25.2.1]{ABC}, where $\pi \colon \fG \to \fC$ is the structure morphism of the group scheme.
\end{enumerate}
\end{lemma}

\begin{proof}
Suppose that we are in the setting of \cref{exist-formal-periods}, under the additional conditions (i) and~(ii') of \cref{exist-formal-periods-reduction-2}.

By assumption~(i) of \cref{exist-formal-periods-reduction-2}, $\fC$ is affine.
Hence we can choose a regular function $f \in \cO(C)$ which has zeroes at all of the finitely many points in $C \setminus C'$, and which is non-zero at $s_0$.
Since $\omega_1, \dotsc, \omega_m$ can have poles only on $C \setminus C'$, for large enough~$N$, $f^N\omega_1, \dotsc, f^N\omega_m$ are restrictions of global sections (over~$C$) of $\Omega^\inv_{G/C}$.

Again since $\fC$ is affine, the sheaf of $\cO_\fC$-modules $\Omega^{\inv}_{\fG/\fC}$ is induced from its $\cO(\fC)$-module of global sections.
Thus
\[ \Omega^{\inv}_{G/C}(C) = \Omega^{\inv}_{\fG/\fC}(\fC) \otimes_{\cO(\fC)} \cO(C) = \Omega^{\inv}_{\fG/\fC}(\fC) \otimes_R K. \]
The last equality holds because $C = \fC \times_{\Spec(R)} \Spec(K)$, and all these schemes are affine, so $\cO(C) = \cO(\fC) \otimes_R K$.
Consequently, after multiplying $f$ by a non-zero element of~$R$, we may assume that $f^N\omega_1, \dotsc, f^N\omega_m$ are restrictions of global sections (over~$\fC$) of $\Omega^\inv_{\fG/\fC}$.

Let $C^{**}$ denote the complement of the zeroes of~$f$ in~$C^*$, and let $C'' = C^{**} \cup \{s_0\}$.
Since we chose $f$ so that $f(s_0) \neq 0$, we have $f \in \cO(C'')^\times$.
In particular, $\eta_1, \dotsc, \eta_n$ are in the $\cO(C'')$-submodule of $H^1_{DR}(G^*/C^*)^\can(C'')$ spanned by 
\[\{\nabla_{d/dx}^kf^N\omega_\ell:k\in\bZ_{\geq 0},1\leq \ell\leq m\}.\]

Assuming that \cref{exist-formal-periods} holds under the additional assumptions (i), (ii') and~(iii), it holds with $\omega_1, \dotsc, \omega_m$ replaced by $f^N\omega_1, \dotsc, f^N\omega_m$ and $C^*$ replaced by $C^{**}$.
In order to establish that \cref{exist-formal-periods} holds with the original $\omega_1, \dotsc, \omega_m$ over~$C^{**}$, we simply need to multiply the resulting power series $F_{ij}$ by $f^{-N}$, which is regular on $C''$ and hence on all $U_v$.

Finally, we observe that, if \cref{exist-formal-periods} holds with $C^*$ replaced by the Zariski dense open subset~$C^{**}$, then it holds for the original~$C^*$ with the same sets~$U_v \subset C''^\van \subset C'^\van$.

This last observation also permits us to assume that (iv) is satisfied, since $A_1(\pi|_{G^*})$ is Zariski dense open in~$C^*$ by \cite[Cor.~25.2.2]{ABC}.
\end{proof}

In the remainder of \cref{sec:formal-periods}, we prove \cref{exist-formal-periods}, under the additional assumptions (i) and~(ii') of \cref{exist-formal-periods-reduction-2} and assumptions (iii) and~(iv) of \cref{exist-formal-periods-reduction-3}. Moreover, we continue to use the notation $\omega_1, \dotsc, \omega_m$ for the differential forms in~$\Omega^\inv_{\fG/\fC}$ given by assumption~(iii).

Since $\phi_\formal^*(\omega_{1,\formal}), \dotsc, \phi_\formal^*(\omega_{m,\formal})$ are sections of $\Omega^{\inv}_{\GG_m^g \times \fC_\formal/\fC_\formal}$, we can write them as linear combinations of the basis $\tilde\omega_{1,\formal}, \dotsc, \tilde\omega_{g,\formal}$:
\begin{equation} \label{eqn:Fdiamond}
\phi_\formal^*(\omega_{i,\formal}) = \sum_{j=1}^g F_{ij}^\diamond \tilde\omega_{j,\formal} \text{ for some } F_{ij}^\diamond \in \cO(\fC_\formal).
\end{equation}
The isomorphism $\fC_\formal \cong \Spf(\powerseries{R}{X})$ from condition~(ii') of \cref{exist-formal-periods-reduction-2} induces a ring isomorphism $\cO(\fC_\formal) \to \powerseries{R}{X}$ and we let $F_{ij} \in \powerseries{R}{X}$ denote the image of $F_{ij}^\diamond$ under this isomorphism.

By the hypotheses of \cref{exist-formal-periods}, we can write
\begin{equation} \label{eqn:gauss-manin-aikl}
\eta_i = \sum_{k=0}^N \sum_{\ell=1}^m a_{ik\ell}^\diamond (\nabla_{d/dx})^k \omega_\ell,
\end{equation}
for each $i=1,\dotsc,n$,
for some $N \in \ZZ_{\geq0}$ and some $a_{ik\ell}^\diamond \in \cO(C')$.
 Let $a_{ik\ell} \in \powerseries{K}{X}$ be the Taylor series of $a_{ik\ell}^\diamond$ in terms of~$x$.
For each $i=1, \dotsc, r$, $j = 1, \dotsc, g$, define
\begin{equation} \label{eqn:Gij-definition}
G_{ij} = \sum_{k=0}^N \sum_{\ell=1}^m a_{ik\ell} \frac{d^k}{dX^k} F_{ij} \in \powerseries{K}{X}.
\end{equation}

These are the power series $F_{ij}$, $G_{ij}$ which appear in the conclusion of \cref{exist-formal-periods}.
In the following subsections, we shall construct open sets $U_v$ and real numbers $r_v$, and prove that they satisfy conclusions (1)--(5) of \cref{exist-formal-periods}.

\subsection{Non-archimedean interpretation of period G-functions}

Let $v$ be non-archimedean place of~$K$ visible to~$R$.
Let
\begin{multline*}
r_v = \min \bigl(\{1\}
         \cup \{ R_v(a_{ik\ell}) : 1 \leq i \leq g, 0 \leq k \leq N, 1 \leq \ell \leq m \}
\\ \cup \{ \abs{x(s)}_v : s \in C(\ov K) \setminus C'(\ov K), x(s) \neq 0 \} \big).
\end{multline*}
Note that this is the minimum element of a finite set, so it is well-defined.
Since the functions $a_{ik\ell}^\diamond$ are algebraic over $K(x)$, the power series $a_{ik\ell}$ are globally bounded by a theorem of Eisenstein (see \cite[\S84]{Die57}).
Hence each $R_v(a_{ik\ell})$ is $> 0$ for all~$v$ and $\geq 1$ for almost all~$v$.
Likewise each $\abs{x(s)}_v$ is $\geq 1$ for almost all~$v$, while we have explicitly excluded those $s$ for which $x(s) = 0$.
Hence $r_v > 0$ for all~$v$ and $r_v = 1$ for almost all~$v$, establishing \cref{exist-formal-periods} for non-archimedean places.

As in~\cref{subsec:uniformisations-places}(1), let $\fp_v \in \Spec(R)$ denote the closed point corresponding to the place~$v$, let $\fC_{v,\formal\dag}$ denote the formal completion of $\fC_v$ along $\fC_{v,0} = \fC_0 \times_{\Spec(R)} \{\fp_v\}$.
We could also describe $\fC_{v,\formal\dag}$ as the fibre product of formal schemes $\fC_\formal \times_{\Spec(R)} \Spf(R_v, (\fp_v))$.
Let $\cC_v = \fC_{v,\formal\dag}^\rig \subset C^\van$.

Under the isomorphism $\fC_\formal \cong \Spf(\powerseries{R}{X}, (X))$, the closed subscheme $\fC_{v,0}$ corresponds to the ideal $(\fp_v, X)$.
The $(\fp_v, X)$-adic completion of the ring $\powerseries{R}{X}$ is $\powerseries{R_v}{X}$.
Therefore, $\fC_{v,\formal\dag}$ is isomorphic to $\Spf(\powerseries{R_v}{X}, (\fp_v, X))$, so $x^\van$ induces an isomorphism of $K_v$-rigid spaces $\cC_v \to D(1,K_v)$.
Thus,
\[ U_v := \cC_v \cap (x^\van)^{-1}(D(r_v,K_v)) \]
satisfies \cref{exist-formal-periods}(2) and~(3).

Furthermore, we claim that $U_v \subset C'^\van$.
Indeed, if $s \in C(\ov K) \setminus C'(\ov K)$ and $x(s) \neq 0$, then $r_v \leq \abs{x(s)}_v$ so $s^\van \not\in (x^\van)^{-1}(D(r_v,K_v)$.
On the other hand, if $s \in C(\ov K) \setminus C'(\ov K)$ and $x(s) = 0$, then $s^\van \not\in \cC_v$ because $x^\van|_{\cC_v}$ is injective and $s_0^\van \in \cC_v$ with $x(s_0) = 0$.
Thus, for every $s \in C(\ov K) \setminus C'(\ov K)$, we have $s^\van \not\in U_v$, so that $U_v \subset C'^\van$.

Since $F_{ij} \in \powerseries{R}{X}$, $R_v(F_{ij}) \geq 1 \geq r_v$ for all non-archimedean places $v$ visible to~$R$.
Furthermore, since $R_v(a_{ik\ell}) \geq r_v$, and using equation~\eqref{eqn:gauss-manin-aikl}, we obtain that $R_v(G_{ij}) \geq r_v$ for such~$v$. 
Thus \cref{exist-formal-periods}(4) holds for all non-archimedean places visible to~$R$.

Write $F_{ij}^{\diamond\vrig}$ for the image of $F_{ij}^\diamond$ under the map
$\cO(\fC_\formal) \to \cO(\fC_{v,\formal\dag}) \to \cO(\cC_v)$
given by base change along $\Spf(R_v, (\fp_v)) \to \Spec(R)$, followed by the $v$-adic rig functor.

\begin{lemma} \label{phi-omega-van}
For $1 \leq i \leq m$, we have the following equation in 
$\Omega^{\inv}_{\tilde\cG_v/\cC_v}$:
\[ \phi_v^*(\omega_i^\van|_{\cC_v}) = \sum_{j=1}^g F_{ij}^{\diamond\vrig} \tilde\omega_i^\van|_{\cC_v}. \]
\end{lemma}

\begin{proof}
Let $\nu_\fC \colon \cC_v = \fC_{v,\formal\dag}^\rig \to \fC_{v,\formal\dag}$ and $\nu_\fG \colon \fG_{v,\formal\dag}^\rig \to \fG_{v,\formal\dag}$ denote the natural morphisms of G-ringed spaces coming from the construction of the rigid generic fibre.
By \cite[7.1.12]{dJ95}, we have $\nu_\fC^* \Omega^1_{\fC_{v,\formal\dag}} \cong \Omega^1_{\cC_v}$ and $\nu_\fG^* \Omega^1_{\fG_{v,\formal\dag}} \cong \Omega^1_{\fG_{v,\formal\dag}^\rig}$.
Since the pullback functors $\nu_\fC^*$ and $\nu_\fG^*$ are right exact, we obtain $\nu_\fG^* \Omega^1_{\fG_{v,\formal\dag}/\fC_{v,\formal\dag}} \cong \Omega^1_{\fG_{v,\formal\dag}^\rig/\cC_v}$.
Consequently
\[ \nu_\fC^* \Omega^\inv_{\fG_{v,\formal\dag}/\fC_{v,\formal\dag}}
   \cong \nu_\fC^* e^* \Omega^1_{\fG_{v,\formal\dag}/\fC_{v,\formal\dag}}
   = e^* \nu_\fG^* \Omega^1_{\fG_{v,\formal\dag}/\fC_{v,\formal\dag}}
   \cong e^* \Omega^1_{\fG_{v,\formal\dag}^\rig/\cC_v}
   \cong \Omega^\inv_{\fG_{v,\formal\dag}^\rig/\cC_v}. \]
Hence we obtain an $\cO(\fC_{v,\formal\dag})$-linear map
\begin{equation} \label{eqn:Omega-G-formal-to-rigid}
\Omega^\inv_{\fG_{v,\formal\dag}/\fC_{v,\formal\dag}}(\fC_{v,\formal\dag})
   \to (\nu_\fC^* \Omega^\inv_{\fG_{v,\formal\dag}/\fC_{v,\formal\dag}})(\cC_v)
   \cong \Omega^\inv_{\fG_{v,\formal\dag}^\rig/\cC_v}(\cC_v)
\end{equation}
Let $\omega_1^\vrig, \dotsc, \omega_g^\vrig \in \Omega^\inv_{\fG_{v,\formal\dag}^\rig/\cC_v}(\cC_v)$ denote the image of $\omega_{1,\formal}|_{\fC_{v,\formal\dag}}, \dotsc, \omega_{g,\formal}|_{\fC_{v,\formal\dag}}$ under this map.

The analogous map
\begin{equation} \label{eqn:Omega-Gm-formal-to-rigid}
\Omega^\inv_{\GG_m^g \times \fC_{v,\formal\dag}/\fC_{v,\formal\dag}}(\fC_{v,\formal\dag})
   \to \Omega^\inv_{\GG_m(1)^g \times \cC_v/\cC_v}(\cC_v)
\end{equation}
maps $\tilde\omega_{1,\formal}|_{\fC_{v,\formal\dag}}, \dotsc, \tilde\omega_{g,\formal}|_{\fC_{v,\formal\dag}}$ to the standard basis for $\Omega^{\inv}_{\GG_m(1)^g \times \cC_v/\cC_v}$, which we denote by $\tilde\omega_1^\vrig, \dotsc, \tilde\omega_g^\vrig$.

We base-change $\phi_\formal \colon \GG_m^g \times \fC_\formal \to \fG_\formal$ to $\fC_{v,\formal\dag}$ and apply Berthelot's generic fibre functor to obtain $\phi_{v,\formal\dag}^\rig \colon \GG_m(1)^g \times \cC_v \to \fG_{v,\formal\dag}^\rig$.
By the naturality and $\cO(\fC_{v,\formal\dag})$-linearity of \eqref{eqn:Omega-G-formal-to-rigid} and~\eqref{eqn:Omega-Gm-formal-to-rigid}, equation~\eqref{eqn:Fdiamond} implies that
\[ (\phi_{v,\formal\dag}^\rig)^*(\omega_i^\vrig) = \sum_{j=1}^g F_{ij}^{\diamond\vrig} \tilde\omega_j^\vrig \]
in $\Omega^{\inv}_{\GG_m(1)^g \times \cC_v/\cC_v}$.

Since all forms in $\Omega^{\inv}_{\tilde\cG_v/\cC_v}$ are invariant under translations by $\tilde\cG_v$,
the restriction map $\Omega^{\inv}_{\tilde\cG_v/\cC_v} \to \Omega^{\inv}_{\fG_{v,\formal\dag}^\rig/\cC_v}$ is injective.
Using the injectivity of this restriction, together with the compatibility of $\phi_v$ with $\phi_{v,\formal\dag}$, we obtain
\[ \phi_v^*(\omega_i^\van|_{\cC_v}) = \sum_{j=1}^g F_{ij}^{\diamond\vrig} \tilde\omega_j^\van|_{\cC_v} \]
in $\Omega^{\inv}_{\tilde\cG_v/\cC_v}$ as required.
\end{proof}

By construction, we have $F_{ij}^{\diamond\vrig} = F_{ij}^\van \circ x^\van$ as analytic functions on $\cC_v$.
Hence evaluating the equation from \cref{phi-omega-van} at points $s \in U_v(L)$, for any finite extension $L/K_v$, we obtain
\[ \phi_{v,s}^*(\omega_{i,s}^\van) = \sum_{j=1}^g F_{ij}^\van(x^\van(s)) \tilde\omega_{j,s}^\van \]
in $\Omega^{\inv}_{(\GG_{m,L}^\van)^g/L}$.
Via the map $\Omega^{\inv}_{(\GG_{m,L}^\van)^g/L} \to H^1_{DR}((\GG_{m,L}^\van)^g/L)$, we obtain the first equation in \cref{exist-formal-periods}(5).

The second equation of \cref{exist-formal-periods}(5) is given by the following lemma.

\begin{lemma} \label{Gij-eqn-rigid}
For all $s \in U_v^*$, we have
$\phi_{v,s}^*(\eta_{i,s}^\van) = \sum_{j=1}^g G_{ij}^\van(x^\van(s)) \tilde\omega_{j,s}^\van$.
\end{lemma}

\begin{proof}
By assumption~(iv) of \cref{exist-formal-periods-reduction-3}, $C^*$ is equal to the Artin set $A_1(\pi|_{G^*})$, as defined in \cite[Def.~25.2.1]{ABC}.
This remains true after extending scalars to~$K_v$.
Hence, by \cite[Theorem~32.2.1]{ABC}, $H^1_{DR}(G^{*\van}/C^{*\van})$ is isomorphic to $H^1_{DR}(G^*/C^*)^\van$ as modules with integrable connection on~$C^{*\van}$
(the condition about Liouville numbers is satisfied because $G^*/C^*$ is defined over a number field).

Therefore, by analytifying equation~\eqref{eqn:gauss-manin-aikl}, we obtain the following equation in $H^1_{DR}(G^{*\van}/C^{*\van})$:
\[ \eta_i^\van = \sum_{k=0}^N \sum_{\ell=1}^m a_{ik\ell}^{\diamond\van} (\nabla^\van_{d/dx})^k \omega_\ell^\van. \]
Restricting to $\cC_v^*$, pulling back by $\phi_v$ and using \cref{phi-omega-van}, we obtain
\begin{align*}
    \phi_v^*(\eta_i^\van|_{\cC_v^*})
  & = \sum_{k=0}^N \sum_{\ell=1}^m a_{ik\ell}^{\diamond\van} (\nabla^\van_{d/dx})^k \phi^*_v(\omega_\ell^\van|_{\cC_v^*})
\\& = \sum_{k=0}^N \sum_{\ell=1}^m \sum_{j=1}^g a_{ik\ell}^{\diamond\van} (\nabla^\van_{d/dx})^k \bigl( F_{\ell j}^{\diamond\vrig} \tilde\omega_j^\van|_{\cC_v^*} \bigr)
\end{align*}
in $H^1_{DR}(\tilde\cG_v^*/\cC_v^*)$.

Since $\tilde\cG_v^*$ is a constant analytic space over $\cC_v^*$ and $\tilde\omega_j^\van$ are obtained by base change from differential forms in~$\Omega^{\inv}_{\GG_{m,\ZZ}/\ZZ}$, and since $F_{ij}^{\diamond\vrig} = F_{ij}^\van \circ x^\van|_{\cC_v}$, we have
\begin{align*}
    (\nabla^\van_{d/dx})^k \bigl( F_{\ell j}^{\diamond\vrig} \tilde\omega_j^\van|_{\cC_v^*} \bigr)
  & = \bigl( (d/dx)^k F_{\ell j}^{\diamond\vrig} \bigr) \tilde\omega_j^\van|_{\cC_v^*}
\\& = \Bigl( \bigl(\frac{d^k}{dX^k} F_{\ell j}^\van \bigr) \circ x^\van \Bigr) \tilde\omega_j^\van|_{\cC_v^*}.
\end{align*}
Therefore, we get
\begin{align*}
    \phi_v^*(\eta_i^\van|_{U_v^*})
  & = \sum_{j=1}^g \sum_{k=0}^N \sum_{\ell=1}^m a_{ik\ell}^{\diamond\van} \Bigl( \bigl(\frac{d^k}{dX^k} F_{\ell j^\van} \bigr) \circ x^\van \Bigr) \tilde\omega_j^\van|_{U_v^*}
\\& = \sum_{j=1}^g (G_{ij}^\van \circ x^\van) \, \tilde\omega_j^\van|_{U_v^*}.
\end{align*}
Evaluating at $s \in U_v^*(L)$ proves the lemma.
\end{proof}

\subsection{Archimedean interpretation of period G-functions}\label{subsec:arch-interp}

Let $v$ be an archi\-medean place of~$K$.
Since $x|_C$ is a local parameter around~$s_0$, we can choose an open set $U_v \subset C'^\van$ containing $s_0^\van$ and $r_v > 0$ such that $x^\van$ restricts to an isomorphism $U_v \to D(r_v, \CC)$.
After shrinking $U_v$ and $r_v$, we may assume that $U_v \subset \cC_v$ and that $r_v \leq R_v(a_{ik\ell})$ for all $i$, $k$, $\ell$.
Then \cref{exist-formal-periods}(1), (2) and~(3) hold for~$v$.
(Note that we can ignore archimedean places for the claim ``$r_v = 1$ for almost all~$v$'' in \cref{exist-formal-periods}(1).)

\begin{lemma} \label{arch-radius-Fij}
For each $i,j$, $R_v(F_{ij}) \geq r_v$.
For all $i=1, \dotsc, g$ and all $s \in U_v$, we have
$\phi_{v,s}^*(\omega_{i,s}^\van) = \sum_{j=1}^g F_{ij}^{\van}(x^\van(s)) \tilde\omega_{j,s}^\van$.
\end{lemma}

\begin{proof}
Let $\fC_{v,\formal\dag}$ denote the fibre product of formal schemes $\fC_\formal \times_{\Spec(R)} \Spec(\CC)$, using the morphism $\Spec(R) \to \Spec(\CC)$ associated with the place~$v$.

Let $C_\formal^\van$ denote the formal analytic completion of $C^\van$ along $\{ s_0^\van \}$.
Since $C^\van$ is a complex analytic curve and $x|_C$ is a local parameter around~$s_0$, there is an isomorphism $\cO(C_\formal^\van) \to \powerseries{\CC}{X}$ given by ``Taylor series in terms of~$x$''.

Since $\tilde\omega_1^\van, \dotsc, \tilde\omega_g^\van$ form an $\cO_{C^\van}$-basis for $\Omega^{\inv}_{(\GG_{m,\CC}^\an)^g \times C^\van / C^\van}$, we can write
\begin{equation} \label{eqn:arch-phi-v-F}
\phi_v^*(\omega_i^\van|_{\cC_v}) = \sum_{j=1}^g F_{ij,\van}^\diamond \tilde\omega_j^\van|_{\cC_v} \text{ for some } F_{ij,\van}^\diamond \in \cO(\cC_v).
\end{equation}

Base-changing from $\cC_v$ to the formal analytic completion $C_\formal^\van$, we obtain
\begin{equation*}
(\phi_v)_\formal^*(\omega_{i,\formal}^\van) = \sum_{j=1}^g (F_{ij,\van}^\diamond)_\formal \, \tilde\omega_{j,\formal}^\van.
\end{equation*}
Meanwhile, the base change of \eqref{eqn:Fdiamond} from $\fC_\formal$ to $C_\formal^\van$ gives
\begin{equation*}
(\phi_\formal)^{\van*}(\omega_{i,\formal}^\van) = \sum_{j=1}^g F_{ij} \, \tilde\omega_{j,\formal}^\van.
\end{equation*}
By the compatibility between $\phi_v$ and $\phi_\formal$, and since $\tilde\omega_{1,\formal}^\van, \dotsc, \tilde\omega_{g,\formal}^\van$ form an $\cO(C_\formal^\van)$-basis for $\Omega^{\inv}_{\GG_m^g \times C_\formal^\van / C_\formal^\van}$, we deduce that
\[ (F_{ij,\van}^\diamond)_\formal = F_{ij} \text{ in } \cO(C_\formal^\van). \]

In other words, $F_{ij} \in \powerseries{R}{X} \subset \powerseries{\CC}{X}$ is the Taylor series of the holomorphic function $F_{ij,\van}^\diamond \in \cO(\cC_v)$ in terms of the local parameter~$x^\van$.

Since the Taylor series of a holomorphic function on an open disc converges to that function on all of the disc, and since $U_v \subset \cC_v$ so $F_{ij,\van}^\diamond$ restricts to a holomorphic function on $U_v \cong D(r_v, \CC)$, this establishes that $R_v(F_{ij}) \geq r_v$ and
\[ F_{ij,\van}^\diamond(s) = F_{ij}^\van(x^\van(s)) \]
for all $s \in U_v$.
Combining this with \eqref{eqn:arch-phi-v-F} establishes the second part of the lemma.
\end{proof}

\Cref{arch-radius-Fij} shows that \cref{exist-formal-periods}(4) and~(5) hold for $F_{ij}$.
Since $r_v \leq R_v(a_{ik\ell})$ for all $i$, $k$, $\ell$, and by equation~\eqref{eqn:Gij-definition}, we deduce that \cref{exist-formal-periods}(4) holds for all $G_{ij}$ as well.
\Cref{exist-formal-periods}(5) for $G_{ij}$ follows by the same argument as the proof of \cref{Gij-eqn-rigid} (using the compatibility between the algebraic and complex analytic Gauss--Manin connections).

\subsection{Archimedean interpretation of period G-functions: further remarks} \label{subsec:arch-interp-integrals}

We can re-express \cref{exist-formal-periods}(5) for archimedean places in terms of period integrals as follows.

Let $\pi : \fG \to \fC$ denote the structure morphism of the semiabelian scheme in \cref{exist-formal-periods}
and let $\tilde\pi \colon \GG_{m,\CC}^g \times U_v \to U_v$ be the projection onto~~$U_v$.

Then $R_1\tilde\pi_*\ZZ$ is a constant local system of rank~$g$.
Let $\tilde\gamma_1, \dotsc, \tilde\gamma_g$ denote the basis for $R_1\tilde\pi_*\ZZ$ for which $\tilde\gamma_i$ comes from the standard generator of $H_1(\GG_{m,\CC}^\an, \ZZ)$ on the $i$-th factor of $(\GG_{m,\CC}^\an)^g$.
This basis is dual to the standard basis $\tilde\omega_1^\van, \dotsc, \tilde\omega_g^\van$ for $\Omega^{\inv}_{\tilde\cG_v/U_v}$ in the sense that, for each $s \in U_v$,
\begin{equation} \label{eqn:tilde-gamma-periods}
\frac{1}{2\pi i} \int_{\tilde\gamma_{j,s}} \tilde\omega_{i,s}^\van = \delta_{ij}.
\end{equation}

Let $\gamma_j = \phi_{v*} \tilde\gamma_j$, which is a global section of the sheaf $R_1\pi_*^\van\ZZ|_{U_v}$.
Similarly to the discussion in \cite[IX, 4.3]{And89}, $\gamma_1, \dotsc, \gamma_g$ form a $\ZZ$-basis for the global sections of the local system $R_1\pi_*^\van\ZZ|_{U_v^*}$ (the so-called ``locally invariant'' sections, that is, invariant under the monodromy action of $\pi_1(U^*_v)\cong\bZ$).

\begin{lemma}\label{lem:arch-periods}
For each $s \in U_v^*$,
\[ \frac{1}{2\pi i} \int_{\gamma_{j,s}} \omega_{i,s}^\van = F_{ij}^\van(x^\van(s)), \quad \frac{1}{2\pi i} \int_{\gamma_{j,s}} \eta_{i,s}^\van = G_{ij}^\van(x^\van(s)). \]
\end{lemma}

\begin{proof}
By the change-of-variable law for integrals, \cref{exist-formal-periods}(5) and equation~\eqref{eqn:tilde-gamma-periods}, 
\begin{align*}
    \frac{1}{2 \pi i} \int_{\gamma_{j,s}} \omega_{i,s}^\van
  & = \frac{1}{2 \pi i} \int_{\phi_{v,s*}\tilde\gamma_{j,s}} \omega_{i,s}^\van
    = \frac{1}{2 \pi i} \int_{\tilde\gamma_{j,s}} \phi_{v,s}^*\omega_{i,s}^\van
\\& = \frac{1}{2 \pi i} \int_{\tilde\gamma_{j,s}} \sum_{k=1}^g F_{ik}^\van(x^\van(s)) \tilde\omega_{k,s}^\van
\\& = F_{ij}^\van(x^\van(s)),
\end{align*}
and similarly for $G_{ij}$.
\end{proof}

This shows that our power series $F_{ij}$ and $G_{ij}$ are the Taylor series of relative periods of $G^*/C^*$, as defined (complex analytically) in \cite[IX, 1.2]{And89}.

\begin{lemma}
$F_{ij}$ and $G_{ij}$ are G-functions.
\end{lemma}

\begin{proof}
Given \cref{lem:arch-periods}, this is essentially \cite[IX, 4.2, Thm.~1]{And89}.
The parts of \cref{exist-formal-periods} which we have already proved give an alternative proof to the one in \cite{And89}, as follows:

\begin{enumerate}
\item Each power series $F_{ij}$ or~$G_{ij}$ satisfies a homogeneous $K(C)$-linear differential equation.
This comes from the action of the Gauss--Manin differential operator $\nabla_{d/dx}$ on the finite-dimensional $K(C)$-vector space $H^1_{DR}(G_{K(C)}/K(C))$.


\item By \cref{exist-formal-periods}(1) and~(4), $F_{ij}$ and~$G_{ij}$ are globally bounded.  This immediately implies the growth conditions on coefficients which appear in the definition of G-functions.
\end{enumerate}

Note that (1) is also a step in the proof in \cite{And89} (see \cite[IX, 1.2]{And89}), while (2) replaces the use of \cite[Thm.~B]{And89} (solutions in $\powerseries{\Qbar}{X}$ of ``geometric differential equations'' are G-functions).
We have used the arguments from \cite[IX, 4.3]{And89} in the proof of \cref{exists-complex-uniformisation}, and do not need to invoke them again at this stage.
\end{proof}

\section{Period relations from endomorphisms}\label{sec:relations}

In this section, we construct polynomial relations between the evaluations of the period G-functions $F_{ij}$ and $G_{ij}$ at points where the corresponding abelian variety has unlikely endomorphisms, at all places where our point is sufficiently close to the degeneration point.

Let $I$ denote the ideal of $\ov\QQ[Y_{ij}, Z_{ij} : 1 \leq i,j \leq g]$ generated by the entries of the $g \times g$ matrix $\underline{Y}^t \underline{Z} - \underline{Z}^t \underline{Y}$.
This is the ideal of trivial relations for an abelian scheme whose generic Mumford--Tate group is $\gGSp_{2g}$ \cite[p.~212]{And89} (this is also proved in \cref{subsec:trivial-relations}).

\begin{theorem} \label{exist-relations}
Let $R$, $K$, $\fC$ be as in~\cref{sec:formal-periods}.
Let $\fG \to \fC$ be a semiabelian scheme such that $\fG_0 \cong \GG_{m,\fC_0}^g$ and that $G^* \to C^*$ is an abelian scheme.
Suppose that the abelian scheme $G^* \to C^*$ is equipped with a polarisation.

Let $\omega_1, \dotsc, \omega_m, \eta_1, \dotsc, \eta_n$ be global sections of $H^1_{DR}(G^*/C^*)^\can$ satisfying the condition from \cref{exist-formal-periods} as well as the following additional conditions:
\begin{enumerate}[(i)]
\item $m=n=g$;
\item $\omega_1, \dotsc, \omega_g, \eta_1, \dotsc, \eta_g$ form an $\cO_{C^*}$-basis for $H^1_{DR}(G^*/C^*)$;
\item with respect to this basis, the symplectic form on $H^1_{DR}(G^*/C^*)$ induced by the polarisation is represented by a matrix of the form $e\fullsmallmatrix{0}{I}{-I}{0}$ for some $e \in \cO(C^*)^\times$.
\end{enumerate}

Let $U_v$, $F_{ij}$, $G_{ij}$ be the sets and power series given to us by \cref{exist-formal-periods}.

Then there is a constant $\newC{relation-degree-mult}$, depending only on~$g$, such that, for every number field $\hat K$ containing $K$ and every point $s \in C^*(\hat K)$, if $G_{s, \Qbar}$ has unlikely endomorphisms, then there exists a polynomial $P \in \Qbar[Y_{ij}, Z_{ij} : 1 \leq i,j \leq g]$ satisfying:
\begin{enumerate}[(a)]
\item $P$ is homogeneous of degree at most $\refC{relation-degree-mult} [\hat K:K]$;
\item for each place $\hat v$ of $\hat K$, if $v:=\hat v|_K$ is visible to~$R$ and $s^\vhatan \in U_v(\hat K_{\hat v})$, then
\[ P^\vhatan \bigl( \underline{F}^\van(x^\van(s^\vhatan)), \, \underline{G}^\van(x^\van(s^\vhatan)) \bigr) = 0; \]
\item $P\notin I$.
\end{enumerate}
\end{theorem}

While \cref{exist-relations} does not formally require $G^*_\CC$ to have generic Mumford--Tate group $\gGSp_{2g}$, in practice it is only useful under this condition, because we would expect an abelian scheme with smaller Mumford--Tate group to have an ideal of trivial relations which is larger than~$I$.
We note also that, when the generic Mumford--Tate group is $\gGSp_{2g}$, there always exist $\omega_1, \dotsc, \omega_g, \eta_1, \dotsc, \eta_g$ satisfying the conditions of \cref{exist-relations} (\cref{gauss-manin-basis}, \cref{omega-eta-basis}).

\subsection{Construction of period relations: first steps} \label{subsec:relations-summary}

As in \cref{exist-relations}, let $s$ be a point in $C^*(\hat K)$ such that $G_{s,\Qbar}$ has unlikely endomorphisms.

To reduce notation, for each place~$\hat v$ of~$\hat K$ such that $v:=\hat v|K$ is visible to~$R$ and $s \in U_v(\hat K_{\hat v})$, write $\underline F_{\hat v}, \underline G_{\hat v} \in \rM_{g \times g}(\hat K_{\hat v})$ for the matrices whose entries are $F_{ij}^\van(x^\van(s^\vhatan))$ and~$G_{ij}^\van(x^\van(s^\vhatan))$ respectively.
According to \cref{exist-formal-periods}(5), the $g \times 2g$ matrix $(\underline F_{\hat v} \; \underline G_{\hat v})$ represents the $\hat K_{\hat v}$-linear map
\[ \phi_{v,s}^* \colon H^1_{DR}(G_s^\van/\hat K_{\hat v}) \to H^1_{DR}((\GG_{m,\hat K_{\hat v}}^\an)^g / \hat K_{\hat v}) \]
with respect to the bases $\{ \omega_{i,s}^\van, \eta_{i,s}^\van \}$ and $\{ \tilde\omega_{j,s}^\van \}$.

Our approach to proving \cref{exist-relations} is as follows:
\begin{enumerate}
\item We construct a homogeneous polynomial $P_{fin} \in \hat K[\underline Y, \underline Z]$ such that, for every non-archimedean place $\hat v$ of~$\hat K$, if $\hat v|_K$ is visible to~$R$ and $s \in U_v(\hat K_v)$, then $P_{fin}^\van(\underline F_{\hat v}, \underline G_{\hat v}) = 0$.
Note that $P_{fin}$ depends on the point~$s$, but is independent of the place~$\hat v$.
\item For each archimedean place $\hat v$ of~$\hat K$, if $s \in U_v(\hat K_v)$, then we construct a homogeneous polynomial $P_{\hat v} \in \hat K[\underline Y, \underline Z]$ such that $P_{\hat v}^\van(\underline F_{\hat v}, \underline G_{\hat v}) = 0$.
\end{enumerate}

Letting $P = P_{fin} \cdot \prod_{\hat v} P_{\hat v}$, where the product runs over all archimedean places $\hat v$ of $\hat K$ for which $s \in U_v(\hat K_v)$, we obtain a polynomial~$P$ satisfying \cref{exist-relations}(b).

We have $\deg(P_{fin}) = g+1$ and $\deg(P_{\hat v}) \leq 2$ for each archimedean place~$\hat v$ of~$\hat K$, so $\deg(P) \leq 2[\hat K:\QQ]+g+1$.

We shall also prove that the polynomials $P_{fin}$ and $P_{\hat v}$ that we construct are not in~$I$.
Since the ideal~$I$ is prime by \cref{lem:prime}, this suffices to prove \cref{exist-relations}(c).

\subsection{Period relations at non-archimedean places}

\begin{lemma}
In the setting of \cref{exist-relations}, let $\hat K$ be a number field containing~$K$ and let $s$ be a point in $C^*(\hat K)$ such that $\End(G_{s,\ov\QQ}) \not\cong \ZZ$.
Then there exists a polynomial $P_{fin} \in \ov\QQ[Y_{ij}, Z_{ij} : 1 \leq i,j \leq g]$ satisfying:
\begin{enumerate}[(i)]
\item $P_{fin}$ is homogeneous of degree $g+1$;
\item for each non-archimedean place $\hat v$ of $\hat K$, if $v:=\hat v|_K$ is visible to~$R$ and $s^\vhatan \in U_v(\hat K_{\hat v})$, then
\[ P_{fin}^\vhatan \bigl( \underline{F}^\van(x^\van(s^\vhatan)), \, \underline{G}^\van(x^\van(s^\vhatan)) \bigr) = 0; \]
\item $P_{fin} \notin I$.
\end{enumerate}
\end{lemma}

\begin{proof}
Choose some $f \in \End(G_{s,\ov\QQ}) \setminus \ZZ$.
Let $\hat L$ be a finite extension of $\hat K$ such that $f \in \End(G_{s,\hat L})$.

The endomorphism $f$ induces an $\hat L$-linear map $f^* \colon H^1_{DR}(G_s/\hat L) \to H^1_{DR}(G_s/\hat L)$, which maps $\Omega^{\inv}_{G_s/\hat L}$ into itself.
With respect to the basis $\omega_{1,s}, \dotsc, \omega_{g,s}, \eta_{1,s}, \dotsc, \eta_{g,s}$, we can represent $f^*$ by a matrix
\[ \fullmatrix{A}{B}{0}{D} \in \rM_{2g}(\hat L), \]
where $A, B, D \in \rM_{g \times g}(\hat L)$.

Let $\hat v$ be a non-archimedean place of~$\hat K$ such that $v:=\hat v|_K$ is visible to~$R$ and $s \in U_v(\hat K_{\hat v})$.
Choose an extension of $\hat v$ to a place of~$\hat L$, which we still denote by $\hat v$.

According to \cite[Satz~5]{Ger70}, there exists an endomorphism $\tilde f$ of~$(\GG_{m,\hat L_{\hat v}}^\an)^g$ such that the following diagram commutes:
\begin{equation} \label{eqn:rigid-lift-diagram}
\begin{tikzcd}
    (\GG_{m,\hat L_v}^\an)^g  \ar[d, "\phi_{v,s}"]  \ar[r, "\tilde f"]
  & (\GG_{m,\hat L_v}^\an)^g  \ar[d, "\phi_{v,s}"]
\\  G_{s,\hat L}^\vhatan      \ar[r, "f^\vhatan"]
  & G_{s,\hat L}^\vhatan.
\end{tikzcd}
\end{equation}
The endomorphism $\tilde f$ induces an $\hat L_{\hat v}$-linear endomorphism $\tilde f^*$ of $H^1_{DR}((\GG_{m,\hat L_{\hat v}}^\an)^g / \hat L_{\hat v})$.
Let $M_{\hat v} \in \rM_{g \times g}(\hat L_{\hat v})$ denote the matrix representing $\tilde f^*$ in terms of the basis $\tilde\omega_{1,s}^\vhatan, \dotsc, \tilde\omega_{g,s}^\vhatan$.

From the diagram~\eqref{eqn:rigid-lift-diagram}, we obtain the equation of pullback maps
\begin{equation} \label{eqn:rigid-lift-pullback}
\tilde f^* \circ \phi_{v,s}^* = \phi_{v,s}^* \circ (f^\vhatan)^* \colon  H^1_{DR}(G_s^\vhatan/\hat L_{\hat v}) \to H^1_{DR}((\GG_{m,\hat L_{\hat v}}^\an)^g / \hat L_{\hat v}).
\end{equation}

By \cref{exist-formal-periods}(5), the $g \times 2g$ matrix representing $\phi_{v,s}^*$ with respect to the bases $\omega_1^\vhatan, \dotsc, \omega_g^\vhatan, \eta_1^\vhatan, \dotsc, \eta_g^\vhatan$ and $\tilde\omega_1^\vhatan, \dotsc, \tilde\omega_g^\vhatan$ is $(\underline F_{\hat v}^t \;\; \underline G_{\hat v}^t)$.
Thus, writing \eqref{eqn:rigid-lift-pullback} in terms of matrices, we obtain
\[ M_{\hat v} (\underline F_{\hat v}^t \;\; \underline G_{\hat v}^t) = (\underline F_{\hat v}^t \;\; \underline G_{\hat v}^t) \fullmatrix{A}{B}{0}{D} \]
or in other words
\begin{align}
   M_{\hat v} \underline F_{\hat v}^t & = \underline F_{\hat v}^t A, \label{eqn:MF}
\\ M_{\hat v} \underline G_{\hat v}^t & = \underline F_{\hat v}^t B + \underline G^t_{\hat v} D. \label{eqn:MG}
\end{align}

We now make use of the adjugate matrix $(\underline F_{\hat v}^t)^\adj$, whose entries are the cofactors of $\underline F_{\hat v}$.
We shall use the well-known fact that
\[ \underline F_{\hat v}^t (\underline F_{\hat v}^t)^\adj = \det(\underline F_{\hat v}) I_g, \]
where $I_g$ denotes the $g \times g$ identity matrix.

Computing $\eqref{eqn:MF} (\underline F_{\hat v}^t)^\adj \underline G_{\hat v}^t - \det(\underline F_{\hat v}) \eqref{eqn:MG}$, we obtain
\[ 0 = \underline F_{\hat v}^t A (\underline F_{\hat v}^t)^\adj \underline G_{\hat v}^t - \det(\underline F_{\hat v}) \underline F_{\hat v}^t B - \det(\underline F_{\hat v}) \underline G_{\hat v}^t D. \]
Thus, defining
\[ \underline P(\underline Y, \underline Z) = \underline Y^t A (\underline Y^t)^\adj \underline Z^t - \det(\underline Y) \underline Y^t B - \det(\underline Y) \underline Z^t D \in \rM_{g \times g}(\ov\QQ[\underline Y, \underline Z]), \]
and letting $P_{ij} \in \hat K[\underline Y, \underline Z]$ denote the $(ij)$-th entry of the polynomial matrix~$\underline P$, we obtain that
\[ P_{ij}(\underline F_{\hat v}, \underline G_{\hat v}) = 0 \]
for all $i,j$.

Note that $f$, $A$, $B$, $D$ are independent of the non-archimedean place $\hat v$, so the same is true of the polynomials~$P_{ij}$.
By construction, these polynomials are homogeneous of degree~$g+1$ in the variables $Y_{ij}$, $Z_{ij}$.

By the lemma below, at least one of the polynomials $P_{ij}$ is not in the ideal~$I$.
Therefore, that polynomial~$P_{ij}$ has the properties required for~$P_{fin}$.
\end{proof}

\begin{lemma} \label{non-arch-non-trivial}
There exist indices $i, j \in \{ 1, \dotsc, g \}$ such that $P_{ij} \not\in I$.
\end{lemma}

\begin{proof}
In order to show that $P_{ij} \not\in I$, it suffices to show that there exist matrices $\underline y, \underline z \in \rM_{g \times g}(\hat K)$ such that every entry of $\underline y^t \underline z - \underline z^t \underline y$ is zero, while $P_{ij}(\underline y, \underline z) \neq 0$.
Thus, in order to prove the lemma, it suffices to find $\underline y, \underline z \in \rM_{g \times g}(\hat K)$ such that $\underline y^t \underline z - \underline z^t \underline y = 0$ while $\underline P(\underline y, \underline z) \neq 0$ (as matrices).

Since $f \not\in \ZZ$, by \cref{scalar-action}, the matrix $\fullmatrix{A}{B}{0}{D}$ is not scalar.
Consequently, one of the following cases occurs:
\begin{enumerate}
\item $B \neq 0$.  In this case, we can take $\underline y = I_g$, $\underline z = 0$ to obtain $\underline P(\underline y, \underline z) = -B \neq 0$ while $\underline y^t \underline z - \underline z^t \underline y = 0-0 = 0$.
\item $B = 0$ and $A \neq D$.  In this case, we can take $\underline y = \underline z = I_g$, to obtain $\underline P(\underline y, \underline z) = A-D \neq 0$ while $\underline y^t \underline z - \underline z^t \underline y = I_g - I_g = 0$.
\item $B=0$, $A=D$ and $A$ is not a scalar matrix.  In this case, there exists a symmetric matrix $\underline z \in \rM_{g \times g}(\hat K)$ which does not commute with~$A$ (indeed, if $A$ is not diagonal, choose $z = E_{ii}$ where $A$ has a non-zero, non-diagonal entry in the $i$-th column; if $A_{ii} \neq A_{jj}$, choose $z = E_{ij} + E_{ji}$, where $E_{ij}$ denotes the matrix with $1$ as the $ij$-entry and zeroes everywhere else).  Choosing this~$\underline z$ along with $\underline y = I_g$, we obtain $\underline P(\underline y, \underline z) = A\underline z - \underline z D \neq 0$ while $\underline y^t \underline z - \underline z^t \underline y = \underline z - \underline z^t = 0$.
\qedhere
\end{enumerate}
\end{proof}

The proof above makes use of the following lemma.

\begin{lemma} \label{scalar-action}
Let $G$ be an abelian variety over a field $\hat L$ which can be embedded into~$\CC$.
Let $f \in \End(G) \setminus \ZZ$.
Then $f^* \in \End(H^1_{DR}(G/\hat L))$ is not a scalar.
\end{lemma}

\begin{proof}
Let $f^*_B$ denote the endomorphism of the singular cohomology $H^1(G^{\CCan}, \ZZ)$ induced by~$f$.
Since $f \not\in \ZZ$ and the action of $\End(G)$ on the $\ZZ$-module $H^1(G^{\CCan}, \ZZ)$ is faithful, $f^*_B$ is not a scalar.
Using the comparison isomorphism
\[ H^1_{DR}(G/\CC) \cong H^1(G^{\CCan}, \ZZ) \otimes_\ZZ \CC, \]
this implies that $f^*$ acting on $H^1_{DR}(G/\hat L)$ is not a scalar.
\end{proof}

\subsection{Period relations at archimedean places} \label{subsec:archimedean-relations}

\begin{lemma} \label{exist-relations-arch}
In the setting of \cref{exist-relations}, let $\hat K$ be a number field containing~$K$.
Let $s$ be a point in $C^*(\hat K)$ such that $G_{s,\Qbar}$ has unlikely endomorphisms.

Let $\hat v$ be an archimedean place of~$\hat K$ such that $s^\vhatan \in U_v$, where $v = \hat v|_K$.

Then there exists a polynomial $P_{\hat v} \in \ov\QQ[Y_{ij}, Z_{ij} : 1 \leq i,j \leq g]$ satisfying:
\begin{enumerate}[(i)]
\item $P_{\hat v}$ is homogeneous of degree at most~$2$;
\item $P_{\hat v}^\vhatan \bigl( \underline{F}^\van(x^\van(s^\vhatan)), \, \underline{G}^\van(x^\van(s^\vhatan)) \bigr) = 0$;
\item $P_{\hat v} \notin I$.
\end{enumerate}
\end{lemma}

\begin{proof}[Proof of \cref{exist-relations-arch}, except for the cases of $\QQ \times \QQ$ or a real quadratic field]
We follow the strategy of \cite[Ch.~X]{And89}.

Let $\gamma_1, \dotsc, \gamma_g$ denote the $\ZZ$-basis for the locally invariant sections of the local system $R_1\pi_*^\vhatan\ZZ|_{U_v^*}$ chosen in \cref{subsec:arch-interp-integrals}.
According to \cite[X, Lemma~2.3]{And89}, the locally invariant sections form a maximal isotropic subsystem of $R_1\pi_*^\vhatan\ZZ|_{U_v^*}$ with respect to the symplectic form induced by the polarisation of $G^* \to C^*$.
Therefore, we can choose $\delta_1, \dotsc, \delta_g$ such that $\gamma_{1,s}, \dotsc, \gamma_{g,s}, \delta_1, \dotsc, \delta_g$ form a symplectic basis for $H_1(G_s^\vhatan, \ZZ)$.

Let $\hat D = \End(G_{s,\ov\QQ}) \otimes \QQ$.
Let $\hat E$ be a maximal commutative semisimple subalgebra of $\hat D$.
For some types of algebra~$\hat D$, we will later impose a more specific choice of~$\hat E$.

The algebra $\hat E$ is a product of fields.
Let $\hat F$ be a number field which contains $\hat K$, the Galois closures of the direct factors of~$\hat E$, the field of definition of all endomorphisms of $G_{s,\ov\QQ}$, and the square root $\sqrt{e(s)}$ where $e \in K(C)^\times$ appears in \cref{exist-relations}(iii).
($\hat F$ here plays the role which was played by $\hat K$ in \cite{And89} and \cite{ExCM}; this notation seems more convenient as we avoid needing to replace $\hat K$ by a finite extension of bounded degree.)
We view $\hat F$ as a subfield of~$\CC$, via an embedding whose restriction to $\hat K$ is associated with the place~$\hat v$.

The action of $\hat E$ on $H_1(G_s^\vhatan, \QQ)$ induces a splitting of $\hat F$-vector spaces
\[ H_1(G_s^\vhatan, \QQ) \otimes_\QQ \hat F = \bigoplus_{\hat\sigma} \hat W_{\hat\sigma} \]
where the sum is over $\QQ$-algebra homomorphisms $\hat\sigma \colon \hat E \to \hat F$.
Similarly, the action of $\hat E$ on $H^1_{DR}(G_{s,\ov\QQ}/\ov\QQ)$ induces a splitting of $\hat F$-vector spaces
\[ H^1_{DR}(G_{s,\hat F}/\hat F) = \bigoplus_{\hat\sigma \colon \hat E \to \hat F} \hat W_{DR}^{\hat\sigma}. \]
Since each direct factor $E_i$ of $\hat E$ acts non-trivially on $H_1(G_s^\vhatan, \QQ)$, and since $\Gal(\ov\QQ/\QQ)$ acts transitively on the homomorphisms $E_i \to \hat F$, all subspaces $\hat W_{\hat\sigma}$ are non-zero.
By the duality between $\hat W_{\hat\sigma} \otimes_{\hat F} \CC$ and $\hat W_{DR}^{\hat\sigma} \otimes_{\hat F} \CC$, we deduce that all subspaces $\hat W_{DR}^{\hat\sigma}$ are non-zero.

Let $\hat W^1$ denote the $\hat F$-linear subspace of $H_1(G_s^\vhatan, \QQ) \otimes_\QQ \hat F$ spanned by $\gamma_{1,s}, \dotsc, \gamma_{g,s}$ (that is, the intersection of $H_1(G_s^\vhatan, \QQ) \otimes_\QQ \hat F$ with the fibre at~$s$ of the maximal constant local subsystem of $R_1\pi_*^\vhatan\QQ|_{U_v^*}$).
Note that
\[ \dim_{\hat F}(\hat W^1) = g = \tfrac12 \dim_{\hat F} (H_1(G_s^\vhatan, \QQ) \otimes_\QQ \hat F). \]

By hypothesis, $\hat D \neq \QQ$.
This implies that $[\hat E:\QQ] \neq 1$.
We remark that, at the beginning of \cite[X, Construction~2.4.1]{And89}, or in \cite[Thm.~8.1]{ExCM}, establishing the analogous fact $[\hat E:E]\neq 1$ requires additional restrictions on $\hat D$ (where $E$ is a maximal commutative semisimple subalgebra of the generic endomorphism algebra of $\fG \to \fC$).
In the present setting, we avoid any restrictions on~$\hat D$ by using the hypothesis that $E=\QQ$.

Therefore, as in \cite[X, Construction~2.4.1]{And89}, we may split into three cases.
\begin{itemize}
\item \textit{Case~1}. $[\hat E:\QQ] \geq 3$.
\item \textit{Case~2}. $[\hat E:\QQ] = 2$ and there exists a homomorphism $\hat\sigma \colon \hat E \to \CC$ such that $\hat W_{\hat\sigma} \cap \hat W^1 \neq \{0\}$.
\item \textit{Case~3}. $[\hat E:\QQ] = 2$ and, for both homomorphisms $\hat\sigma \colon \hat E \to \CC$, we have $\hat W_{\hat\sigma} \cap \hat W^1 = \{0\}$.
\end{itemize}

In Cases~1 and~2, we can follow \cite[X, Construction~2.4.1]{And89} to obtain a non-zero polynomial in $\hat F[\underline Y, \underline Z]$, homogeneous of degree~$1$, which vanishes at $(\underline F_{\hat v}, \underline G_{\hat v})$.
Since all generators of~$I$ are homogeneous of degree~$2$, this polynomial is not in~$I$.

In Case~3, $\hat D$ may be $\QQ \times \QQ$, a real quadratic field, an imaginary quadratic field, or a quaternion algebra over~$\QQ$ (split or non-split).
The cases where $\hat D$ is $\QQ \times \QQ$ or a real quadratic field are dealt with in the the next subsection.

If $\hat D$ is an imaginary quadratic field, we have $\hat E = \hat D$.
If $\hat D$ is a quaternion algebra, then, by \cite[Lemma~8.7]{ExCM}, we may choose $\hat E \subset \hat D$ to be an imaginary quadratic field preserved by the Rosati involution.
Consequently, in either of these cases, the argument from \cite[X, Construction~2.4.1, Case~3]{And89} applies.
(It is explained in the penultimate paragraph of the proof of \cite[Theorem~8.1]{ExCM} why the argument in \cite[X, Construction~2.4.1, Case~3]{And89} requires $\hat E$ to be a CM field preserved by the Rosati involution.)
This gives a homogeneous polynomial $P_{\hat v}$ of degree~$2$ such that $P_{\hat v}(\underline F_{\hat v}, \underline G_{\hat v}) \neq 0$.  It is shown in \cite[X, Lemma~3.3]{And89} that this polynomial~$P_{\hat v}$ is not in~$I$.
\end{proof}

\subsection{Period relations at archimedean places: Case 3, real quadratic endomorphism algebra}

We will now prove the remaining case of \cref{exist-relations-arch}, namely, Case~3, with $\hat D = \hat E$ being either $\QQ \times \QQ$ or a real quadratic field, using the same notation as in~\cref{subsec:archimedean-relations}.
For these algebras $\hat D$, the hypothesis that $G_s$ has unlikely endomorphisms implies that $g>2$.

Let $\hat\sigma, \hat\tau$ denote the two homomorphisms $\hat E \to \hat F$.
Let $p \colon \hat W^1 \to \hat W_{\hat\sigma}$ and $q \colon \hat W^1 \to \hat W_{\hat\tau}$ denote the restriction to $\hat W^1$ of the projections coming from the splitting
\[ H_1(G_s^\vhatan, \QQ) \otimes_\QQ \hat F = \hat W_{\hat\sigma} \oplus \hat W_{\hat\tau}. \]
The conditions that $\dim_{\hat F}(\hat W^1) = g$ and $\hat W_{\hat\sigma} \cap \hat W^1 = \hat W_{\hat\tau} \cap \hat W^1 = \{0\}$ imply that
\[ \dim_{\hat F}(\hat W_{\hat\sigma}) = \dim_{\hat F}(\hat W_{\hat\tau}) = g, \]
and that $p \colon \hat W^1 \to \hat W_{\hat\sigma}$, $q \colon \hat W^1 \to \hat W_{\hat\tau}$ are isomorphisms.

Let
\[ 2\pi i \langle \cdot, \cdot \rangle \colon H_1(G_s^\vhatan, \hat F) \times H_1(G_s^\vhatan, \hat F) \to \hat F(2\pi i) \]
denote the symplectic form induced by the polarisation of~$G_s$.
The Rosati involution of $\hat D$ is the identity.  It follows that $\hat W_{\hat\sigma}$ is orthogonal to $\hat W_{\hat\tau}$ with respect to~$\langle \cdot, \cdot \rangle$.
Consequently, the non-degeneracy of~$\langle \cdot, \cdot \rangle$ implies that $\langle \cdot, \cdot \rangle$ restricts to non-degenerate symplectic forms on $\hat W_{\hat\sigma}$ and $\hat W_{\hat\tau}$ respectively.
(This is where the present case differs from Case~3 with $\hat E$ an imaginary quadratic field, as considered in \cite[X, Construction~2.4.1]{And89}, since in that case the Rosati involution restricts to complex conjugation on $\hat E$ and the subspaces $\hat W_{\hat\sigma}, \hat W_{\hat\tau}$ are $\psi$-isotropic.)
This implies that $\dim_{\hat F}(\hat W_{\hat\sigma}) = g$ is even.

The action of $\hat E$ on $H^1_{DR}(G_{s,\ov\QQ}/\ov\QQ)$ induces a splitting of $\hat F$-vector spaces
\[ H^1_{DR}(G_{s,\hat F}/\hat F) = \hat W_{DR}^{\hat\sigma} \oplus \hat W_{DR}^{\hat\tau}. \]
The duality between de Rham cohomology and singular homology implies that
\[ \dim_{\hat F}(\hat W_{DR}^{\hat\sigma}) = \dim_{\hat F}(\hat W_{DR}^{\hat\tau}) = g. \]

Let $\langle \cdot, \cdot \rangle_{DR}$ denote the symplectic form on $H^1_{DR}(G_{s,\hat F}/\hat F)$ induced by the polarisation of $G_s$ via the isomorphism $H_1^{DR}(G^*/C^*)^\vee \to H_1^{DR}((G^*)^\vee/C^*)$ given by \cite[Thm.~5.1.6]{BBM82}).
By the same arguments as for $\langle \cdot, \cdot \rangle$, $\hat W_{DR}^{\hat\sigma}$ is orthogonal to $\hat W_{DR}^{\hat\tau}$ with respect to $\langle \cdot, \cdot \rangle_{DR}$, and the restrictions of $\langle \cdot, \cdot \rangle_{DR}$ to $\hat W_{DR}^{\hat\sigma}$ and $\hat W_{DR}^{\hat\tau}$ respectively are non-degenerate.

Choose a symplectic $\hat F$-basis $\hat\omega_1, \dotsc, \hat\omega_g, \hat\eta_1, \dotsc, \hat\eta_g$ for $H^1_{DR}(G_{s,\hat F}/\hat F)$ such that:
\begin{enumerate}
\item $\hat\omega_1, \dotsc, \hat\omega_{g/2}, \hat\eta_1, \dotsc, \hat\eta_{g/2}$ form a symplectic $\hat F$-basis for $\hat W_{DR}^{\hat\sigma}$; and
\item $\hat\omega_{g/2+1}, \dotsc, \hat\omega_g, \hat\eta_{g/2+1}, \dotsc, \hat\eta_g$ form a symplectic $\hat F$-basis for $\hat W_{DR}^{\hat\tau}$.
\end{enumerate}
As the notation suggests, we can choose this basis so that $\hat\omega_1, \dotsc, \hat\omega_g$ are invariant differential forms on $G_{s,\hat{F}}$, but we won't use this extra condition on $\hat\omega_1, \dotsc, \hat\omega_g$.

Let $\fullmatrix{A}{B}{C}{D} \in \rM_{2g}(\hat F)$ be the matrix which expresses
$\hat\omega_1, \dotsc, \hat\omega_g, \hat\eta_1, \dotsc, \hat\eta_g$
as linear combinations of $\omega_{1,s}, \dotsc, \omega_{g,s}, \eta_{1,s}, \dotsc, \eta_{g,s}$. That is,
\begin{align*}
   \hat\omega_i & = \sum_{k=1}^g A_{ki} \omega_{k,s} + \sum_{k=1}^g C_{ki} \eta_{k,s},
\\ \hat\eta_i   & = \sum_{k=1}^g B_{ki} \omega_{k,s} + \sum_{k=1}^g D_{ki} \eta_{k,s}.
\end{align*}
Since $\omega_1, \dotsc, \omega_g, \eta_1, \dotsc, \eta_g$ satisfy condition~(iii) of \cref{exist-relations}, while $\hat\omega_1, \dotsc, \hat\omega_g$, $\hat\eta_1, \dotsc, \hat\eta_g$ form a symplectic basis, we have
\[ \sqrt{e(s)} \fullmatrix{A}{B}{C}{D} \in \gSp_{2g}(\hat F). \]

Let $\underline{\hat F}_{\hat v}$, $\underline{\hat G}_{\hat v}$ denote the $g \times g$ matrices of periods with entries
\begin{equation*}
    (\underline{\hat F}_{\hat v})_{ij} = \frac{1}{2\pi i} \int_{\gamma_{j,s}} \hat\omega_i^\vhatan, \qquad
(\underline{\hat G}_{\hat v})_{ij} = \frac{1}{2\pi i} \int_{\gamma_{j,s}} \hat\eta_i^\vhatan. 
\end{equation*}
Then we have
\begin{align}\label{eqn:hatF}
   \underline{\hat F}_{\hat v} & = A^t \underline F_{\hat v} + C^t \underline G_{\hat v},
\\ \label{eqn:hatG}
   \underline{\hat G}_{\hat v} & = B^t \underline F_{\hat v} + D^t \underline G_{\hat v}.
\end{align}

Let us also write $\underline{\hat H}_{\hat\sigma}$ for the $g \times g$ matrix made up of the first $g/2$ rows of $\underline{\hat F}_{\hat v}$ and the first $g/2$ rows of $\underline{\hat G}_{\hat v}$.
Thus, the entries of $\underline{\hat H}_{\hat\sigma}$ are the periods of a basis of $\hat W_{DR}^{\hat\sigma}$ with respect to a basis of $\hat W^1$.

We can describe the entries of $\underline{\hat H}_{\hat\sigma}$ in a slightly different way as follows.
For the top half of $\underline{\hat H}_{\hat\sigma}$, for each $i = 1, \dotsc, g/2$ and $j = 1, \dotsc, g$, we have $\hat\omega_i \in \hat W_{DR}^{\hat\sigma}$ and $q(\gamma_{j,s}) \in \hat W_{\hat\tau}$, so
\[ \frac{1}{2\pi i} \int_{q(\gamma_{j,s})} \hat\omega_i^\vhatan = 0. \]
Consequently,
\begin{equation*}
    (\underline{\hat H}_{\hat\sigma})_{ij}
   = (\underline{\hat F}_{\hat v})_{ij}
   = \frac{1}{2\pi i} \int_{p(\gamma_{j,s})} \hat\omega_i^\vhatan + \frac{1}{2\pi i} \int_{q(\gamma_{j,s})} \hat\omega_i^\vhatan
   = \frac{1}{2\pi i} \int_{p(\gamma_{j,s})} \hat\omega_i^\vhatan. 
\end{equation*} 
A similar argument shows that, for the bottom half of $\underline{\hat H}_{\hat\sigma}$, for each $i = 1, \dotsc, g/2$ and $j = 1, \dotsc, g$, we have
\begin{equation*}
    (\underline{\hat H}_{\hat\sigma})_{g/2+i,j}
   = (\underline{\hat G}_{\hat v})_{ij}
   = \frac{1}{2\pi i} \int_{p(\gamma_{j,s})} \hat\eta_i^\van. 
\end{equation*} 
Thus we can interpret $\underline{\hat H}_{\hat\sigma}$ as a period matrix for the basis $\hat\omega_1, \dotsc, \hat\omega_{g/2}, \hat\eta_1, \dotsc, \hat\eta_{g/2}$ of $\hat W_{DR}^{\hat\sigma}$ with respect to the basis $p(\gamma_{1,s}), \dotsc, p(\gamma_{g,s})$ of $\hat W_{\hat\sigma}$.

\begin{lemma} \label{case-3-real-exists-relation}
There exists a non-zero homogeneous polynomial $Q \in \hat F[X_{ij} : 1 \leq i,j \leq g]$ of degree~$2$ such that $Q(\underline{\hat H}_{\hat\sigma}) = 0$.
\end{lemma}

\begin{proof}
Since the duality between $H^1_{DR}(G_{s,\hat F}/\hat F) \otimes_{\hat F} \CC$ and $H_1(G_s^\vhatan, \QQ) \otimes_\QQ \hat \CC$ is compatible with the symplectic forms $\langle \cdot, \cdot \rangle_{DR}$ and $\langle \cdot, \cdot \rangle$ (up to a factor $1/2\pi i$), and since $\hat\omega_1, \dotsc, \hat\omega_{g/2}, \hat\eta_1, \dotsc, \hat\eta_{g/2}$ is a symplectic basis for $\hat W_{DR}^{\hat\sigma}$, we obtain that
\begin{equation} \label{eqn:case-3-relation-hatH}
\underline{\hat H}_{\hat\sigma}^t J \underline{\hat H}_{\hat\sigma} = \frac{1}{2\pi i} M,
\end{equation}
where $J = \fullmatrix{0}{I_{g/2}}{-I_{g/2}}{0}$ is the matrix representing the standard symplectic form on $\CC^g$ and $M$ is the matrix representing the non-degenerate symplectic form $\langle \cdot, \cdot \rangle|_{\hat W_{\hat\sigma}}$ with respect to the (not necessarily symplectic) basis $p(\gamma_{1,s}), \dotsc, p(\gamma_{g,s})$.
Since this is an $\hat F$-basis and $\langle \cdot, \cdot \rangle|_{\hat W_{\hat\sigma}}$ is an $\hat F$-symplectic form, we have $M \in \rM_g(\hat F)$.

Let $R,S \in \QQ[X_{i,j} : 1 \leq i,j \leq g]$ be the polynomials which give respectively the $(1,2)$- and $(1,g/2+2)$-entries of the matrix $\underline X^t J \underline X$, that is,
\begin{align*}
   R(\underline X) & = \sum_{k=1}^{g/2} X_{k,1}X_{g/2+k,2} - \sum_{k=1}^{g/2} X_{g/2+k,1}X_{k,2},
\\ S(\underline X) & = \sum_{k=1}^{g/2} X_{k,1}X_{g/2+k,g/2+2} - \sum_{k=1}^{g/2} X_{g/2+k,1}X_{k,g/2+2}.
\end{align*}
Note that we use the hypothesis $g>2$ here to ensure that the $(1,2)$-entry of the $g \times g$ matrix $\underline X^t J \underline X$ is in the top left quadrant, while the $(1,g/2+2)$-entry exists.

For any values of the matrix entries $M_{1,2}, M_{1,g/2+2} \in \hat F$, there exist $\lambda, \mu \in \hat F$, not both zero, such that
\begin{equation} \label{eqn:lambda-mu-combination}
\lambda M_{1,2} + \mu M_{1,g/2+2} = 0.
\end{equation}
Let $Q = \lambda R + \mu S \in \hat F[X_{i,j} : 1 \leq i,j \leq g]$.
Thanks to \eqref{eqn:case-3-relation-hatH} and \eqref{eqn:lambda-mu-combination}, $Q$ gives a homogeneous relation of degree~$2$ between entries of $\underline{\hat H}_{\hat\sigma}$:
\begin{equation} \label{eqn:Q-hatH}
Q(\underline{\hat H}_{\hat\sigma}) = 0.
\end{equation}
The polynomial $Q$ is non-zero because $R, S$ are both non-zero, $R$ involves only terms $X_{k,i}X_{g/2+k,j}$ where $\{i,j\}=\{1,2\}$, and $S$ involves only terms $X_{k,i}X_{g/2+k,j}$ where $\{i,j\}=\{1,g/2+2\}$.
\end{proof}

Write $\hat P \in \hat F[Y_{ij}, Z_{ij} : 1 \leq i,j \leq g]$ for the polynomial obtained from $Q \in \hat F[X_{i,j} : 1 \leq i,j \leq g/2]$ by the substitutions
\begin{equation} \label{eqn:X-YZ}
\left.
\begin{aligned}
   X_{i,j} &= Y_{ij}
\\ X_{g/2+i,j} &= Z_{ij}
\end{aligned}
\right\}
\text{ for } 1 \leq i \leq g/2, \; 1 \leq j \leq g.
\end{equation}
In other words, the top half of the matrix~$\underline X$ is substituted by the top half of~$\underline Y$, and the bottom half of ~$\underline X$ is substituted by the top half of~$\underline Z$.
Let
\[ P_{\hat v}(\underline Y, \underline Z) = \hat P(A^t \underline Y + C^t \underline Z, B^t \underline Y + D^t \underline Z). \]
By \eqref{eqn:hatF}, \eqref{eqn:hatG} and the definition of~$\underline{\hat H}_{\hat\sigma}$, which follows those equations, we have
\[ P_{\hat v}(\underline F_{\hat v}, \underline G_{\hat v}) = \hat P(\underline{\hat F}_{\hat v}, \underline{\hat G}_{\hat v}) = Q(\underline{\hat H}_{\hat\sigma}) = 0. \]
Thus $P_{\hat v}$ gives a homogeneous relation of degree~$2$ between the $v$-adic periods $\underline F_{\hat v}$ and~$\underline G_{\hat v}$.
It only remains to check that this relation is not in the ideal~$I$.

\begin{lemma} \label{case-3-real-non-trivial}
$P_{\hat v}$ is not in the ideal~$I$.
\end{lemma}

\begin{proof}
First note that the substitution $\underline Y \mapsto A^t \underline Y + C^t \underline Z$, $\underline Z \mapsto B^t \underline Y + D^t \underline Z$ induces an automorphism~$\Phi$ of the polynomial ring $\hat F[Y_{ij}, Z_{ij} : 1 \leq i,j \leq g]$, because $\fullsmallmatrix{A}{B}{C}{D}$ is an invertible matrix.
Furthermore, since $e^{1/2}\fullsmallmatrix{A}{B}{C}{D}$ is symplectic, we have
\begin{align*}
  & \phantom{={}} (A^t \underline Y + C^t \underline Z)^t(B^t \underline Y + D^t \underline Z) - (B^t \underline Y + D^t \underline Z)^t(A^t \underline Y + C^t \underline Z)
\\& = \underline Y^t (AB^t - BA^t) \underline Y + \underline Y^t (AD^t - BC^t) \underline Z + \underline Z^t (CB^t - DA^t) \underline Y + \underline Z^t (CD^t - DC^t) \underline Z
\\& = e^{-1/2}(\underline Y^t \underline Z - \underline Z^t \underline Y).
\end{align*}
Multiplying the generators by a scalar $e^{1/2}$ does not change the ideal~$I$, so $\Phi(I) = I$.

Thus, it suffices to show that $\hat P = \Phi(P_{\hat v}) \not\in I$.
We shall prove this by constructing matrices $\underline y, \underline z \in \rM_{g \times g}(\hat F)$ such that every entry of $\underline y^t \underline z - \underline z^t \underline y$ is zero, while $\hat P(\underline y, \underline z) \neq 0$.

To that end, let
\[ \underline y = \fullmatrix{I_{g/2}}{0}{-I_{g/2}}{0}, \quad \underline z = \fullmatrix{\underline v}{\underline w}{\underline v}{\underline w}, \]
where $\underline v, \underline w \in \rM_{g/2 \times g/2}(\hat F)$ are to be chosen later.
These matrices do indeed satisfy $\underline y^t \underline z - \underline z^t \underline y = 0$.
Letting $\underline x$ be the matrix constructed from $\underline y$ and~$\underline z$ by the substitutions~\eqref{eqn:X-YZ}, we obtain
\[ \underline x = \fullmatrix{I_{g/2}}{0}{\underline v}{\underline w}. \]
We can calculate
\[ \underline x^t J \underline x = \fullmatrix{\underline v - \underline v^t}{\underline w}{-\underline w^t}{0}, \]
so
\[ \hat P(\underline y, \underline z) = \lambda R(\underline x) + \mu S(\underline x) = \lambda(v_{12} - v_{21}) + \mu w_{12}. \]
Since $\lambda$ and~$\mu$ are not both zero, we can choose $\underline v$ and $\underline w$ to make $\hat P(\underline y, \underline z)$ non-zero.
\end{proof}

\begin{remark} \label{just-likely-relations}
If we consider Case~3, with $\hat E \cong \QQ \times \QQ$ or a real quadratic field, and with $g=2$ (corresponding to a ``just likely intersection'' between a curve and a Hilbert modular surface in $\cA_2$), both sides of \eqref{eqn:case-3-relation-hatH} are skew-symmetric $2 \times 2$ matrices, so \eqref{eqn:case-3-relation-hatH} gives only a single independent inhomogeneous relation of the form
\[ \bigl( \text{$\ov\QQ$-polynomial in the entries of } \underline{\hat H}_{\hat\sigma} \bigr) = \frac{1}{2 \pi i} M_{ij}. \]
Because of the factor $1/2\pi i$, this is not a relation between the entries of $\underline{\hat H}_{\hat\sigma}$ with $\ov\QQ$-coefficients.
It can be checked that following the same method with $\hat\sigma$ replaced by~$\hat\tau$ results in an inhomogeneous relation which is equivalent to \eqref{eqn:case-3-relation-hatH} modulo the ideal of trivial relations~$I$.
Thus, when $g=2$, we cannot eliminate $\frac{1}{2\pi i}$ and obtain a $\ov\QQ$-relation between archimedean periods.
\end{remark}

\subsection{Primality of the ideal of trivial relations} \label{subsec:prime-ideal}

The proof that the ideal~$I$ is prime has two parts: we show that the subvariety~$V$ of $\gM_{g,\CC}^2$ defined by~$I$ is irreducible, and that $I$ is a radical ideal.

In order to show that $V$ is irreducible, we show that it is the Zariski closure of the projection of a variety $W$ which is easily seen to be irreducible.
This fact will also be used later, in \cref{subsec:trivial-relations}.
The variety $W \subset \gM_{2g,\CC}$ is defined as follows, for some $\mu \in \CC^\times$:
\begin{equation} \label{eqn:Theta}
W = \Bigl\{ M \in \gM_{2g,\CC} : M^t\fullsmallmatrix{0}{I_g}{-I_g}{0}M = \mu\fullsmallmatrix{0}{I_g}{-I_g}{0} \Bigr\}.
\end{equation}
(The number $\mu$ is not important in proving that~$I$ is irreducible, but it will be important when \cref{lem:dom} is applied in \cref{subsec:trivial-relations}.)
Since $\mu \neq 0$, $W$ is a $\gSp_{2g,\CC}$-torsor over $\Spec(\CC)$ and, in particular, it is irreducible.

If we view $2g \times 2g$ matrices as being made up of $g \times g$ blocks $\fullmatrix{\underline{Y}}{\underline{Y}'}{\underline{Z}}{\underline{Z}'}$,
then $W$ is equivalently described by the equations
\begin{gather}
   \underline{Y}^t \underline{Z} - \underline{Z}^t \underline{Y} = 0; \label{eqn:YtZ}
\\ (\underline{Y}')^t \underline{Z}' - (\underline{Z}')^t \underline{Y}' = 0;
\\ \underline{Y}^t \underline{Z}' - \underline{Z}^t \underline{Y}' = \mu I_g.
\end{gather}
Note that the entries of the matrix \eqref{eqn:YtZ} are precisely the generators of the ideal~$I$.

Let $\pi \colon \gM_{2g,\CC} \to \gM_{2g \times g,\CC}$ denote the projection
\[ \fullmatrix{\underline{Y}}{\underline{Y}'}{\underline{Z}}{\underline{Z}'}
\mapsto \begin{pmatrix} \underline{Y} \\ \underline{Z} \end{pmatrix}. \]

\begin{lemma} \label{lem:dom}
The morphism $\pi|_W:W \to V$ is dominant. In particular, $V$ is irreducible.
\end{lemma}

\begin{proof}
Let $V_0$ denote the Zariski open subset of $V$ defined by the condition that the columns in $\gM_{2g\times g}$ are linearly independent. 

We claim that $\pi(W)$ contains $V_0$. To see this, let $M \in V_0(\CC)$. Then, by the definition of $V_0$, the columns of $M$ span a maximal totally isotropic subspace of $\CC^{2g}$ for the standard symplectic form. We can therefore extend it to a symplectic basis of $\CC^{2g}$ and, after scaling by $\mu^{-1}$, the columns formed by these new vectors appended to $M$ yield a matrix $\tilde{M} \in W(\CC)$.

To complete the proof, it remains to show that $V_0$ is Zariski dense in $V$.
Let $M\in V(\CC)$ and let $w_1,\ldots,w_g\in\CC^{2g}$ denote its columns. By the definition of $V$, these vectors lie in a totally isotropic subspace of $\CC^{2g}$. Without loss of generality, suppose that $w_1,\ldots,w_r$ are linearly independent and that $w_{r+1},\ldots,w_g$ are in the span of $w_1,\ldots,w_r$. Choose $u_{r+1},\ldots, u_g$ such that $w_1,\ldots, w_r, u_{r+1},\ldots,
u_g$ form a basis for a maximal totally isotropic space of $\CC^{2g}$.
For any $t=(t_{r+1},\ldots, t_g) \in \CC^{g-r}$ and $i=r+1,\ldots,g$, define
$w_i'(t) = w_i + t_i u_i$.
If $t_{r+1},\ldots, t_g$ are all non-zero, then
$w_1,\ldots, w_r, w_{r+1}'(t),\ldots, w_g'(t)$
are linearly independent, which is to say that the matrix with these columns is in~$V_0$.
Therefore, since $(0,...,0)$ is in the Zariski closure of $\GG_m^{g-r}$ in $\AAA^{g-r}$, we conclude that $M$ belongs to the Zariski closure of~$V_0$.
\end{proof}

\begin{lemma}\label{lem:prime}
The ideal $I$ is prime.
\end{lemma}

\begin{proof}
Thanks to \cref{lem:dom}, the variety $V \subset \gM_{2g \times g}$ defined by~$I$ is irreducible.
Hence it suffices to show that $I$ is a radical ideal.

By definition, $I$ is generated in $\ov\QQ[\underline{Y},\underline{Z}]$ by the elements
\[
f_{ij}=\sum_{k=1}^gY_{ki}Z_{kj}-Z_{ki}Y_{kj},\quad 1\leq i,j\leq g.
\]
The Jacobian matrix whose entries are the partial derivatives of the $f_{ij}$ with respect to the $Y_{ij}$ and $Z_{ij}$ is the derivative of the map $\gM_g(\CC)^2\to\gM_g(\CC)$ given by $\underline{Y}^t \underline{Z} - \underline{Z}^t \underline{Y}$. As such, its evaluation at a point $(\underline{A}, \underline{B})$ is the linear map $\gM_g(\CC)^2\to\gM_g(\CC)$ given by $\underline{A}^t \underline{Z} + \underline{Y}^t \underline{B} - \underline{B}^t \underline{Y} - \underline{Z}^t \underline{A}$. For $\underline{A}=I$ and $\underline{B}=0$, this becomes $\underline{Z}-\underline{Z}^t$, the image of which is precisely the set of skew symmetric matrices. In particular, the rank of the Jacobian matrix at this point is $g(g-1)/2$.

On the other hand, since $f_{ji}=-f_{ij}$ for all $i,j$, $I$ is generated by $g(g-1)/2$ elements.
Now the result follows from \cref{prop:martins-fact}, below. 
\end{proof}

\begin{proposition}\label{prop:martins-fact}
Let $k$ be a perfect field and let $V$ be an irreducible algebraic subvariety of $\AAA^n_k$. Suppose that $V=V(I)$ for some ideal $I=(f_1,\ldots,f_m)$ of $k[X_1,\ldots,X_n]$. If the Jacobian matrix $(\partial f_i/\partial X_j)_{ij}$ has rank at least~$m$ at some point of $V$, then $I$ is radical.
\end{proposition}

\begin{proof}
Let $R = k[X_1, \dotsc, X_n]/I$.

Since the rank of the Jacobian matrix is~$m$ at some point of~$V$, by \cite[Theorem 16.19]{eisenbud1995}, $\codim(V) \geq m$.
On the other hand, since $I$ is generated by $m$ elements, $\codim(V) \leq m$.
Thus $\codim(V) = m$.
By \cite[Theorem 18.13]{eisenbud1995}, we deduce that $R$ is Cohen-Macaulay.

Let $J \subset R$ denote the ideal generated by the $m \times m$ minors of the Jacobian matrix, taken modulo~$I$.
By our hypothesis that the Jacobian matrix has rank at least~$m$ at some point of $V$, $V(J)$ is a proper subvariety of~$V$.  Hence $J$ has codimension at least~$1$ in~$R$.

Therefore, by \cite[Theorem 18.15 (a)]{eisenbud1995}, $R$ is reduced.
In other words, $I$ is a radical ideal.
\end{proof}

We thank one of the anonymous referees for suggesting the following remark.

\begin{remark}
\Cref{lem:prime} and \cref{prop:martins-fact} are not essential for the proof of \cref{exist-relations}(c).
Indeed, the proofs of \cref{non-arch-non-trivial,case-3-real-non-trivial} show that, in the cases which they deal with, the polynomials $P_{ij}$ and $P_{\hat v}$, respectively, are not in the radical ideal~$\sqrt{I}$.
One can likewise verify that the polynomials from \cite[X, Construction~2.4.1]{And89} are not in~$\sqrt{I}$.
Hence \cref{lem:dom} suffices to establish that the product $P = P_{fin} \cdot \prod_{\hat v}  P_{\hat v}$ is not in~$\sqrt{I}$; \textit{a fortiori} it is not in~$I$.
\end{remark}

\section{Proof of height and Galois bounds} \label{sec:main-proofs}

In this section, we use the tools developed in Sections \ref{sec:prelims}--\ref{sec:relations} to prove our main results, as stated in \cref{intro}. First, we prove some auxiliary results. 

\subsection{Auxiliary results} 

As in \cite[Definition 8.3.1]{Liu}, we define a \defterm{fibered surface} to be a flat projective morphism $X\to S$ where $X$ is an integral scheme of dimension~$2$ and $S$ is a Dedekind scheme (which, for us, includes the condition $\dim S=1$). We refer to a fibered surface $X\to S$ as regular if $X$ is a regular scheme. As in
\cite[Definition 10.1.1]{Liu}, if $C$ is a smooth connected projective curve over a number field $K$, we define a \defterm{(regular) $\cO_K$-model} of $C$ to be a (regular) fibered surface $\fC\to\Spec(\cO_K)$ equipped with an isomorphism $\fC_K \to C$. By \cite[Proposition 10.1.8]{Liu}, such a model always exists. If $\fC$ and $\fC'$ are two $\cO_K$-models of $C$, we say that a morphism $\fC\to\fC'$ is a \defterm{morphism of $\cO_K$-models} if it restricts to the identity on $C$. 
 
The following lemma will be used repeatedly in the sections to follow. 

\begin{lemma}\label{lem:ext}
Let $K$ be a number field, and let $f:C_2\to C_1$ denote a morphism of smooth connected projective curves over $K$. Let $\fC_1$ (resp. $\fC'_2$) denote an $\cO_K$-model of $C_1$ (resp. $C_2$). Then there exists a regular $\cO_K$-model $\fC_2$ of $C_2$ such that (i) there is a morphism of $\cO_K$-models $\fC_2\to\fC'_2$ and (ii) $f \colon C_2 \to C_1$ extends to a morphism $\fC_2\to\fC_1$.
\end{lemma}

\begin{proof}
The restriction of~$f$ to the generic point of~$\fC'_2$ gives rise to a rational $\cO_K$-map $\fC'_2 \dashrightarrow \fC_1$ [EGA~I, Prop.~7.1.11].
Let $U \subset \fC'_2$ be the domain of definition of this rational map.
Thus $U$ is a Zariski open subset of $\fC'_2$ and there is a morphism $f_U \colon U \to \fC_1$ whose restriction to the generic point of $\fC'_2$ is the same as that of~$f$.
By \cite[Prop.~4.1.16]{Liu}, $U$ contains all codimension~$1$ points of $\fC'_2$.  Thus $C_2 \subset U$, and $f_U|_{C_2} = f$.

Let $\Gamma \subset U\times\fC_1$ denote the graph of $f_U$, and let $\overline{\Gamma}$ denote its Zariski closure in $\fC'_2\times\fC_1$ (with the reduced scheme structure). Then $\pi:\overline{\Gamma}\to\Spec(\cO_K)$, obtained from restricting the structure morphism $\fC'_2\times\fC_1\to\Spec(\cO_K)$, is a fibered surface. (Indeed, it is clearly integral and projective, and flatness follows from the surjectivity of $\pi$, as noted in \cite[Definition 8.3.1]{Liu}.)

The generic fibre of $\ov{\Gamma}$ is the graph of~$f$, which is isomorphic to $C_2$ and hence smooth. Therefore, by \cite[Corollary 8.3.51]{Liu}, $\ov{\Gamma}$ admits a desingularization in the strong sense. In other words, we obtain a proper birational morphism $\fX\to\ov{\Gamma}$, with $\fX$ a regular scheme, which is an isomorphism above every regular point of $\ov{\Gamma}$. In particular, the 
generic fibre of $\fX \to \ov\Gamma \to \Spec(\cO_K)$ is isomorphic to $C_2$. Furthermore, $\fX$ is integral and, by \cite[Exercise 8.3.27(a)]{Liu}, $\fX\to\Spec(\cO_K)$ is projective. Since $\fX\to\Spec(\cO_K)$ is flat by surjectivity, we deduce that $\fC_2:=\fX$ is a regular $\cO_K$-model of $C_2$.

The composition of $\fC_2 \to \ov\Gamma$ with the projection $\fC'_2\times\fC_1\to\fC'_2$ yields a morphism of $\cO_K$-models $\fC_2\to\fC'_2$.
Meanwhile, the composition of $\fC_2\to\ov{\Gamma}$ with the projection $\fC'_2\times\fC_1\to\fC_1$ extends $f \colon C_2 \to C_1$.
\end{proof}

We derive the following, which will be used in \cref{sec:G-funcs}.

\begin{lemma}\label{lem:ext2}
    Let $K$ be a number field. Let $C$ be a smooth connected projective curve over $K$ and let $\fC''$ denote a regular $\cO_K$-model of $C$. Let $x:C\to\PP^1$ and $\sigma_k:C\to C$ (for $k=1,\ldots,n$) be morphisms over $K$. Then there exist regular $\cO_K$-models $\fC$ and $\fC'$ of $C$ and a morphism $\fC\to\fC''$ such that $x$ extends to a morphism $\fC\to\PP^1_{\cO_K}$ and the $\sigma_k$ extend to morphisms $\fC'\to\fC$.
\end{lemma}

\begin{proof}
First, we apply \cref{lem:ext} with $f=x$, $C_2=C$, $C_1=\PP^1$, $\fC'_2=\fC''$, and $\fC_1=\PP^1_{\cO_K}$, thereby yielding an extension $\fC\to\PP^1_{\cO_K}$ of $x$, for some regular $\cO_K$-model $\fC$ of $C$, and a morphism $\fC\to\fC''$ of $\cO_K$-models. Next, we apply \cref{lem:ext} with $f=\sigma_1$, $C_2=C_1=C$, and $\fC'_2=\fC_1=\fC$, thereby yielding an extension $\tilde{\sigma}_1:\fC_1\to\fC$ of $\sigma_1$, for some regular $\cO_K$-model $\fC_1$ of $C$, and a morphism $\iota_1:\fC_1\to\fC$ of $\cO_K$-models. Next, we we apply \cref{lem:ext} with $f=\sigma_2$, $C_2=C_1=C$, $\fC'_2=\fC$, and $\fC_1=\fC_1$, thereby yielding an extensions $\tilde{\sigma}_2:\fC_2\to\fC_1$ of $\sigma_2$, for some regular $\cO_K$-model $\fC_2$ of $C$, and a morphism $\iota_2:\fC_2\to\fC$ of $\cO_K$-models. Composing, we obtain morphisms
\begin{align*}
  & \fC_2\xrightarrow{\iota_2}\fC_1\xrightarrow{\tilde{\sigma}_1}\fC,
\\& \fC_2\xrightarrow{\tilde{\sigma}_2}\fC_1\xrightarrow{\iota_1}\fC
\end{align*}
extending $\sigma_1$ and $\sigma_2$ respectively.
Iterating this procedure yields the result.   
\end{proof}

We will also use the following lemma in \cref{sec:G-funcs}.

\begin{lemma}\label{lem:point}
Let $K$ be a number field and let $\fC\to\Spec (O_K)$ be a fibered surface. Let $s$ denote a closed point $\Spec (K)\to C$ of the generic fibre and let $S$ denote the Zariski closure of $s$ in $\fC$. Then there exists an isomorphism $\Spec(\cO_K)\to S$ the composition of which with $\Spec (K)\to\Spec(\cO_K)$ is $s$.
\end{lemma}

\begin{proof}
This is an easy consequence of Zariski's Main Theorem \cite[Cor. 4.4.9]{EGAIII}. The following is an alternative, elementary proof. 

By \cite[Proposition 8.3.4]{Liu}, the structure morphism $S\to\Spec(\cO_K)$ is finite and surjective. In particular, $S$ is affine, isomorphic to $\Spec (A)$, say. Therefore, the morphisms $\Spec (K)\to S\to\Spec(\cO_K)$ correspond to the ring homomorphisms $\cO_K\to A\to K$. Since $\Spec (K)\to S$ is dominant and $A$ is reduced, $A\to K$ is injective. However, since $S\to\Spec(\cO_K)$ is finite, $A$ is an integral extension of~$\cO_K$. Therefore, $\cO_K\to A$ is an isomorphism. The second claim is now obvious. 
\end{proof}

\subsection{Rewriting Theorem \ref{main-bound} in terms of semiabelian schemes}

In order to prove \cref{main-bound}, we claim that it suffices to establish the following.

\begin{theorem}\label{scheme-bound}
Let $C$ be a smooth projective geometrically irreducible algebraic curve over a number field~$K$.
Let $G\to C$ be a semiabelian scheme of relative dimension $g\geq 2$.

Suppose there is a dense open subset $C^* \subset C$ and a point $s_0 \in C(K) \setminus C^*(K)$ such that $G^* := G|_{C^*} \to C^*$ is a principally polarised abelian scheme and $G_{s_0} \cong \GG_{m,K}^g$.
Suppose further that the map $m:C^*\to\cA_g$ to the coarse moduli space is non-constant, and
that $m(C^*)$ is Hodge generic in~$\Ag$.
    
Then there exist positive constants $\newC{scheme-bound-mult}$ and $\newC{scheme-bound-exp}$ such that, for every point $s\in C(\Qbar)$ for which $G_s$ has unlikely endomorphisms, we have
\[[K(s):K]\geq\refC{scheme-bound-mult}\abs{\disc(\End(G_s))}^{\refC{scheme-bound-exp}}.\]
\end{theorem}

In order to see that \cref{scheme-bound} implies \cref{main-bound}, let $C$ be as in \cref{main-bound}. By \cite[Proposition 9.4]{ExCM}, we obtain a smooth projective geometrically irreducible curve $\tilde{C}'$,
an open subset $\tilde{C}\subset\tilde{C}'$,
a point $\tilde{s}_0\in\tilde{C}'\setminus\tilde{C}$,
a finite surjective morphism $\tilde{C}\to C$, and
a semiabelian scheme $G'\to \tilde{C}'$
such that $G'|_{\tilde{C}}$ is a principally polarised abelian scheme of relative dimension $g$, the map $\tilde{C}\to\cA_g$ associated with $G'|_{\tilde{C}}\to\tilde{C}$ is the composition $\tilde{C}\to C\hookrightarrow\cA_g$, and the fibre $G'_{\tilde{s}_0}$ is a torus.

These data are defined over a number field~$K$.
After replacing $K$ by a finite extension, we may assume that the torus $G'_{\tilde{s}_0}$ is split.
As such, we find ourselves in the situation of Theorem \ref{scheme-bound} (choosing $\tilde{C}'$ for~$C$, $\tilde{C}$ for $C^*$ and $\tilde{s}_0$ for $s_0$).
If $s\in C(\Qbar)$ is such that $A_s$ has unlikely endomorphisms, then $G_{\tilde{s}}$ will have unlikely endomorphisms for any preimage $\tilde{s}$ of~$s$ under the morphism $\tilde{C} \to C$.
The degree of the field extension $[K(\tilde{s}):\QQ(s)]$ is bounded independently of~$s$.
We conclude that Theorem \ref{scheme-bound} implies Theorem \ref{main-bound}. 

\subsection{Reduction of Galois bound to height bound}

In order to prove \cref{scheme-bound}, we claim that it suffices to establish the following.

\begin{theorem}\label{height-bound}
Consider the situation in \cref{scheme-bound} and let $h$ denote a Weil height on $C$. Then there exist positive constants $\newC{height-bound-mult}$ and $\newC{height-bound-exp}$ such that, for any point $s\in C^*(\Qbar)$ for which $G_s$ has unlikely endomorphisms, we have
\[h(s)\leq \refC{height-bound-mult}[K(s):K]^{\refC{height-bound-exp}}.\]
\end{theorem}

The proof that \cref{height-bound} implies \cref{scheme-bound} proceeds exactly as in the latter half of the proof of \cite[Theorem~6.5]{QRTUI}.
The key ingredient is a theorem, due to Masser and Wüstholz \cite{MW:endo}, which implies that there exist positive constants $\newC{MW-mult}$ and $\newC{MW-exp}$ such that, for any point $s\in C^*(\Qbar)$, we have
\[\abs{\disc(\End(G_s))}\leq\refC{MW-mult}\max\{h_F(G_s),[K(s):K]\}^{\refC{MW-exp}}.\]
Here $h_F(G_s)$ denotes the semistable Faltings height.
Furthermore, by a result of Faltings, we have 
\[|h_F(G_s)-h(s)|=O(\log h(s)).\] 
\Cref{scheme-bound} now follows immediately from \cref{height-bound}.   

\subsection{Proof of Theorem \ref{height-bound}: integral models}

\begin{lemma} \label{height-bound-integral-models}
In order to prove \cref{height-bound}, it suffices to prove it under the additional assumption that there exists a regular $\cO_K$-model $\fC$ for~$C$, a semiabelian scheme $\fG \to \fC$ and a non-constant rational function $x \in K(C)$ with the following properties:
\begin{enumerate}[(i)]
\item $\fG_K \cong G$ as a $C$-group scheme;
\item every point $s \in C(\ov K)$ such that $x(s)=0$ is a simple zero of~$x$, lies in $C(K)$, and satisfies $G_s \cong \GG_{m,K}^g$;
\item $x \colon C \to \PP^1$ is Galois (that is, $\Aut_x(C) := \{ \sigma \in \Aut(C) : x \circ \sigma = x \}$ acts transitively on each $\Qbar$-fibre of $x$).
\end{enumerate}
\end{lemma}

\begin{proof}
Let $C$, $C^*$, $G$, $G^*$, $s_0$ and $K$ denote the data referred to in the statement of Theorem \ref{scheme-bound}.

The proof has two stages: first we construct integral models of $C$ and~$G$ using Gabber's lemma, then we use \cite[Lemma~5.1]{Y(1)} to construct a suitable rational function~$x$.
Both stages require replacing~$C$ by a finite cover, and extending the base field~$K$.
The full proof is summarised in a commutative diagram at the end.

Let $\fC$ denote a regular $\cO_K$-model of $C$. Spreading out, $G$ extends to a semiabelian scheme $\fG_N$ over the base change $\fC_N$ of $\fC$ to $\cO_K[\frac{1}{N}]$, for some $N\in\bN$. Furthermore, for some open subset $\fC_N^*$ of $\fC_N$, the restriction $\fG_N^* := \fG_N|_{\fC_N^*}$ is an abelian scheme.
Now we are in a position to apply Gabber's Lemma (see \cite[Proposition 1.16]{deligne83}). 

We use the version of Gabber's Lemma stated in \cite{conrad:gabber}.
In the notations of \cite[Theorem 3.1]{conrad:gabber}, we set $S=\fC$, $X=\fC^*_N$, and $A=\fG^*_N$. Hence, we obtain a proper surjection $\tilde{\fX}^*\to\fC^*_N$ and an open immersion $\tilde{\fX}^*\hookrightarrow\tilde{\fX}$ into a proper $\fC$-scheme $\tilde{\fX}$ such that the pullback abelian scheme $\tilde{\fG}^* \to \tilde{\fX}^*$ of $\fG^*_N \to \fC^*_N$ extends to a semiabelian scheme $\tilde{\fG}\to\tilde{\fX}$. By Chow's Lemma, we can and do assume that $\tilde{\fX} \to \Spec (\cO_K)$ is projective.

Let $\tilde{C}$ denote a closed (geometrically) irreducible curve in $(C\times_{\fC}\tilde{\fX})_{\Qbar}$ not contained in a fibre of the projection to $C_{\Qbar}$. Then $\tilde{C}$ is projective and $\tilde{C}\to C_{\Qbar}$ is surjective.
Let $\tilde{s}_0\in\tilde{C}$ denote a lift of $s_0\in C$.
Let $L$ denote a finite extension of $K$ over which $\tilde{C}$ and $\tilde{s}_0$ are defined.
Let $\tilde{C}^*:=C^*_L\times_{C_L}\tilde{C}$, which is an open subvariety of $\tilde{C}$. 

Let $\tilde\fC$ denote the Zariski closure of $\tilde{C}$ in $\tilde{\fX} \times_{\Spec (\cO_K)} \Spec (\cO_L)$.
This is an $\cO_L$-model of $\tilde{C}$.
(Note that $\tilde{C} \to \Spec (\cO_L)$ is projective and its image contains the generic point of $\Spec (\cO_L)$, so it is surjective, hence flat as observed in \cite[Def.~8.3.1]{Liu}.)


Applying \cite[Lemma 5.1]{Y(1)} to $\tilde C$, after possibly replacing $L$ by a finite extension, we obtain a smooth projective geometrically irreducible
\footnote{\cite[Lemma~5.1]{Y(1)} erroneously says that $C$ ($\tilde C$ in our application) and $C_4$ are irreducible, when it should say that both are geometrically irreducible.}
algebraic curve $C_4$ over~$L$,
a non-constant morphism $\nu:C_4\to \tilde C$, and
a non-constant rational function $x\in K(C_4)$
such that 
\begin{itemize}
\item every point $s \in C_4(\Qbar)$ such that $x(s)=0$ is a simple zero of $x$ and satisfies $\nu(s)=\tilde{s}_0$;
\item $x \colon C_4 \to \PP^1$ is Galois.
\end{itemize}
After replacing $L$ by a further finite extension, we may assume that every zero of~$x$ in $C_4(\Qbar)$ lies in~$C_4(L)$.

By \cref{lem:ext}, there exists a regular $\cO_L$-model $\fC_4$ of $C_4$ such that $\nu \colon C_4 \to \tilde{C}$ extends to a morphism $\fC_4\to\tilde{\fC}$. We let $\fG_4$ denote the pullback of $\tilde{\fG}$ along $\fC_4 \to \tilde{\fC}$, and we let $G_4$ denote the restriction of $\fG_4$ to $C_4$.

Let $C^*_4 = \tilde{C}^* \times_{\tilde{C}} C_4$, which is an open subvariety of~$C_4$.
Let $G^*_4 = G_4|_{C^*_4}$, which is an abelian scheme over $C^*_4$.

Choose a point $s_1 \in C_4(L)$ such that $x(s_1)=0$.
Since $\nu(s_1) = \tilde{s}_0$, we have $G_{4,s_1} \cong \GG_{m,L}^g$.

Now $C_4$, $C_4^*$, $G_4$, $G_4^*$, $s_1$ and~$L$ satisfy the conditions of \cref{scheme-bound}, and it suffices to prove \cref{height-bound} with $C_4$, $C_4^*$, $G_4$, $G_4^*$, $s_1$ and~$L$ in place of $C$, $C^*$, $G$, $G^*$, $s_0$ and~$K$.
Furthermore, $\fC_4$ and $\fG_4$ provide integral models of $C$ and~$G$ respectively, as required by \cref{height-bound-integral-models}, while $x \in K(C_4)$ satisfies conditions (ii) and~(ii) of \cref{height-bound-integral-models}.

We show all the objects in this proof in the following commutative diagram (we use hook arrows to designate open immersions on the base schemes).
 
\begin{center}
\begin{tikzcd}
& G^* \arrow[r] \arrow[d]                      
& G \arrow[rrr, "\mathrm{spreading\ out}" description] \arrow[d]
& 
&
& \fG_N \arrow[d, "\mathrm{semiabelian}" description]
& \fG^*_N \arrow[d, "\mathrm{abelian}" description] \arrow[l]                                     \\
& C^* \arrow[r, hook]
& C \arrow[rr] \arrow[rrr, bend left]
&
& \mathfrak{C}
& \fC_N \arrow[l, hook']
& \fC^*_N \arrow[l, hook']                                                                        \\
& \tilde{C}^* \arrow[u] \arrow[r, hook]
& \tilde{C} \arrow[u] \arrow[r]
& \tilde{\fC} \arrow[r]
& \tilde{\fX} \arrow[u, "\mathrm{proper}" description]
&
& \tilde{\fX}^* \arrow[u, "\mathrm{proper\ surjective}" description] \arrow[ll, hook']            \\
  C_4^* \arrow[r, hook] \arrow[ru]
& C_4 \arrow[ru, "\nu" description] \arrow[r]
& \fC_4 \arrow[ru]
& \tilde{\fG}|_{\tilde{\fC}} \arrow[r] \arrow[u]
& \tilde{\fG} \arrow[u, "\mathrm{semiabelian}" description]
&
& {\tilde{\fG}^*} \arrow[u, "\mathrm{abelian}" description] \arrow[uuu, bend right=49] \arrow[ll]
\\
  G_4^* \arrow[u] \arrow[r]
& G_4 \arrow[u] \arrow[r]
& \fG_4 \arrow[u] \arrow[ru] 
\end{tikzcd}
\end{center}
The diagram should be read in the following order:
the upper-left square represents the data of \cref{scheme-bound}. The upper-right portion represents spreading out to $\cO_K[\frac{1}{N}]$. The lower-right rectangles represent the application of Gabber's Lemma.
Finally the lower-left squares and diagonal arrows represent the application of \cite[Lemma~5.1]{Y(1)} and \cref{lem:ext}.
\end{proof}

\subsection{Proof of Theorem \ref{height-bound}: G-functions}\label{sec:G-funcs}

Suppose that we are in the situation of \cref{height-bound}, under the additional assumptions of \cref{height-bound-integral-models}.
Let $s_1,\ldots,s_n\in C(K)$ denote the zeroes of $x$.
By condition~(iii) of \cref{height-bound-integral-models}, for each $k=1,\ldots,n$, there is an automorphism $\sigma_k\in\Aut_x(C)$ such that $\sigma_k(s_1)=s_k$. (In particular, $\sigma_1$ is the identity.) As in \cite{Y(1)}, rather than analysing the periods of $G$ near each of the points $s_1,\ldots,s_{n}$, we will analyse the periods of $\sigma^*_kG$ near $s_1$ for $k=1,\ldots,n$. 

By \cref{lem:ext2} (with $\fC''=\fC$), after possibly replacing $\fC$ and pulling back, we can and do assume that $x:C\to\PP^1$ extends to a morphism $\fC\to\PP^1_{\cO_K}$ and each automorphism $\sigma_k:C\to C$ extends to a morphism $\fC'\to\fC$ for $\fC'$ another regular $\cO_K$-model of $C$. We let $\fC_0$ denote the Zariski closure of the image of the morphism $\Spec (K)\to C\to\fC'$ associated with $s_1\in C(K)$ (with the reduced subscheme structure). By \cref{lem:point}, $\fC_0$ is the image of a section $\Spec(\cO_K)\to\fC'$ as in the setup of \cref{sec:formal-periods}.
For $k=1,\ldots,n$, we let $\fG^{(k)}$ denote the pullback of the semiabelian scheme~$\fG$ along $\fC'\xrightarrow{\sigma_k}\fC$,
and we let $\fC'_{\formal}$ denote the formal completion of $\fC'$ along~$\fC_0$. We summarise these constructions in the following diagram.

\medskip
\begin{center}
\begin{tikzcd}
                & \fG^{(k)}_{\formal} \arrow[r] \arrow[d] & \fG^{(k)} \arrow[r] \arrow[d] & \fG \arrow[d]      &               \\
\fC_0 \arrow[r] & \fC'_{\formal} \arrow[r, hook]          & \fC' \arrow[r, "\sigma_k"]    & \fC \arrow[r, "x"] & \PP^1_{\cO_K}
\end{tikzcd}
\end{center}
\medskip

After possibly shrinking $C^*$, we may assume that $\sigma_k(C^*) = C^*$ for all~$k$, with $C^*$ still being an open subset of $C$ (indeed, $C^*$ will be the complement of finitely many closed $\Aut_x(C)$-orbits in~$C$).
Let $C' = C^* \cup \{s_1\}$, which is a Zariski open subset of~$C$.

We let $G^{(k)}$ denote the pullback of $G$ along $C' \hookrightarrow C \xrightarrow{\sigma_k} C$.
Since $\sigma_k(C^*) \subset C^*$, $G^{(k)}|_{C^*} \to C^*$ is a principally polarised abelian scheme for each~$k$.
Furthermore, by condition~(ii) of \cref{height-bound-integral-models}, $G^{(k)}_{s_1} \cong G_{s_k} \cong \GG_{m,K}^g$. 

Recall the notation $H^1_{DR}(G^*/C^*)^\can$ from section~\ref{subsec:formal-periods-statement}.
The following lemma will be used in Section \ref{subsec:trivial-relations}.

\begin{lemma} \label{omega-eta-basis}
We can replace $C^*$ with a Zariski open subset and $C'$ with the new $C^* \cup \{s_1\}$ such that, for every $k=1,\dotsc,n$, there exist 
\[ \omega_1^{(k)}, \dotsc, \omega_g^{(k)} \in \Omega^\inv_{G^{(k)}/C'}(C')\text{ and }\eta_1^{(k)}, \dotsc, \eta_g^{(k)} \in H^1_{DR}(G^{(k)*}/C^*)^\can(C') \]
satisfying the following conditions:
\begin{enumerate}[(i)]
\item the restrictions $\omega_1^{(k)}|_{C^*}, \dotsc, \omega_g^{(k)}|_{C^*}, \eta_1^{(k)}|_{C^*}, \dotsc, \eta_g^{(k)}|_{C^*}$ form an $\cO_{C^*}$-basis for $H^1_{DR}(G^{(k)*}/C^*)$;
\item with respect to this basis, the symplectic form $\psi$ on $H^1_{DR}(G^{(k)*}/C^*)$ induced by the polarisation is represented by a matrix of the form $e_k\fullsmallmatrix{0}{I}{-I}{0}$, for some $e_k \in \cO(C^*)^\times$.
\end{enumerate}
\end{lemma}

\begin{proof}
Since $H^1_{DR}(G^{(k)}_{K(C)}/K(C))$ is a vector space over $K(C)$, and its subspace $\Omega^\inv_{G^{(k)}_{K(C)}/K(C)}$ is isotropic with respect to~$\psi$, we can choose a symplectic basis \[\cB:=\omega_1^{(k)\prime}, \dotsc, \omega_g^{(k)\prime}, \eta_1^{(k)\prime}, \dotsc, \eta_g^{(k)\prime}\subset H^1_{DR}(G^{(k)}_{K(C)}/K(C))\]
with the property that $\omega_1^{(k)\prime}, \dotsc, \omega_g^{(k)\prime}\in \Omega^\inv_{G^{(k)}_{K(C)}/K(C)}$.
After replacing $C^*$ by a Zariski open subset, we may assume that $\cB$ forms a $\cO_{C^*}$-basis for $H^1_{DR}(G^{(k)*}/C^*)$.

After shrinking $C^*$ (and thus $C'$) again, we may assume that $H^1_{DR}(G^{(k)*}/C^*)^\can$ is a free $\cO_{C'}$-module.
Then $H^1_{DR}(G^{(k)*}/C^*)^\can(C')$ is an $\cO(C')$-lattice in the $\cO(C^*)$-module $H^1_{DR}(G^{(k)*}/C^*)(C^*)$, and $\cO(C^*) = \cO(C')[1/x]$.
Hence we can choose $N_k \in \ZZ$ such that $x^{N_k}\omega_1^{(k)\prime}, \dotsc, x^{N_k}\omega_g^{(k)\prime}, x^{N_k}\eta_1^{(k)\prime}, \dotsc, x^{N_k}\eta_g^{(k)\prime}$ lie in $H^1_{DR}(G^{(k)*}/C^*)^\can(C')$.
Taking $\omega_j^{(k)} = x^{N_k}\omega_j^{(k)\prime}$ and $\eta_j^{(k)} = x^{N_k}\eta_j^{(k)\prime}$ proves~(i).

With respect to this basis, $\psi$ is represented by the matrix $e_k\fullsmallmatrix{0}{I}{-I}{0}$, where $e_k = x^{2N_k}$.
Since $G^* \to C^*$ is an abelian scheme, while $G_s \cong \GG_m^g$ for all zeroes $s$ of~$x$ in~$C$, $x$ has no zeroes on~$C^*$.  After shrinking $C^*$ again, we may assume that $x$ has no poles on~$C^*$.
Then $e_k \in \cO(C^*)^\times$, proving~(ii).
\end{proof}

We replace $C^*$ as described in Lemma \ref{omega-eta-basis}.  After possibly further shrinking~$C^*$, we may assume that $\sigma_k(C^*)=C^*$ for all~$k=1,\ldots,n$.  We then replace $C'$ by the new $C^* \cap \{s_1\}$.

Since $m(C)$ is Hodge generic in $\Ag$, the abelian scheme $G^{(k)} \to C'$ satisfies the conditions of \cref{gauss-manin-basis}.
Hence, by \cref{gauss-manin-basis}, the $\cO(C^*)$-$\nabla_{d/dx}$-module generated by $\omega_1^{(k)}|_{C^*}, \dotsc, \omega_g^{(k)}|_{C^*}$ is all of $H^1_{DR}(G^{(k)^*}/C^*)$. In particular, it contains $\eta_1^{(k)}, \dotsc, \eta_g^{(k)}$.
Thus, $\omega^{(k)}_{1},\ldots,\omega^{(k)}_{g},\eta^{(k)}_{1},\ldots,\eta^{(k)}_{g}$ satisfy the requirements of \cref{exist-formal-periods} and \cref{exist-relations}.

We let $\tilde{\omega}_{1},\ldots,\tilde{\omega}_{g}$ denote the standard $\cO_{\fC'}$-basis for $\Omega^{\inv}_{\GG^g_m\times\fC'/\fC'}$ as in \cref{exist-formal-periods}.

By \cite[Example 10.1/5]{BLR90}, since $C\cap \fC_0=\{s_1\}$ and $G^{(k)}_{s_1} \cong \GG_{m,K}^g$, the semiabelian scheme $\fG^{(k)}\times_{\fC'}\fC_0$ is a split torus of dimension $g$ over $\fC_0\cong\Spec(\cO_K)$. Therefore, let $\phi^{(k)}_\formal \colon \GG_m^g \times \fC'_\formal \to \fG^{(k)}_\formal$ be a formal uniformisation of $\fG^{(k)}_\formal$ over~$\fC'_\formal$.
For each place $v$ of~$K$, let $\phi^{(k)}_v$ denote the $v$-adic uniformisation of $\fG^{(k)}$ compatible with~$\phi^{(k)}_\formal$, defined over an open set $\cC_v^{(k)} \subset C^\van$, as defined in~\cref{subsec:uniformisations-places}.

For each $k\in\{1,\ldots,n\}$, we apply \cref{exist-formal-periods} with $R=\cO_K$, $\fC=\fC'$, $\fC_0$ and $C^*$ as defined above, $s_0=s_1$, $\fG=\fG^{(k)}$, $x=x\circ\sigma_1$, $\omega_i=\omega^{(k)}_{i}$, $\eta_i=\eta^{(k)}_{i}$, $\phi_\formal=\phi^{(k)}_\formal$ and $\phi_v=\phi^{(k)}_v$. Hence, we obtain G-functions $F^{(k)}_{ij},G^{(k)}_{ij}\in \powerseries{K}{X}$ for $i,j=1,\ldots,g$ and, for each place $v$ of $K$, admissible open sets $U^{(k)}_{v}\subset C'^{\van}$ and real numbers $r^{(k)}_{v}$ such that
\begin{enumerate}
\item $r^{(k)}_v > 0$ for all~$v$, and $r^{(k)}_v = 1$ for almost all~$v$;
\item $s_1^\van \in U_v^{(k)} \subset \cC_v^{(k)}$;
\item $x^{\van}$ maps $U^{(k)}_{v}$ isomorphically onto the open disc $D(r^{(k)}_{v},K_v)$;
\item the $v$-adic radii of convergence of $F^{(k)}_{ij}$ and $G^{(k)}_{ij}$ are at least $r^{(k)}_{v}$ for all $i$ and~$j$;
\item for every place $v$ of~$K$, for every finite extension of complete fields $L/K_v$ and for all $s \in U^{(k)*}_{v}(L)$, we have
\begin{align*}
   \phi_{v,s}^{(k)*} \omega_{i,s}^{(k)\van} & = \sum_{j=1}^g F^{(k)\van}_{ij}(x^\van(s)) \tilde\omega_{j,s}^{\van} \text{ for } 1 \leq i \leq g,\\
   \phi_{v,s}^{(k)*} \eta_{i,s}^{(k)\van} & = \sum_{j=1}^g G^{(k)\van}_{ij}(x^\van(s)) \tilde\omega_{j,s}^{\van} \text{ for } 1 \leq i \leq g,
\end{align*}
in the analytic de Rham cohomology $H^1_{DR}((\GG^\van_m)^g/L)$.
\end{enumerate}

\subsection{Proof of Theorem \ref{height-bound}: global relations}

This section follows the blueprint of \cite[sec. 5.H.]{Y(1)}. As in \cite{Y(1)}, some of the group schemes $G^{(k)}$ may be generically isogenous to each other.  We must pick one of the schemes $G^{(k)}$ from each generic isogeny class, in order to eliminate trivial relations between their periods near $s_1$. To that end, let $\bar{\eta}$ denote a geometric generic point of $C$, and define an equivalence relation $\sim$ on $\{1,\ldots,n\}$ by \[ k\sim k'\text{ if there exists an isogeny }G^{(k)}_{\bar{\eta}}\to G^{(k')}_{\bar{\eta}}.\]
We let $\Lambda$ denote a set of representatives for $\sim$ and, for any $k=1,\ldots,n$, we let $[k]$ denote the unique element in $\Lambda$ such that $k\sim[k]$. By abuse of notation, for superscripts, we will write $[k]$ instead of $([k])$.

Consider the set of $G$-functions
\[\cG=\{F^{(k)}_{ij},G^{(k)}_{ij}:k\in\Lambda\text{ and }i,j=1,\ldots,g\}\]
and, for any place $v$ of $K$, define $R_v$ to to be $\min(R(F^v):F\in\cG)$. We will construct a global relation between the elements of $\cG$ at $x(s)$ for any $s\in C^*(\Qbar)$ for which $G_s$ has unlikely endomorphisms. 

However, in order to control the radii within which global relations are required to hold, we will add an additional G-function to the set $\cG$.
To construct this G-function, we first apply \cite[Lemma 5.5]{Y(1)} to $C$ and $x$. We deduce that, for each place $v$ of $K$, there exists a real number $r_v>0$ and open sets $U_{v,1},\ldots,U_{v,n}$ of $C^\van$ such that
\begin{enumerate}[(i)]
    \item $r_v \geq 1$ for almost all $v$;
\item $s_k^\van \in U_{v,k}$;
\item for each $v$, the sets $U_{v,1}, \ldots, U_{v,n}$ are pairwise disjoint and $U_{v,1} \cup \cdots \cup U_{v,n}$ is equal to the preimage under $x^\van$ of $D(r_v,K_v)$;
\item for each $v$ and $k$, $x^\van$ restricts to a bijection from $U_{v,k}$ to the open disc $D(r_v,K_v)$. 
\end{enumerate}

Shrinking $r_v$ and $U_{v,k}$, we can and do assume that $r_v \leq r^{(k)}_v$ for all $v$ and~$k$.
It follows that, for all places~$v$ of~$K$ and all $k,k'$ with $1 \leq k,k' \leq n$, we have $U_{v,k} \subset \sigma_k(U_v^{(k')})$.
Indeed, since $x \circ \sigma_k = x$, $x^\van$ maps $\sigma_k^\van(U_v^{(k')})$ isomorphically onto $D(r_v^{(k')},K_v)$, which contains $D(r_v,K_v)$.
Consequently, $\sigma_k^\van(U_v^{(k')})$ contains one of the sets $U_{v,1}, \dotsc, U_{v,n}$.
Since $s_k^\van = \sigma_k^\van(s_1^\van) \in \sigma_k^\van(U_v^{(k')})$, we conclude that $U_{v,k} \subset \sigma_k^\van(U_v^{(k')})$.

Now choose $\zeta\in K^\times$ such that $|\zeta^\van|_{v}\leq r_v$ for all places $v$ of $K$ for which $r_v<1$ (which is the case for only finitely many $v$ by (i)). We let $H$ denote the G-function $\zeta/(\zeta-X)=\sum_{k=0}^\infty(X/\zeta)^k$ and we let $\cG'=\cG\cup\{H\}$. We define $R'_{v}:=\min(1,R(H^\van),R_v)$. In particular, we have $R'_v\leq r_v$.

\medskip

Now let $s\in C^*(\Qbar)$ be such that $G_s$ has unlikely endomorphisms, and let $\hat{K}=K(s)$. For each $k=1,\ldots,n$, we have $\sigma^{-1}_k(s)\in C^*(\hat{K})$ and $G^{(k)}_{\sigma^{-1}_k(s)}\cong G_s$ has unlikely endomorphisms.
Since $k \sim [k]$, $G^{(k)}_{\sigma^{-1}_k(s),\Qbar}$ is isogenous to $G^{[k]}_{\sigma^{-1}_k(s),\Qbar}$.
Hence the latter also has unlikely endomorphisms.
Therefore, applying \cref{exist-relations} to the semiabelian scheme $\fG^{[k]}$ and the point $\sigma^{-1}_k(s) \in C^*(\hat{K})$, we obtain a homogeneous polynomial $P_k\in\Qbar[\underline{Y},\underline{Z}]$ of degree at most $\newC{mult-rel-1k}[\hat{K}:K]$, not in the ideal~$I$.

Let $\hat{v}$ be a place of $\hat{K}$ and let $v:=\hat{v}|_K$.
Suppose that $\abs{x^\van(s^\vhatan)}_{\hat v}<R'_v\leq r_v$. We have $s^\vhatan \in U_{v,k}(\hat K)$ for some $k\in\{1,\ldots,n\}$. Hence $\sigma_k^{-1}(s)^\vhatan\in U^{[k]}_{v}(\hat K) \subset C^{\van}(\hat K)$. Therefore,
\begin{align*}
   P^\van_{k} \bigl( \underline{F}^{[k]\van}&(x^\van(\sigma^{-1}_k(s)^\vhatan)), \, \underline{G}^{[k]\van}(x^\van(\sigma^{-1}_k(s)^\vhatan)) \bigr)
\\ = {} & P^\van_{k} \bigl( \underline{F}^{[k]\van}(x^\van(s^\vhatan)), \, \underline{G}^{[k]\van}(x^\van(s^\vhatan)) \bigr) = 0, 
\end{align*} 
where $\underline{F}^{[k]}$ (resp. $\underline{G}^{[k]}$) denotes the tuple of G-functions $(F^{[k]}_{11},\ldots,F^{[k]}_{gg})$ (resp. $(G^{[k]}_{11},\ldots,G^{[k]}_{gg})$). 

We define a new polynomial $P\in\Qbar[\underline{Y}^{(k)},\underline{Z}^{(k)}:k\in\Lambda]$ by taking the product of the $P_k(\underline{Y}^{[k]},\underline{Z}^{[k]})$, for all $k=1,\ldots,n$. We conclude that $P$ is a global relation of degree at most $\newC{mult-rel-manyk}[\hat{K}:K]$ between the elements of $\cG'$ at $x(s)$.

\subsection{Proof of Theorem \ref{height-bound}: trivial relations} \label{subsec:trivial-relations}

This section follows \cite[sec.~5.I]{Y(1)}, which is based on arguments in \cite[sec.~7]{Papas:heights}.

We claim that the global relation $P$ is non-trivial. \cref{height-bound} then follows, by \cite[Theorem E]{And89}.

In order to verify the claim, we fix an archimedean place $v$ of $K$, and we consider the fibre product $G^{*\Lambda}$ over $C^*$ of the $G^{*(k)}$ for $k\in\Lambda$. We let $\gM$ denote the generic Mumford--Tate group of this abelian scheme and we let $\pi:G^{*\Lambda}\to C^*$ denote its structure morphism. For notational ease, we relabel $n=|\Lambda|$.
Since the image of the curve $C^*$ in $\Ag$ is Hodge generic, the generic Mumford--Tate group of each $G^{*(k)}$ is $\gGSp_{2g}$.
Hence, the derived subgroup $\gM^\der$ of $\gM$ is a subgroup of $\gSp_{2g}^n$ that surjects on to each factor. 

\begin{proposition}\label{prop:Sp2gn}
We have $\gM^\der=\gSp_{2g}^n$.
\end{proposition}

In order to prove \cref{prop:Sp2gn}, we need the following lemmas.

\begin{lemma}\label{auts-of-sp}
The automorphism group of $\gSp_{2g}$ as a $\QQ$-algebraic group is equal to $\gPGSp_{2g}(\QQ)$.
\end{lemma}

\begin{proof}
This is \cite[Theorem 2.8]{PR94}, using the fact that a Dynkin diagram of type~$C$ has no non-trivial symmetries.
\end{proof}

\begin{lemma}\label{aut-lift}
Every $\phi\in\Aut(\gPSp_{2g})$ lifts to an automorphism of $\gSp_{2g}$.  
\end{lemma}

\begin{proof}
Let $\phi\in\Aut(\gPSp_{2g})$ and let $\pi:\gSp_{2g}\to\gPSp_{2g}$ denote the natural morphism. Applying \cite[Proposition 20.6]{milneiAG} with $G=\gPSp_{2g}$, $\tilde{G}=\gSp_{2g}$, $G'=\tilde{G}$ and $\varphi=\phi\circ\pi$, we obtain a morphism $f:\gSp_{2g}\to\gSp_{2g}$ such that $\varphi\circ f=\pi$. Since $\gSp_{2g}$ is simply connected, it follows that $f$ is an isomorphism, and so $\tilde\phi=f^{-1}$ fulfils the requirements of the lemma.
\end{proof}

\begin{proof}[Proof of \cref{prop:Sp2gn}]
Consider the product decomposition $\gSp_{2g}^n=\gSp_{2g}^{n-1}\times\gSp_{2g}$. By induction, it suffices to prove the lemma under the assumption that the projection of $\gM^\der$ on to the first of these two factors is surjective. Under this assumption, Goursat's Lemma tell us that $\gM^\der$ is equal to the preimage in $\gSp_{2g}^{n-1}\times\gSp_{2g}$ of the graph of an isomorphism 
\[\gSp_{2g}^{n-1}/\gK_2\to\gSp_{2g}/\gK_1,\]
where $\gK_i$ is the kernel of the $i$-th projection of $\gSp_{2g}^{n-1}\times\gSp_{2g}$.

The normal subgroups of $\gSp_{2g}$ are $\{1\}$, $\{\pm 1\}$, and $\gSp_{2g}$ itself, and the normal subgroups of $\gSp_{2g}^{n-1}$ are products thereof. If $\gK_1=\{1\}$, then $\gK_1$ must be the product of $n-2$ factors $\gSp_{2g}$ of $\gSp_{2g}^{n-1}$ so that $\gSp_{2g} \to \gSp_{2g}^{n-1}/\gK_2$ is the projection onto one factor.
Hence, by \cref{auts-of-sp}, after possibly reordering the factors, $\gM^{\rm der}$ is the subgroup of $\gSp^n_{2g}$ consisting of elements $(g_1,\ldots,g_n)$ such that $g_n=hg_1h^{-1}$ for some $h\in\gPSp_{2g}(\QQ)$.
However, this implies that there is an isogeny $G^{(1)}_{\bar\eta} \to G^{(n)}_{\bar\eta}$, contradicting the definition of~$\Lambda$.

On the other hand, if $\gK_1 = \{\pm 1 \}$, then, after reordering the factors, $\gK_2 = \{\pm 1\} \times \gSp_{2g}^{n-2}$ and $\gSp_{2g}^{n-1} \to \gSp_{2g}^{n-1}/\gK_2$ is projection onto the first factor composed with $\gSp_{2g} \to \gPSp_{2g}$.
By \cref{aut-lift}, the isomorphism $\gSp_{2g}^{n-1}/\gK_2 \to \gSp_{2g}/\gK_1$ lifts to an isomorphism $\tilde\phi \colon \gSp_{2g}^{n-1}/\gSp_{2g}^{n-2} \to \gSp_{2g}$.
Then, the preimage in $\gSp_{2g}^{n-1} \times \gSp_{2g}$ of the graph of $\tilde\phi$ is an index-$2$ subgroup of $\gM^\der$.
This is impossible since $\gM^{\der}$ is connected.

Therefore, $\gK_1=\gSp_{2g}$ and $\gM^{\rm der}=\gSp^n_{2g}$, as claimed.
\end{proof}

Let $U^*_{v,1}:=U_{v,1}\cap C^{*\van}$ and, for each $k\in\Lambda$, let $\pi^{(k)}:G^{*(k)}\to C^*$ denote the structure morphism. Let $\gamma^{(k)}_1,\ldots,\gamma^{(k)}_g\in R_1\pi^{(k)\van}_*\ZZ|_{U^*_{v,1}}$ denote the basis of locally invariant sections as in \cref{subsec:arch-interp-integrals}. By \cref{lem:arch-periods}, for any $s\in U_{v,1}^{*}$,
\begin{align*}
    Y^{(k)}_{ij}(s) & :=\frac{1}{2\pi i} \int_{\gamma^{(k)}_{j,s}} \omega_{i,s}^{(k)\van} = F_{ij}^{(k)\van}(x^\van(s)),\\
Z^{(k)}_{ij}(s) & :=\frac{1}{2\pi i} \int_{\gamma^{(k)}_{j,s}} \eta_{i,s}^{(k)\van} = G_{ij}^{(k)\van}(x^\van(s)).
\end{align*} 

Choosing a simply connected open subset $V_{u,1}$ of $U^*_{v,1}$ and completing $\gamma^{(k)}_1,\ldots,\gamma^{(k)}_g$ to a symplectic basis of $R_1\pi^{(k)\van}_*\ZZ|_{V_{v,1}}$ (for the symplectic form induced by the polarisation), we consider these functions as the first $g$ columns of a period matrix 
\[\fullmatrix{Y^{(k)}}{Y'^{(k)}}{Z^{(k)}}{Z'^{(k)}}:V_{v,1}\to\gM_{2g}(\CC)\]
for $G^{*(k)}$ and we consider this matrix as the $k^{\rm th}$ block of a period matrix 
\[\cP:V_{v,1}\to\gM_{2g}(\CC)^n\]
for the fibre product $G^{*\Lambda}\to C^*$. 

By the Riemann relations, and by \cref{omega-eta-basis}(ii), we deduce that the graph of $\cP$ is contained in the subvariety $\Theta$
of $C_\CC^*\times\gM^n_{2g,\CC}$ defined by the following equations for all $k\in\Lambda$:
\begin{gather}
   (Y^{(k)})^tZ^{(k)}-(Z^{(k)})^tY^{(k)}=0; \label{eqn:YkZk}
\\ (Y^{'(k)})^tZ^{'(k)}-(Z^{'(k)})^tY^{'(k)}=0;
\\ (Y^{(k)})^tZ^{'(k)}-(Z^{(k)})^tY^{'(k)}=\frac{e_k}{2\pi i}I_g.
\end{gather}
Since $e_k \in \cO(C^*)^\times$ the variety $\Theta$ is a $\gSp_{2g,\CC}^n$-torsor over $C_\CC^*$.

\begin{lemma}\label{lem:P-dense}
The graph of $\cP$ is Zariski dense in $\Theta$.
\end{lemma}

\begin{proof}
Arguing exactly as in the proof of \cite[Lemma 5.9]{Y(1)} (based on André's Normal Monodromy Theorem \cite[Thm.~1]{And92}), and using \cref{prop:Sp2gn}, we deduce that, for any Hodge generic point $z\in V_{v,1}$, the monodromy representation 
\[\pi_1(C^{*\van},z)\to\gM^{\der}(\QQ)=\gSp_{2g}(\QQ)^n\] 
attached to $R_1\pi^\van_*\QQ$ has Zariski dense image. We now argue exactly as in the proof of \cite[Lemma 7.2]{Papas:heights}.
\end{proof}

Let $\pi:\gM_{2g}(\CC)^n\to\gM_{2g\times g}(\CC)^n$ denote the projection on to the first $g$ columns in each $2g\times 2g$ matrix, and let $\Theta'$ denote the subvariety of $C_\CC^*\times\gM_{2g\times g,\CC}^n$ defined by the equations~\eqref{eqn:YkZk}.

\begin{corollary}
The graph of $\pi\circ\cP$ is Zariski dense in $\Theta'$.   
\end{corollary}

\begin{proof}
Thanks to \cref{lem:dom}, the morphism $\pi \colon \Theta_s \to \Theta'_s$ is dominant for every $s \in C^*(\CC)$.
Hence, $\id \times \pi \colon \Theta \to \Theta'$ is dominant.
Therefore, the result follows from Lemma \ref{lem:P-dense}.
\end{proof}

Now, by an argument entirely analogous to \cite[Proposition 5.11]{Y(1)}, we conclude that any trivial relation between the elements of $\cG$ is contained in the ideal $I^\Lambda$ of $\Qbar[\underline{Y}^{(k)},\underline{Z}^{(k)}:k\in\Lambda]$ generated by the equations \eqref{eqn:YkZk}. 

It remains, therefore, to show that $P\notin I^\Lambda$. However, this follows from the fact that $I^\Lambda$ is prime (which follows from the fact that $I$ is prime, using \cite[Prop.~5.17]{milneAG}). Indeed, suppose that $P \in I^\Lambda$. Then, since $I^\Lambda$ is prime, $P_k(\underline{Y}^{[k]},\underline{Z}^{[k]})\in I^\Lambda$ for some $k\in\{1,\ldots,n\}$. Since this implies $P_k\in I$, we obtain a contradiction.

\bibliographystyle{amsalpha}
\bibliography{galois-bounds}

\end{document}